\RequirePackage{etoolbox}
\csdef{input@path}{{style/}{graphics/}}
\documentclass[seceqn,aap,number,MSNbibl,dvips]{arximspdf}
\usepackage{url,breakurl}
\usepackage{graphicx}

\doi{10.1214/14-AAP1010} 
\volume{25}
\issue{2}
\pubyear{2015}
\firstpage{753}
\lastpage{822}

\makeatletter
\newtheorem{cla}{Claim}
\newtheorem{hypothesis}{Hypothesis}

\newtheorem{lemma}{Lemma}
\newtheorem{theorem}{Theorem}
\newtheorem{proposition}{Proposition}
\newtheorem{corollary}{Corollary}

\newproclaim{definition}{Definition}
\newproclaim{note}{Note}
\newproclaim{remark}{Remark}
\newproclaim{phase}{Phase transitions in polytope geometry}
\newproclaim{state}{State evolution}
\newproclaim{universality}{Universality}

\newcommand{\rright}{\right}
\newcommand{\lleft}{\left}
\newcommand{\rrVert}{\Vert}
\newcommand{\rrvert}{\vert}
\newcommand{\llVert}{\Vert}
\newcommand{\llvert}{\vert}

\def\tbeta{\widetilde{\beta}}
\def\hc{\hat{c}}
\def\tS{\widetilde{S}}
\def\tv{\widetilde{v}}
\def\tZ{\widetilde{Z}}
\def\oom{\overline{\omega}}
\def\sT{\mathsf{T}}
\def\id{\mathrm{I}}
\def\tR{\widetilde{R}}
\def\tsigma{\widetilde{\sigma}}
\def\tq{\tilde{q}}

\def\cF{\mathcal{F}}
\def\tcF{\widetilde{\mathcal{F}}}
\def\hcF{\widehat{\mathcal{F}}}
\def\reals{\mathbb{R}}
\def\naturals{\mathbb{N}}

\def\E{\mathbb{E}}
\def\tA{\tilde{A}}
\def\normal{\mathsf{N}}
\def
\prob{\mathbb{P}}
\def\ex{\setminus}
\def\cT{\mathcal{T}}
\def\cTall{\overline{\mathcal{T}}}
\def\cU{\mathcal{U}}
\def\cUall{\overline{\mathcal{U}}}
\def\root{\textrm{root}}
\def\gA{\tilde{A}}
\def\vv{\mathbf{v}}
\def\vx{\mathbf{x}}
\def\vm{\mathbf{m}}
\def\vk{\mathbf{k}}
\def\tvx{\widetilde{\mathbf{x}}}
\def\tvz{\widetilde{\mathbf{z}}}
\def\tg{\widetilde{g}}

\def\v1{\mathbf{1}}
\def\vu{\mathbf{u}}
\def\vy{\mathbf{y}}
\def\vz{\mathbf{z}}
\def\ve{\mathbf{e}}
\def\vc{\mathbf{c}}
\def\vd{\mathbf{d}}
\def\cB{\mathcal{B}}
\def\vG{\mathbf{G}}
\def\ind{\mathbf{1}}
\def\root{\circ}

\def\F{\mathsf{F}}
\def\Ons{\mathsf{B}}
\def\POns{\mathsf{D}}
\def\div{\operatorname{div}}
\def\tx{\tilde{x}}
\def\bx{\bar{x}}
\def\bbx{\mathbf{\bar{x}}}

\def\tvx{\mathbf{\tilde{x}}}
\def\tvz{\mathbf{\tilde{z}}}

\def\cQ{\mathcal{Q}}
\def\cZ{\mathcal{Z}}
\def\cX{\mathcal{X}}
\def\face{\mathfrak{F}}
\def\de{\mathrm{d}}
\def\eps{{\varepsilon}}
\def\supp{\operatorname{supp}}
\def\hx{\widehat{x}}
\def\oS{\overline{S}}
\def\sign{\operatorname{sign}}
\def\cE{\mathcal{E}}
\def\Var{\operatorname{Var}}
\def\hSigma{\widehat{\Sigma}}
\def\tR{\widetilde{R}}

\newcommand{\eqref}[1]{(\ref{#1})}
\newcommand{\binom}[2]{{{#1}\choose{#2}}}
\renewcommand{\emptyset}{\varnothing}

\newcommand{\fracc}[2]{{#1}/(#2)}

\newcommand{\fraca}[2]{{#1}/{#2}}

\newcommand{\fracb}[2]{(#1)/{#2}}
\makeatother

\begin{document}
\begin{frontmatter}

\title{Universality in polytope phase transitions and
message passing algorithms}
\runtitle{Universality in polytope phase transitions and algorithms\hspace*{7pt}}

\begin{aug}
\author[A]{\fnms{Mohsen}~\snm{Bayati}}, 
\author[B]{\fnms{Marc}~\snm{Lelarge}\thanksref{T1}} 
\and
\author[C]{\fnms{Andrea}~\snm{Montanari}\corref{}\thanksref{T2}\ead[label=e3]{montanari@stanford.edu}\corref{}}
\runauthor{M. Bayati, M. Lelarge and A. Montanari}
\affiliation{Stanford University, INRIA and ENS, and Stanford University}
\address[A]{M. Bayati\\
Graduate School of Business\\
Stanford University\\
Knight Management Center\\
655 Knight Way, E363\\
Stanford, California 94305\\
USA} 
\address[B]{M. Lelarge\\
INRIA\\
and\\
ENS\\
23, avenue d'Italie\\
75214, Paris\\
France}
\address[C]{A. Montanari\\
Department of Electrical Engineering\\
\quad and Department of Statistics\\
Stanford University\\
350 Serra Mall\\
Stanford, California 94305-9505\\
USA\\
\printead{e3}}
\end{aug}
\thankstext{T1}{Supported by the French Agence Nationale de la
Recherche (ANR) under reference ANR-11-JS02-005-01 (GAP project).}
\thankstext{T2}{Supported in part by  NSF CAREER award
CCF-0743978,  NSF Grant DMS-08-06211, and  AFOSR Grant
FA9550-10-1-0360.}

\received{\smonth{7} \syear{2012}}
\revised{\smonth{12} \syear{2013}}

%
\begin{abstract}
We consider a class of nonlinear mappings $\mathsf{F}_{A,N}$ in $\mathbb{R}^N$
indexed by
symmetric random matrices $A\in\mathbb{R}^{N\times N}$ with independent
entries. Within spin glass theory, special cases of these mappings
correspond to
iterating the TAP equations and were studied by
Bolthausen [\textit{Comm. Math. Phys.} \textbf{325} (2014) 333--366].
Within information theory, they are known as
``approximate message passing'' algorithms.

We study the high-dimensional (large $N$) behavior of the iterates of
$\mathsf{F}$ for polynomial functions $\mathsf{F}$, and prove that it is universal;
that is, it depends only on the
first two moments of the entries of $A$, under a sub-Gaussian tail condition.
As an application, we prove the universality of a certain phase
transition arising in
polytope geometry and compressed sensing. This solves, for a broad
class of random projections, a conjecture by David Donoho and Jared
Tanner.
\end{abstract}

%

\end{frontmatter}

\section{Introduction and main results}

Let $A\in\reals^{N\times N}$ be a random Wigner matrix, that is, a
symmetric random
matrix with i.i.d. entries $A_{ij}$ satisfying $\E\{A_{ij}\}=0$ and
$\E\{A_{ij}^2\}=1/N$. Considerable effort has been devoted to
studying the distribution of the eigenvalues of such a matrix \cite
{Guionnet,BaiSilverstein,TaoVuReview}.
The \emph{universality phenomenon} is a striking recurring theme in
these studies. Roughly speaking, many asymptotic properties of the
joint eigenvalues' distribution are independent of the entries'
distribution, as long as the latter has the prescribed first two
moments and
satisfies certain tail conditions. We refer to \cite
{Guionnet,BaiSilverstein,TaoVuReview} and references
therein for a selection of such results. Universality is extremely
useful because it allows us to compute asymptotics for one
type of distribution
of the entries (typically, for Gaussian entries) and then export the results
to a broad class of distributions.

In this paper we are concerned with random matrix universality,
albeit we do not focus on eigenvalues properties. Given
$A\in\reals^{N\times N}$, and an initial condition $x^0\in\reals^N$
independent of $A$, we consider the sequence $(x^t)_{t\ge0}$ $t\in
\naturals$ defined by letting, for $t\ge0$,
%
\begin{equation}\label{eq:SimpleIteration}
\hspace*{15pt}x^{t+1} = A f\bigl(x^t;t\bigr) -
\mathsf{b}_t f\bigl(x^{t-1};t-1\bigr),\qquad
\mathsf{b}_t \equiv\frac{1}{N}\div\bigl(f(x;t)\bigr)
\bigg|_{x=x^t},  %
\end{equation}
where, by convention, $\mathsf{b}_0=0$. Here for each $t\ge0$, $f( \cdot;t)\dvtx \reals^N\to\reals^N$ is a separable function, that is, $f(z;t) =
(f_1(z_1;t),\dots,f_N(z_N;t))$.
We also assume that the functions $f_i( \cdot;t)\dvtx \reals\to\reals$ are
polynomials of bounded degree.
In addition, \textit{div} denotes the divergence operator, and in
particular, $\mathsf{b}_t = N^{-1}\sum_{i=1}^Nf'_i(x_i^t;t)$.

The present paper is concerned with the asymptotic distribution
of $x^t$ as $N\to\infty$ with $t$ fixed, and establishes the
following results:

\begin{universality*} As $N\to\infty$, the finite-dimensional marginals
of the
distribution of $x^t$ are asymptotically insensitive to the
distribution of the entries of $A_{ij}$.
\end{universality*}

\begin{state*} The entries of $x^t$ are
asymptotically Gaussian with zero mean, and variance that can be
explicitly computed through a one-dimensional recursion that we
will refer to as \emph{state evolution}.
\end{state*}

\begin{phase*}
As an application, we use state evolution to prove
universality of a phase transition on polytope geometry, with
connections to compressed sensing. This solves, for a broad class of
random matrices with independent entries, a conjecture put
forward by Donoho and Tanner
\cite{Donoho2005b,DonohoTannerUniversality}.
\end{phase*}

In order to illustrate the usefulness of the first two technical
results, we will start the presentation of our results from the third one.

Before stating our results, it is useful to comment on the
special form of the iteration (\ref{eq:SimpleIteration}), and in
particular on the role of the memory term $\mathsf{b}_t f(x^{t-1};t-1)$
(which is inspired from the so-called ``Onsager correction'' in
statistical physics \cite
{thouless1977solution,mezard1987spin,bolthausen2012iterative}).
The function of this term is to cancel, to leading order, the
effect of
correlations
between $x_i^{t+1}$ and $\{x^s_{i}\dvtx  s\le t\}$. This cancelation is
particularly transparent in our proof technique, whereby $x^{t}_i$ is
expressed as a sum of monomials in $A_{jk}$, with $1\le j,k\le n$,
and is indexed by labeled trees. The memory term effectively cancels
the contribution of
``one-step reversing'' trees.

Without such memory term, the properties of the resulting
iteration change crucially. In particular, it is no longer true that
$x^t_i$ is approximately Gaussian as $N\to\infty$; see Section~\ref
{sec:Sketch} for further clarification on this point.

\subsection{Universality of polytope neighborliness}

A polytope $Q$ is said to be \emph{centrosymmetric} if $x\in Q$
implies $-x\in Q$.
Following \cite{Donoho2005b,Donoho2005a} we say that such a polytope is
\emph{$k$-neighborly} if the condition below holds:
\begin{longlist}[(I)]
\item[(I)] Every subset of $k$ vertices of $Q$
which does not contain an antipodal pair, spans a
$(k-1)$-dimensional face.
\end{longlist}
The \emph{neighborliness} of $Q$ is the largest value of $k$ for which
this condition holds. The prototype of neighborly polytope is the
$\ell_1$ ball $C^n \equiv\{x\in\reals^n \dvtx \llVert x\rrVert _1\le1\}$, whose
neighborliness is indeed equal to $n$.

It was shown in a series of papers \cite
{Donoho2005a,Donoho2005b,DoTa05a,DoTa05b,DoTa08}
that polytope neighborliness has tight connections with the geometric properties
of random point clouds, and with sparsity-seeking methods to solve
underdetermined systems of linear equations. The latter are in turn
central in a number of applied domains, including model selection for
data analysis and compressed sensing. For the reader's convenience,
these connections will be
briefly reviewed in Section~\ref{sec:Polytope}.

Intuitive images of low-dimensional polytopes suggest that ``typical''
polytopes are not neighborly: already selecting $k=2$ vertices does
lead to a segment that connects them and passes through the interior
of $Q$. This conclusion is spectacularly wrong
in high dimension. Natural random constructions lead to
polytopes whose neighborliness scales \emph{linearly} in the dimension.
Motivated by the above applications, and following
\cite{Donoho2005a,Donoho2005b,DoTa05a,DoTa05b},
we focus here on a weaker notion
of neighborliness. Roughly speaking, this corresponds to the largest
$k$ such that \emph{most} subsets of $k$ vertices of $Q$ span a
$(k-1)$-dimensional face. In order to formalize this notion, we denote
by $\face(Q;\ell)$ the number of $\lfloor\ell\rfloor$-dimensional faces
of $Q$.

\begin{definition}
Let $\cQ= \{Q^n\}_{n\ge0}$ be a sequence of centrosymmetric
polytopes indexed by $n$ where $Q_n$ has $2n$ vertices and has
dimension $m=m(n)$:
$Q^n\subseteq\reals^m$. We say that $\cQ$ has \emph{weak neighborliness}
$\rho\in(0,1)$ if for any $\xi>0$,
\begin{eqnarray*}
\lim_{n\to\infty}\frac{\face
(Q^n; m(n)\rho(1-\xi))}{\face(C^{n}; m(n)\rho(1-\xi))} &=& 1,
\\
\lim_{n\to\infty}\frac{\face
(Q^n; m(n)\rho(1+\xi))}{\face(C^{n}; m(n)\rho(1+\xi))} &=& 0 .
\end{eqnarray*}
If the sequence $\cQ$ is random, we say that $\cQ$ has \emph{weak
neighborliness}
$\rho$ \emph{(in probability)} if the above limits hold \emph{in
probability}.
\end{definition}

In other words, a sequence of polytopes $\{Q^n\}_{n\ge0}$ has weak
neighborliness $\rho$,
if for large $n$ the $m$-dimensional polytope $Q^n$ has close to the
maximum possible number of $k$
faces, for all $k<m\rho(1-\xi)$.

\begin{note}\label{note:neighborly}
Note that previously the neighborliness of a polytope was defined to be
the largest integer $k$ satisfying
condition (I). However, in our definition, weak neighborliness refers
to the fraction $k/n$. This is due to the fact that
weak neighborliness is defined in the limit $n\to\infty$.
\end{note}

The existence of weakly neighborly polytope sequences is clear when
$m(n)=n$ since in this case we can take $Q^n=C^n$ with $\rho=1$, but
the existence is highly
nontrivial when $m$ is only a fraction of $n$.

It comes indeed as a surprise that this is a generic situation as
demonstrated by the following construction. For a matrix
$A\in\reals^{m\times n}$ and $S\subseteq\reals^n$, let
$AS\equiv\{Ax\in\reals^m \dvtx  x\in S\}$. In particular, $AC^n$ is
the centrosymmetric $m$-dimensional polytope obtained by projecting
the $n$-dimensional $\ell_1$ ball to $m$ dimensions. The following
result was proved in \cite{Donoho2005b}.

\begin{theorem}[(Donoho \cite{Donoho2005b})]\label{thm:Donoho2005}
There exists a function $\rho_*\dvtx (0,1)\to(0,1)$ such that the
following holds.
Fix $\delta\in(0,1)$. For each $n\in\naturals$, let $m(n) =
\lfloor n\delta\rfloor$ and define $A(n)\in\reals^{m(n)\times n}$ to
be a random matrix with i.i.d. Gaussian entries.

Then, the sequence of polytopes $\{A(n)C^n\}_{n\ge0}$ has weak
neighborliness $\rho_*(\delta)$ in probability.
\end{theorem}

A characterization of the curve $\delta\mapsto\rho_*(\delta)$ was provided
in \cite{Donoho2005b}, but we omit it here since a more explicit expression
will be given below.

The proof of Theorem~\ref{thm:Donoho2005} is based on exact
expressions for the number of faces $\face(A(n)C^n;\ell)$. These are in
turn derived from earlier works in polytope geometry by Affentranger
and Schneider
\cite{Affentranger} and by Vershik and Sporyshev
\cite{VershikSporyshev}. This approach
relies in a fundamental way on the invariance of the distribution of
$A(n)$ under rotations.

Motivated by applications to data analysis and signal processing,
Donoho and Tanner \cite{DonohoTannerUniversality} carried out
extensive numerical simulations for random polytopes of the form $A(n)C^n$
for several choices of the distribution of $A(n)$. They formulated a
\emph{universality hypothesis} according to which the
conclusion of Theorem~\ref{thm:Donoho2005} holds for a far broader
class of random matrices. The results of their numerical simulations were
consistent with this hypothesis.

\setcounter{footnote}{2}

Here we establish the first rigorous result indicating universality
of polytope neighborliness for a broad class of random
matrices. Define the curve $(\delta,\rho_*(\delta))$, $\delta\in
(0,1)$, parametrically
by letting, for $\alpha\in(0,\infty)$,
%
\begin{eqnarray}
\label{eq:PhaseBoundaryApp1}
\delta& = &\frac{2\phi(\alpha)}{\alpha+2(\phi(\alpha)-\alpha
\Phi
(-\alpha))},
\\
\label{eq:PhaseBoundaryApp2}
\rho& = &1-\frac{\alpha\Phi(-\alpha)}{\phi(\alpha)}, %
\end{eqnarray}
where $\phi(z) = e^{-z^2/2}/\sqrt{2\pi}$ is the Gaussian density and
$\Phi(x)\equiv\int_{-\infty}^x\phi(z)\, \de z$ is the Gaussian
distribution.
Explicitly, if the above functions on the right-hand side of
equations (\ref{eq:PhaseBoundaryApp1}), (\ref{eq:PhaseBoundaryApp2})
are denoted by $f_{\delta}(\alpha)$, $f_{\rho}(\alpha)$,
then\footnote{It is easy to
show that $f_{\delta}(\alpha)$ is strictly decreasing in
$\alpha\in[0,\infty)$, with $f_{\delta}(0)=1$,
$\lim_{\alpha\to\infty}f_{\delta}(\alpha) = 0$, and
hence $f^{-1}_{\delta}$ is well defined on $[0,1]$. Further properties
of this curve can be found in \cite{DMM09,NSPT}.}
$\rho_*(\delta) \equiv f_{\rho}(f_{\delta}^{-1}(\delta))$.

Here we extend the scope of Theorem~\ref{thm:Donoho2005} from Gaussian
matrices to matrices with independent sub-Gaussian\footnote{See equation
\eqref{eq:def-subgaussian}
for the definition of sub-Gaussian random variables.} entries (not
necessarily identically distributed).

\begin{theorem}\label{thm:Polytope}
Fix $\delta\in(0,1)$. For each $n\in\naturals$, let $m(n) =
\lfloor n\delta\rfloor$ and define $A(n)\in\reals^{m(n)\times n}$ to
be a random matrix with independent sub-Gaussian entries,
with zero mean, unit variance and common scale factor $s$ independent
of $n$.
Further assume $A_{ij}(n) = \tA_{ij}(n)+\nu_0 G_{ij}(n)$ where $\nu_0>0$
is independent of $n$ and $\{G_{ij}(n)\}_{i\in[m],j\in[n]}$ is a
collection of
i.i.d. $\normal(0,1)$ random variables independent of $\tA(n)$.

Then the sequence of polytopes $\{A(n)C^n\}_{n\ge0}$ has weak
neighborliness $\rho_*(\delta)$ in probability.
\end{theorem}

It is likely that this theorem can be improved in two
directions. First, a milder tail condition than sub-Gaussianity is
probably sufficient. Second, we are assuming that the distribution of
$A_{ij}$ has an arbitrarily small Gaussian component. This is not
necessary for the upper bound on neighborliness, and appears to be an
artifact of the proof of the lower bound.

The proof of Theorem \ref{thm:Polytope} is provided in Section~\ref{sec:Polytope}.
By comparison, the most closely related result toward universality
is by Adamczak, Litvak, Pajor, and Tomczak-Jaegermann \cite{Adamczak}.
For a class of matrices $A(n)$ with i.i.d. columns, these authors prove that
$A(n)C^n$ has neighborliness scaling linearly with $n$. This, however,
does not suggest that a limit weak neighborliness exists, and is
universal, as established instead in Theorem~\ref{thm:Polytope}.

At the other extreme, universality of compressed sensing phase
transitions can be conjectured from the results of the nonrigorous
replica method \cite{KabashimaTanaka,RanganFletcherGoyal}.

\subsection{Universality of iterative algorithms}\label{sec:UniversalityResults}

We will consider here and below a setting that is
somewhat more general than the one described by
equation (\ref{eq:SimpleIteration}).
Following the terminology of \cite{DMM09}, we will refer to such an
iteration as to the approximate message passing (AMP)
iteration/algorithm.

We generalize iteration (\ref{eq:SimpleIteration}) to take place in the
vector space
$\mathcal{V}_{q,N}\equiv(\reals^q)^N\simeq\reals^{N\times q}$.
Given a
vector $x\in\mathcal{V}_{q,N}$, we shall most often regard it as an
$N$-vector with entries in $\reals^q$, namely $x =
(\vx_1,\dots,\vx_N)$, with $\vx_i\in\reals^q$.
Components of
$\vx_i\in\reals^q$ will be indicated as $(x_i(1),\dots,x_i(q))\equiv\vx_i$.

There are several motivations for considering such a generalization. On
one hand, it is necessary for the application to high-dimensional
polytope geometry presented in the previous section. The reader might
have noticed that the random matrix in Theorem~\ref{thm:Polytope}
is rectangular. This is a different setting from that of iteration
(\ref{eq:SimpleIteration}), whereby the random matrix $A$ is square and
symmetric. The generalization to $x\in\mathcal{V}_{q,N}$ introduced
here, with
$A$ square and symmetric, covers the case of rectangular matrices as
well through a suitable reduction. In a nutshell, given a rectangular
random matrix $A'$, the reduction
consists of constructing a symmetric matrix that has $A'$ as
submatrix; cf. Section~\ref{sec:NonSymmetric} for details.

Additional motivations for the generalization introduced here come
from the application of AMP algorithms to a variety of problems in
signal processing. For instance the authors of
\cite{krzakala2012statistical,donoho2013information} study compressed
sensing reconstruction for ``spatially coupled'' sensing matrices. These
are random matrices with independent but not identically distributed
entries. As already discussed in \cite
{donoho2013information,javanmard2013state} for the
case of Gaussian entries, a rigorous analysis of this algorithm
requires generalizing the setting of
(\ref{eq:SimpleIteration}).

Several other applications require a generalization of iteration
(\ref{eq:SimpleIteration}), including the analysis of generalized AMP
algorithms
\cite{RanganGAMP}, AMP reconstruction of block-sparse signals
\cite{donoho2013accurate}, the analysis of phase retrieval algorithms
\cite{schniter2012compressive}
and so on. All of these applications can be treated within the setting
introduced here, although our rigorous analysis requires the use of
polynomial nonlinearities.

A brief sketch of some proof ideas for the ``scalar'' case of equation
(\ref{eq:SimpleIteration})
can be found in Section~\ref{sec:Sketch}.

Given a matrix $A\in\reals^{N\times N}$, we let it act on
$\mathcal{V}_{q,N}$ in the natural way, namely for $v', v\in\mathcal
{V}_{q,N}$ letting
$v'=Av$ be given by
$\vv'_i = \sum_{j=1}^NA_{ij}\vv_j$ for all $i\in[N]$.
Here and below $[N]\equiv\{1,\dots,N\}$ is the set of first $N$ integers.
In other words we identify $A$ with the
Kronecker product $A\otimes\id_{q\times q}$.

\begin{definition}\label{def:GeneralDef}
An \emph{AMP instance} is a triple $(A,\cF,x^0)$ where:
\begin{longlist}[(3)]
\item[(1)] $A\in\reals^{N\times N}$ is a symmetric matrix with $A_{i,i}=0$
for all $i\in[N]$.
\item[(2)] $\cF= \{f^k\dvtx  k\in[N]\}$ is a collection of mappings
$f^k\dvtx \reals^q\times\naturals\to\reals^q$, $(\vx,t)\mapsto f^k(\vx,t)$
that are locally Lipschitz in
their first argument.
\item[(3)] $x^0\in\mathcal{V}_{q,N}$ is an initial condition.
\end{longlist}
Given $\cF= \{f^k\dvtx  k\in[N]\}$, we define $f( \cdot;t)\dvtx \mathcal
{V}_{q,N}\to
\mathcal{V}_{q,N}$ that maps $v$ to $v'=f(v;t)$, and is
given by $\vv'_i = f^i(\vv_i;t)$ for all $i\in[N]$.
\end{definition}

\begin{definition}
The \emph{approximate message
passing orbit} corresponding to the instance $(A,\cF,x^0)$ is the
sequence of vectors $\{x^t\}_{t\ge0}$, $x^t\in\mathcal{V}_{q,N}$
defined as follows, for $t\ge0$,
%
\begin{equation}\label{eq:AMPGeneralDef}
x^{t+1} = A f\bigl(x^t;t\bigr) -
\Ons_t f\bigl(x^{t-1};t-1\bigr).  %
\end{equation}
Here $\Ons_t\dvtx  \mathcal{V}_{q,N}\to\mathcal{V}_{q,N}$ is the linear
operator that maps $v$ to
$v' = \Ons_t v$, and is defined by
%
\begin{equation}
\vv'_i = \biggl(\sum_{j\in[N]}A_{ij}^2
\frac{\partial f^j}{\partial
\vx
}\bigl(\vx^t_j,t\bigr) \biggr)
\vv_i,
\end{equation}
with $\frac{\partial f^j}{\partial\vx}$ denoting the Jacobian matrix
of $f^j( \cdot;t)\dvtx \reals^q\to\reals^q$.
\end{definition}

The above definition can also be summarized by the following
expression for the evolution of a single coordinate under AMP:
%
\begin{equation}\label{eq:AMPGeneralDef_bis}
\vx^{t+1}_i = \sum
_{j\in[N]}A_{ij}f^{j}\bigl(
\vx^t_j,t\bigr)- \sum_{j\in[N]}A_{ij}^2
\frac{\partial
f^j}{\partial\vx}\bigl(\vx^{t}_j,t\bigr) f^{i}
\bigl(\vx^{t-1}_i,t-1\bigr).%
\end{equation}
%
Notice that equation (\ref{eq:SimpleIteration}) corresponds to the special
case $q=1$, in which we replaced $A_{ij}^2$ by $\E\{A_{ij}^2\} = 1/N$
for simplicity of exposition. The term $\Ons_t f(x^{t-1};t-1)$ in
equation (\ref{eq:AMPGeneralDef})
is the correct generalization of the term $\mathsf{b}_t f(x^{t-1};t-1)$
introduced in the $q=1$ case;
cf. equation (\ref{eq:SimpleIteration}). Namely it cancels, to leading
order, the correlations between $\vx^{t+1}_i$ and $\{\vx^{s}_i, s\le
t\}$.

Recall that a centered random variable $X$ is \emph{sub-Gaussian} with scale
factor $\sigma^2$ if, for all $\lambda>0$, we have
%
\begin{equation}\label{eq:def-subgaussian}
\E \bigl( e^{\lambda X} \bigr)\leq e^{{\sigma^2\lambda
^2}/{2}}.
\end{equation}

\begin{definition}
Let $\{(A(N),\cF_N,x^{0,N})\}_{N\ge1}$ be a sequence of AMP instances
indexed by
the dimension $N$, with $A(N)$ a random matrix and $x^{0,N}$ a random
vector. We say that the sequence is $(C,d)$-\emph{regular} (or, for short,
\emph{regular}) \emph{polynomial} sequence if:
\begin{longlist}[(3)]
\item[(1)] For each $N$, the entries $(A_{ij}(N))_{1\le i<j\le N}$ are
independent centered random variables. Further they are sub-Gaussian
with common scale factor $C/N$ [explicitly, there exists an
$N$-independent $C>0$ such
that $\log\E(e^{\lambda A_{ij}})\le(C\lambda)^2/(2N^2)$, cf.
equation (\ref{eq:def-subgaussian})].
\item[(2)] For each $N$, the functions $f^i( \cdot;t)$ in $\cF_N$
[possibly random, as long as they are independent from $A(N)$,
$x^{0,N}$] are polynomials with maximum
degree $d$ and coefficients bounded by $C$.
\item[(3)] For each $N$, $A(N)$ and $x^{0,N}$ are independent. Further,
we have\break $\sum_{i=1}^N\exp\{\llVert \vx_i^{0,N}\rrVert _2^2/C\}\le NC$
with probability converging to one as $N\to\infty$.
\end{longlist}
\end{definition}

We state now our universality result for the algorithm (\ref
{eq:AMPGeneralDef}).

\begin{theorem}\label{thm:Universality}
Let $(A(N),\cF_N,x^{0,N})_{N\ge1}$ and $(\gA(N),\cF
_N,x^{0,N})_{N\ge1}$
be any two $(C,d)$-regular polynomial sequences of instances, that
differ only in the
distribution of the random matrices $A(N)$ and $\gA(N)$.

Denote by $\{x^t\}_{t\ge0}$, $\{\tx^t\}_{t\ge0}$ the corresponding AMP
orbits.
Assume further that for all $N$ and all $i<j$, $\E\{A_{ij}^2\} = \E\{
\gA
_{ij}^2\}$.
Then, for any set of\vspace*{1pt} polynomials $\{p_{N,i}\}_{N\ge0, 1\le i\le N}$
$p_{N,i}\dvtx \reals^q\to\reals$, with
degree bounded by $d$ and coefficients bounded by a constant $B$ for
all $N$ and
$i\in[N]$, we have
%
\begin{equation}
\lim_{N\to\infty} \frac{1}{N}\sum
_{i=1}^N \bigl\{ \E p_{N,i}\bigl(\vx
^t_i\bigr) - \E p_{N,i}\bigl(
\tvx^t_i\bigr) \bigr\} = 0 . %
\end{equation}
\end{theorem}

%
\subsection{State evolution}\label{sec:StateEvolutionResults}

Theorem~\ref{thm:Universality} establishes that the behavior of the
sequence $\{x^t\}_{t\ge0}$ is, in the high-dimensional limit, insensitive to the distribution of the entries of
the random matrix $A$. In order to characterize this limit, we need to
make some assumption on the collection of functions $\cF_N$. In
particular, we need to relate the functions $\cF_N$ to the functions
$\cF_{N'}$ in order to have a high-dimensional ($N\to\infty$) limit.

Informally, we define a \emph{converging sequence} by requiring that
for each $N$, there exists a partition $[N] = C^N_1\cup
C^N_2\cup\cdots\cup C^N_k$ (with $k$ a fixed integer independent of
$N$), and independent random variables $Y(i)$
taking values in $\reals^{\tq}$, indexed by
$i\in[N]$, such that:
\begin{itemize}
\item the function $f^i$ only depends on the partition index of $i\in
[N]$, and on the value of $Y(i)$;
\item the distribution of $Y(i)$ only depends on the partition index
of $i\in[N]$;
\item the fractional size $\llvert C^N_a\rrvert /N$ is $N$-independent for large $N$.
\end{itemize}
There are a few points to make precise, and this is done in the
definition below.

\begin{definition}\label{def:Converging}
We say that the sequence of AMP instances $\{(A(N), \cF_N,\break x^{0,N})\}
_{N\ge0}$
is \emph{polynomial and converging} (or simply \emph{converging}) if
it is
$(C,d)$-regular and there exists:
(i) an integer $k$;
(ii) a symmetric matrix $W\in\reals^{k\times k}$ with nonnegative entries;
(iii) a function $g\dvtx \reals^q\times\reals^{\tq}\times[k]\times
\naturals\to\reals^q$, with $g(\vx,Y,a,t) =
(g_1(\vx,Y,a,t),\dots,g_q(\vx,Y,a,t))$ and, for each
$r\in[q]$, $a\in[k]$, $t\in\naturals$, $g_r( \cdot,Y,a,t)$ a
polynomial with degree $d$ and coefficients bounded by $C$;
(iv) $k$ probability measures $P_1$, \dots, $P_k$ on $\reals^{\tq}$,
with $P_a$ a finite mixture of (possibly degenerate) Gaussians for each
$a\in[k]$;
(v) for each $N$, a finite partition $C^N_1\cup C^N_2\cup\cdots\cup
C^N_k=[N]$; (vi) $k$ positive semidefinite matrices
$\hSigma^0_1,\dots,\hSigma^0_k\in\reals^{q\times q}$, such that the
following happens:
\begin{longlist}
\item[(1)] for each $a\in[k]$, we have $\lim_{N\to\infty}
\llvert C^N_a\rrvert /N =
c_a\in(0,1)$;
\item[(2)] for each $N\ge0$, each $a\in[k]$ and each $i\in C_a^N$, we
have $f^i(\vx,t) = g(\vx,Y(i),a,t)$ where $Y(1),\dots,Y(N)$ are
independent random variables with $Y(i)\sim P_a$ whenever
$i\in C^N_a$ for some $a\in[k]$;
\item[(3)] for each $N$, the entries $\{A_{ij}(N)\}_{1\leq i<j\leq N}$ are
independent sub-\break Gaussian random variables with scale factor $C/N$,
$\E A_{ij}=0$, and, for $i\in C^N_a$ and $j\in C^N_b$, $\E
\{A_{ij}^2\} = W_{ab}/N$;\vspace*{1pt}
\item[(4)] for each $a\in[k]$, in probability,
%
\begin{equation}\label{eq:InitialSE}
\lim_{N\to\infty} \frac{1}{\llvert C_a^N\rrvert }\sum
_{i\in C_a^N}g \bigl(\vx ^0_i,Y(i),a,0
\bigr) g \bigl(\vx^0_i,Y(i),a,0 \bigr)^{\sT} =
\hSigma^0_a.%
\end{equation}
\end{longlist}
\end{definition}

With a slight abuse of notation, we will sometimes denote a converging
sequence by $\{(A(N),g,x^{0,N})\}_{N\ge0}$. We use capital letters to
denote the $Y(i)$'s to emphasize that they are random and do not
change across iterations.

Our next result establishes that the low-dimensional marginals of $\{
x^t\}$ are
asymptotically Gaussian.
\emph{State evolution} characterizes the covariance of these
marginals. For each $t\ge1$, state evolution defines a set of $k$
positive semidefinite matrices $\Sigma^t =
(\Sigma^t_1,\Sigma^t_2,\dots,\Sigma^t_k)$, with
$\Sigma^t_a\in\reals^{q\times q}$. These are obtained by letting, for
each $t\ge1$,
%
\begin{eqnarray}\label{eq:GeneralSE} %
\Sigma^{t}_a &=& \sum
_{b=1}^k c_b
W_{ab} \hSigma^{t-1}_b,
\\
\hSigma^t_a &=& \E \bigl\{ g\bigl(Z^t_a,Y_a,a,t
\bigr) g\bigl(Z^t_a,Y_a,a,t
\bigr)^{\sT} \bigr\}, %
\end{eqnarray}
for all $a\in[k]$.
Here $Y_a\sim P_a$, $Z^t_a\sim\normal (0, \Sigma^t_a  )$ and
$Y_a$ and $Z^t_a$ are independent.

\begin{theorem}\label{thm:SE}
Let $(A(N),\cF_N,x^0)_{N\ge0}$ be a polynomial and converging sequence
of AMP
instances, and denote by $\{x^t\}_{t\ge0}$ the corresponding AMP
sequence. Then for each $t\ge1$, each $a\in[k]$ and each locally
Lipschitz function
$\psi\dvtx \reals^q\times\reals^{\tq}\to\reals$ such that $\llvert \psi(\vx,y)\rrvert \le
K(1+\llVert y\rrVert _2^2+\llVert \vx\rrVert _2^2)^K$, we
have, in probability,
%
\begin{equation}
\lim_{N\to\infty} \frac{1}{\llvert C_a^N\rrvert }\sum
_{j\in C^N_a} \psi\bigl(\vx ^t_j,Y(i)\bigr)
= \E\bigl\{\psi(Z_a,Y_a)\bigr\}, %
\end{equation}
where $Z_a\sim\normal(0,\Sigma_a^t)$ is independent of $Y_a\sim P_a$.
\end{theorem}

We conclude by mentioning that, following \cite{DMM09},
generalizations of algorithm (\ref{eq:AMPGeneralDef}) were
studied by several groups
\cite{SchniterTurbo,RanganGAMP,MalekiComplex}, for a number of
applications. Universality results analogous to the one proved here
are expected to hold for such generalizations as well.

\subsection{Outline of the paper}

The paper is organized as follows.
Before delving into the details of the analysis, Section~\ref
{sec:Sketch} provides
an informal discussion of the main proof ideas for the case $q=1$. After
some preliminary facts and notations in Section~\ref{sec:Notations},
Section~\ref{sec:Polynomial}
considers the AMP iteration (\ref{eq:AMPGeneralDef}) and proves
Theorems~\ref{thm:Universality} and~\ref{thm:SE}. In order to achieve
our goal, we introduce two
different iterations whose analysis provides useful intermediate
steps. We also prove a generalization of Theorem~\ref{thm:SE} to
estimate functions of messages at two distinct times
$\psi(\vx_i^t,\vx_i^{s},Y(i))$.

Section~\ref{sec:NonSymmetric} proves a generalization of Theorem~\ref{thm:SE} to the case of rectangular (nonsymmetric) matrices $A$.
This is achieved by effectively embedding the rectangular
matrix, into a larger symmetric matrix and applying our results for
symmetric matrices.

The generalization to rectangular matrices is finally used in Section~\ref{sec:Polytope} to prove our result on the universality of polytope
neighborliness, Theorem~\ref{thm:Polytope}. This is done via a
correspondence with compressed sensing reconstruction established in
\cite{Donoho2005b}, and a sharp analysis of an AMP iteration that
solves this reconstruction problem.
%
%
\section{Universality of iterative algorithms: Sketch of main ideas}\label{sec:Sketch}

In this section we sketch some key ideas in the proof of Theorems
\ref{thm:Universality} and \ref{thm:SE}. For the sake of clarity,
we shall focus on the special scalar recursion
(\ref{eq:SimpleIteration}) with $f(x;t) = f(x)$ kept constant across
iterations and $(A_{ij})_{i<j}$ independent centered sub-Gaussian, with
$\E\{A_{ij}^2\} = 1/N$. As in the statement of Theorems \ref
{thm:Universality} and
\ref{thm:SE}, we further assume that $f( \cdot)$ is separable and
polynomial.
Finally, we shall only consider the initial condition $x^0=\v1$ (the
all-ones vector).
While this setting is significantly more restrictive than the one of
Theorems \ref{thm:Universality} and \ref{thm:SE}, it is sufficient to
elucidate all the main ideas. For a complete
treatment of the general case, we refer the reader to Section~\ref
{sec:Polynomial}.

In order to clarify the role of the memory term in equation (\ref
{eq:SimpleIteration}), it is instructive to
first consider the case $k=1$, $f(x)$ equal to the identity function
[i.e., $f(x)=((\vx_1),(\vx_2),\dots,(\vx_n))]$,
and drop the memory term, thus defining the
sequence $\bx^t\in\reals^N$ by
%
\begin{equation}
\bx^{t+1} = A \bx^t. %
\end{equation}
Let us focus on, say, coordinate $1$ of $\bx^t$.
An explicit calculation yields (recall that, by convention, $A_{ii}=0$)
%
\begin{eqnarray}
\bbx^1_1 & =& \sum
_{i\in[N]}A_{1i},
\\\label{eq:T=2}
\bbx^2_1 & = &\sum_{i\in[N]}
\sum_{j\in[N]}A_{1i}A_{ij} = \sum
_{i\in[N]} A_{1i}^2 + \sum
_{i\in[N]}\sum_{j\in[N]\setminus
1}A_{1i}A_{ij}. %
\end{eqnarray}
Consider first $t=1$.
Under our assumptions, $\bx^1_1$ is a sum of i.i.d. random variables
with mean $0$ and variance $1/N$. By the central limit theorem, it
converges in distribution to a standard Gaussian random variable, as
predicted by Theorem~\ref{thm:SE}.

Consider next $t=2$. In equation (\ref{eq:T=2}) we decomposed the sum over
$\{i,j\}$ in a sum over terms with $j=1$, and a sum over terms with
$j\neq1$. The first sum converges almost surely to $1$ by the law of
large numbers.
It is easy to see that the second sum has expectation equal to zero
and variance equal to $(N-1)/N$ that converges to $1$. Indeed, a
slightly more complicated
calculation shows that it converges to a standard Gaussian. Overall,
$\bbx^{2}_1$ converges in distribution to a Gaussian with mean $1$ and
variance $1$, unlike what is predicted by Theorem~\ref{thm:SE} for
$\vx^2_1$. (Theorem~\ref{thm:SE} always predicts $\vx^t_i$ to have
asymptotically zero mean.)

Notice that the terms in the sum (\ref{eq:T=2}) are indexed by an
ordered triple $(1,i,j)$ with $i,j\in[N]$, $1\neq i$, $i\neq j$.
We can identify such a triple with a length $2$ rooted (directed
acyclic) path with vertices labeled by $1$ (the root), $i$, $j$: $j\to
i\to1$.
The terms that lead to a nonzero mean are those corresponding to
$j=1$, that is, with a one-step reversal in the order in which they visit
labels of $[N]$. These are paths of the form $1\to i\to1$.

Consider now adding back the memory term $\mathsf{b}_t f(x^{t-1};t-1)$.
It is easy to check that, in the present case [namely $f(x;t) = x$],
equation (\ref{eq:SimpleIteration}) reduces to $x^{t+1} = Ax^t-x^{t-1}$
and, in particular
\[
x^{2} = A^2\v1 -\v1 . %
\]
Comparing with equation (\ref{eq:T=2}), we see that the memory term
asymptotically cancels the effect
of one-step reversing paths.
The same analysis can be developed, with additional labor, to
subsequent iterations. At each $t$, the memory term cancels the effect
of one-step reversing paths, and the residual terms match the
prediction of Theorem~\ref{thm:SE}.

The proof follows a similar argument for a general polynomial $f(x)$.
As in the linear case, each coordinate $x^t_i$ can be expressed as a
sum of monomials in the independent random variables
$(A_{ij})_{i<j}$. The main difference is that now these monomials are
indexed by rooted \emph{trees} instead of rooted \emph{paths} with
vertex labels
in $[N]$. To see this, consider the special case
%
\begin{equation}
f(x) = \bigl((\vx_1)^3,(
\vx_2)^3,\dots,(\vx_n)^3\bigr).
\end{equation}
Then, a direct calculation of iteration (\ref{eq:SimpleIteration})
yields
%
\begin{eqnarray}
\hspace*{25pt}\vx^1_1 &=& \sum
_{i\in[N]}A_{1i},
\\\label{eq:X2Example}
\hspace*{25pt}\vx^2_1 &=& \sum_{i\in[N]}\sum
_{j,k,l\in
[N]}A_{1i}A_{ij}A_{ik}A_{il}-
\mathsf{b}_1 ,\qquad \mathsf{b}_1 = \frac{1}{N}\sum
_{i\in[N]}3 \biggl(\sum_{j\in[N]}
A_{ij} \biggr)^{2}.  %
\end{eqnarray}
The monomials in
the sum appearing in the expression for $\vx^2_1$ in equation (\ref
{eq:X2Example})
can be associated to rooted directed trees as per Figure~\ref{fig:Tree}.
In this simple example it is easy to check that the memory term
exactly cancels the contribution of one-step reversing trees, that is, the
terms in the sum with labels $j=1$, or $k=1$, or $l=1$.
The other terms in the sum correspond to nonreversing trees (cf.
Section~\ref{sec:Polynomial}),
and their total contribution is asymptotically Gaussian with mean $0$
and variance as predicted\footnote{In the present case, since $q=k=1$,
state
evolution is a recursion for the single scalar $\Sigma^t$. We have
$\Sigma^1=1$,
$\Sigma^{2}= \E\{g(Z^1)^2\}$ for $Z^1\sim\normal(0,1)$ and $g(x) =
x^3$, whence $\Sigma^2 = 15$.} by state evolution,
equation (\ref{eq:GeneralSE}). Since this sum is a polynomial in the
independent random variables $(A_{ij})_{i<j}$, the method of moments
provides a natural path to prove the last statement.

\begin{figure}

\includegraphics{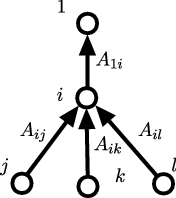}

\caption{Graphical representation of a term in
equation (\protect\ref{eq:X2Example}).}\label{fig:Tree}
\end{figure}

The actual proof of Theorems
\ref{thm:Universality} and \ref{thm:SE} in Section~\ref{sec:Polynomial}
follows the same intuition as
above, but of course, requires several technical steps:
\begin{longlist}[(6)]
\item[(1)] We introduce new quantities $\vz^t_{i}\in\reals^q$,
$i\in
[N]$ that are exactly equal to a sum of monomials in
the independent random variables $(A_{ij})_{i<j}$, indexed by labeled
nonreversing trees; see Lemma~\ref{lem:treerep}. (We refer to
Section~\ref{sec:Polynomial}
for a precise definition of ``nonreversing trees.'')
\item[(2)] We prove that, for our purposes, the distribution of the random
variable $\vx^t_i$ is accurately approximated by the distribution
of $\vz^t_i$; see Proposition~\ref{prop:amp}.
\item[(3)] We prove, under the same assumptions as in our universality result,
Theorem~\ref{thm:Universality}, the distribution of $\vz^t_i$ is
insensitive to the distribution of the matrix entries
$(A_{ij})_{i<j}$; cf. Proposition~\ref{prop:mom-equ}. This is done by
the moment method. Any moment of
$\vz^t_i$ is written as the expectation of a polynomial in the
$(A_{ij})_{i<j}$. We show that the only terms that matter are the ones
in which each $A_{ij}$ appears with degree at most two. Hence the
expectation only depends on the first two moments of the matrix
entries, which are fixed by assumption.

Together with the previous point, this immediately implies Theorem~\ref{thm:Universality}.
\item[(4)] In other to prove Theorem~\ref{thm:SE}, we introduce a third
sequence $\vy^t_{i}\in\reals^q$, $i\in[N]$ that is analogous to
the $\vz^t_i$ except for the fact that an independent copy of the
random variables $(A_{ij})_{i<j}$ is used at each generation in the
tree. This is analogous to drawing an independent copy of $A$ at each
iteration of a certain message passing algorithm (both $z^t$ and $y^t$
admit an iterative definition).
\item[(5)] In Proposition~\ref{prop:mom-equ-iid}, we prove that the
distribution of $\vz_i^t$ (and hence $\vx_i^t$) is, for our
purposes, accurately approximated by the distribution of $\vy^t_i$.
\item[(6)] Finally we exploit the fact that a fresh matrix $A$ is sampled
at each iteration to prove that state evolution holds for $\vy^t_i$;
cf. Proposition~\ref{prop:SE}.

By the previous point, this implies Theorem~\ref{thm:SE}.
\end{longlist}

In the next section we introduce some basic facts and notation. We
will implement the above strategy in Section~\ref{sec:Polynomial}.

%
%
\section{Notations and basic simplifications}\label{sec:Notations}

We will always view vectors as column vectors. The transpose of vector
$v$ is the row vector indicated by $v^{\sT}$.
Analogously, the transpose of a matrix
(or vector) $M$ is denoted by $M^{\sT}$.
For a vector $v\in\reals^m$, we denote its $\ell_p$ norm, $p\ge1$
by $\llVert v\rrVert _p \equiv(\sum_{i=1}^m\llvert v_i\rrvert ^p)^{1/p}$. This is extended in
the usual way to $p=\infty$. We will often omit the subscript if $p=2$.
For a
matrix $M$, we denote by $\llVert  M \rrVert _p$ the corresponding $\ell_p$
operator norm. The standard scalar product of $u,v\in\reals^m$ is
denoted by $ \langle u,v \rangle = \sum_{i=1}^m u_iv_i$.
Given $v\in\reals^m$, $w\in\reals^n$, we denote by $[v,w]\in\reals
^{m+n}$ the
(column) vector obtained by
concatenating $v$ and $w$.
The identity matrix is
denoted by $\id$, or $\id_{m\times m}$ if the dimensions need to be
specified.
The indicator function is $\ind( \cdot)$.
%
The set of first $m$ integers is indicated by $[m]= \{1,\dots,m\}$.
Finally, given $\vx= (x(1),x(2),\dots,x(q))\in\reals^q$ and $\vm=
(m(1),\dots,m(q))\in\naturals^q$, we write
%
\begin{equation}
\vx^{\vm} \equiv\prod_{r=1}^q
x(r)^{m(r)}. %
\end{equation}

Following the common practice, degenerate Gaussian distributions will
be considered Gaussian, without further qualification. In particular,
any distribution with finite support in $\reals^k$ is a finite mixture
of Gaussians.

In our proof of Theorem~\ref{thm:SE}
we will make use of the following simplification that lightens
somewhat the notation.

\begin{remark}\label{remark:NonRandom}
For proving Theorem~\ref{thm:SE}, it is sufficient to consider the
case in which $g\dvtx (\vx,Y,a,t)\mapsto g(\vx,Y,a,t)$ is independent of
$Y$.
\end{remark}

\begin{pf}
The basic idea of the construction is to enlarge $q$ in such a way to
keep track of the value of $Y(i)$ in the a subset of the coordinates of
$\vx_i^t$.

First of all, we can assume without loss of generality that the
measures $P_a$ are Gaussian.
Indeed if, for instance, $P_a$ is a mixture of $\ell$ Gaussians, $P_a =
w_1 P_{a,1}+w_2 P_{a,2}+\cdots+w_{\ell}P_{a,\ell}$, then
we can replace effectively the partition element $C^N_a$ by a finer partition
$C_{a,1}^N,\dots,C_{a,\ell}^N$ whereby $C_{a,1}^{N}\cup\cdots\cup
C_{a,\ell}^N=C_a^N$ and $\llvert C_{a,1}^N\rrvert,\dots,\llvert C_{a,\ell}^N\rrvert $ are
multinomial with parameters $(w_1,\dots,w_{\ell})$. Notice that this
finer partition is random, but $\llvert C_{a,i}^N\rrvert /N\to c_aw_{i}$ almost
surely, and therefore the theorem applies.

Assume therefore that the $P_a$ are Gaussian.
By replacing $g(\vx,Y,a,t)$ by $g'(\vx,Y,a,t)
=g(\vx,Q_aY+v_a,a,t)$ for suitable matrices $Q_a$, and vectors $v_a$,
we can always assume $Y_a\sim\normal(0,\id_{\tq\times\tq})$
for all $a$. Assume therefore $Y_a\sim\normal(0,\id_{\tq\times\tq})$.
Enlarge the space by letting $k'=k+\tq$, $N'=(\tq+1)N$ and
$C^{N'}_{a} =
\{N\ell+1,\dots,N(\ell+1)\}$, for $a=k+\ell>k$, while $C^{N'}_a =
C^N_a$ for $a\le k$. We further let $q'=q+\tq$
and define new functions $g':\reals^{q'}\times\reals^{\tq}\times
[k']\times
\naturals\to\reals^{q'}$ independent of the second argument ($Y$) as follows.
For $\vx\in\reals^q$, $\tvx\in\reals^{\tq}$, we let
\begin{eqnarray*}
g'_r \bigl((\vx,\tvx),Y,a,t \bigr) & = &
g_r(\vx,\tvx,a,t) \qquad\mbox{for } r\in\{1,\dots,q\}, a\in\{1,\dots,k
\},
\\
g'_{r} \bigl((\vx,\tvx),Y,a,t \bigr) &= &
0\qquad \mbox{for } r\in\{q+1,\dots,q+\tq\}, a\in\{1,\dots,k\},
\\
g'_{r} \bigl((\vx,\tvx),Y,a,t \bigr) & =
&0 \qquad\mbox{for } r\in\{1,\dots,q\}, a\in\{k+1,\dots,k+\tq\},
\\
g'_{q+\ell} \bigl((\vx,\tvx),Y,k+\ell',t
\bigr) & = & \ind\bigl(\ell =\ell'\bigr) \qquad\mbox{for } \ell,
\ell'\in\{1,\dots,\tq\}. %
\end{eqnarray*}
We further use matrix $A'$ constructed as follows:
$A'_{ij}= A_{ij}$ for $i,j\le N$ and $A_{ij}\sim\normal(0,1/N)$ if
$i>N$ or $j>N$. [Notice that $\E\{(A'_{ij})^2\} = 2/N'$, but this
amounts just to an overall rescaling and is of course immaterial.]
Clearly the functions $g'$ do not depend on $Y$ as claimed.
Further, $\tvx\sim\normal(0,\id_{\tq\times\tq})$ at all
iterations. Hence
the new iteration is identical to the original one when restricted on
$\{x_i(r)\dvtx  i\le N, r\le q\}$.
\end{pf}

%
%
\section{Proofs of Theorems \texorpdfstring{\protect\ref{thm:Universality}}{3} and \texorpdfstring{\protect\ref{thm:SE}}{4}}\label{sec:Polynomial}

In this section we consider the AMP iteration
(\ref{eq:AMPGeneralDef}), and prove
Theorem~\ref{thm:Universality} and Theorem~\ref{thm:SE}, and indeed
generalize the latter.

We extend the state evolution (\ref{eq:GeneralSE}) by defining for
each $t\geq s\geq0$ and for all $a\in[k]$, a positive semidefinite matrix
$\Sigma^{t,s}_a\in\reals^{(2q)\times(2q)}$ as follows. For boundary
conditions, we set
%
\begin{eqnarray}
\hspace*{30pt}\hSigma^{0,0}_a = \pmatrix{
\hSigma^0_a & \hSigma^0_a\vspace*{3pt}
\cr
\hSigma^0_a & \hSigma^0_a },\qquad
\hSigma^{t,0}_a = \pmatrix{ \hSigma^t_a
& 0
\cr
0& \hSigma^0_a }, \qquad\hSigma^{0,t}_a
= \pmatrix{ \hSigma^0_a & 0
\cr
0& \hSigma^t_a
}, %
\end{eqnarray}
with $\hSigma^t_a$ defined per equation (\ref{eq:GeneralSE}).
For any $s,t\ge1$, we set recursively
%
\begin{eqnarray}\label{eq:GeneralstSE}
\Sigma^{t,s}_a&=&\sum
_{b=1}^k c_bW_{ab}
\hSigma^{t-1,s-1}_b,
\\
\label{eq:GeneralstSE2}
\hSigma^{t,s}_a& = &\E \bigl\{ X_aX_a^{\sT}
\bigr\},\nonumber\\[-8pt]\\[-8pt]
X_a &\equiv&\bigl[g\bigl(Z^t_a,Y_a,a,t
\bigr),g\bigl(Z^s_a,Y_a,a,s\bigr)\bigr] ,
\bigl(Z_a^t,Z_a^s\bigr)\sim
\normal\bigl(0,\Sigma^{t,s}_a\bigr).\nonumber
\end{eqnarray}
Recall that $[g(Z^t_a,Y_a,a,t),g(Z^s_a,Y_a,a,s)]\in\reals^{2q}$ is the
vector obtained by concatenating $g(Z^t_a,Y_a,a,t)$ and $g(Z^s_a,Y_a,a,s)$.
Note that taking $s=t$ in (\ref{eq:GeneralstSE}), we recover the
recursion for $\Sigma^t_a$ given by equation (\ref{eq:GeneralSE}). Namely,
for all $t$ we have
%
\begin{eqnarray}\label{eq:OneTimeMatch}
\Sigma^{t,t}_a = \pmatrix{
\Sigma^t_a & \Sigma^t_a\vspace*{2pt}
\cr
\Sigma^t_a & \Sigma^t_a }.%
\end{eqnarray}

\begin{theorem}\label{thm:PolySE}
Let $\{(A(N),\cF_N,x^{0,N})\}_{N\ge1}$ be a polynomial and converging
sequence of instances and denote by $\{x^t\}_{t\ge0}$ the
corresponding AMP
orbit.

Fix $s,t\ge1$. If $s\neq t$,
further assume that the initial condition $x^{0,N}$ is
obtained by letting $\vx_i^{0,N}\sim Q_{a}$ independent and
identically distributed, with $Q_a$ a finite mixture of Gaussians for
each $a$.
Then, for each $a\in[k]$, and each locally Lipschitz function
$\psi\dvtx \reals^q\times\reals^q\times\reals^{\tq}\to\reals$ such that
$\llvert \psi(\vx,\vx',y)\rrvert \le K(1+\llVert y\rrVert _2^2+\llVert \vx\rrVert _2^2+\llVert \vx'\rrVert _2^2)^K$, we
have, in
probability,
\[
\lim_{N\to\infty} \frac{1}{\llvert C_a^N\rrvert }\sum
_{j\in C^N_a} \psi\bigl(\vx^t_j,
\vx^s_j,Y(j)\bigr) = \E \bigl[ \psi\bigl(Z^t_a,Z^s_a,Y_a
\bigr) \bigr],
\]
where $(Z^t_a,Z^s_a)\sim\normal(0,\Sigma^{t,s}_a)$ is independent of
$Y_a\sim P_a$.
\end{theorem}

Throughout this section, we will assume that
$\{(A(N),\cF_N,x^{0,N})\}$, $\{(\gA(N),\break \cF_N,x^{0,N})\}$, etc. are
$(C,d)$-\emph{regular polynomial} sequences of AMP
instances.
We will often omit explicit mention of this hypothesis.
Notice that Theorem~\ref{thm:Universality} holds \emph{per
realization} of
the functions $\cF_N$.
Because of this and because of Remark~\ref{remark:NonRandom}, we will
consider hereafter
$\cF_N$ to be nonrandom.

The rest of this section is organized as follows. In Section~\ref
{sec:MPDefinition} we introduce two new iterations that are useful
intermediary steps for our analysis. We show that the corresponding
variables admit representations as sums over trees in
Section~\ref{sec:TreeRep} and use them to prove basic properties of
these recursions in Sections \ref{sec:ProofMom}, \ref{sec:ProofMomIID}
and \ref{sec:PolyAMP}.
Theorems \ref{thm:Universality} and \ref{thm:PolySE}
are then proved in Sections \ref{sec:ProofUniversalityPolynom},
\ref{sec:ProofPolySE}. Because of equation (\ref{eq:OneTimeMatch}),
Theorem~\ref{thm:SE} follows as a special case of Theorem~\ref{thm:PolySE}.
Indeed, we will show that both
statements are equivalent through a reduction argument. Depending on
the application, Theorem~\ref{thm:PolySE} might be a more convenient
formulation of the state evolution and will be used in Section~\ref
{sec:NonSymmetric}.

%
%
\subsection{Message passing iteration}\label{sec:MPDefinition}

We define two new message passing sequences corresponding
to the instance $(A,\cF,x^{0,N})$.
For each $i\in[N]$ we use the short notation $[N]\ex i$ to denote the
set $[N]\ex\{i\}$.
We now define the sequence of vectors $(\vz^{t}_{i\to j})_{t\in
\naturals}$, where for each $i\neq j\in[N]$, $\vz^{t}_{i\to j}$ is
a vector in $\reals^q$ or equivalently for each $t\in\naturals$, we
can see $(\vz^t_{i\to j})$ as an $N\times N$ matrix with entries in
$\reals^q$ (diagonal elements are never used). The initial condition is
denoted by $\vz^0_{i\to j}\in\reals^q$ for any $i,j\in[N]$ and is
independent of $j$, such that $\vz^0_{i\to j}=\vx^{0,N}_{i}$ for all
$j\neq i$.
The $r$th coordinate of the vector $\vz_{i\to j}^{t+1}$ is defined by
the following recursion for $t\geq0$:
%
\begin{eqnarray}\label{eq:rec}z_{i\to j}^{t+1}(r)&=& \sum
_{\ell\in[N]\ex j} A_{\ell i} f^\ell_r
\bigl(\vz_{\ell\to i}^{t},t\bigr),
\end{eqnarray}
where $f^\ell_r(\cdot,t)\dvtx \reals^q\to\reals$ is the $r$th
coordinate of
$f^\ell(\cdot,t)$.

We also define for each $i\in[N]$ and $t\geq0$, the vector $\vz
^{t+1}_i\in\reals^q$ by
%
\begin{eqnarray}\label{eq:defz}z^{t+1}_i(r) &=& \sum
_{\ell\in[N]} A_{\ell i} f^\ell _r
\bigl(\vz_{\ell\to i}^{t},t\bigr).
\end{eqnarray}
Our first result establishes universality of the moments of
$\vz^t_{i\to j}$ for polynomial sequences of
instances.

\begin{proposition}\label{prop:mom-equ}
Let $(A(N),\cF_N,x^{0,N})_{N\ge1}$ and $(\gA(N),\cF
_N,x^{0,N})_{N\ge1}$
be any two $(C,d)$-regular polynomial sequences of AMP instances, that
differ only in the
distribution of the random matrices $A(N)$ and $\gA(N)$.
Assume that for all $N$ and all $i<j$, $\E\{A_{ij}^2\} = \E\{\gA
_{ij}^2\}$.
Denote by $\vz^t_i$ the orbit (resp., $\tilde{\vz}^t_i$) defined by
(\ref{eq:defz}) while iterating (\ref{eq:rec}) with matrix $A$
(resp., $\gA$).
Then for any $t\geq1$ and any $\vm=(m(1),\dots, m(q))\in\naturals
^q$, there
exists $K$ independent of $N$ such that, for any $i\in[N]$,
%
\begin{equation}\label{eq:mom-equ} \bigl\llvert \E \bigl[\bigl(\vz_{i}^t
\bigr)^{\vm} \bigr] -\E \bigl[ \bigl(\tvz_{i}^t
\bigr)^{\vm} \bigr]\bigr\rrvert \leq KN^{-1/2}.
\end{equation}
\end{proposition}

The proof of this proposition is provided in Section~\ref{sec:ProofMom}.

\begin{note}
In this statement and in the rest of this section, $K$ is always
understood as a function of
$d,t,q,m,C$ which may vary from line to line, but which is independent
of $N$.
\end{note}

Our second message passing sequence is defined as follows: for a
$(C,d)$-regular sequence of instances $(A(N),\cF_N,x^{0,N})_{N\ge1}$, we
define for each $N$, an i.i.d. sequence of $N\times N$ random matrices
$\{A^t\}_{t\in\naturals}$ such that $A^0=A(N)$. Then we define $(\vy
^t_{i\to j})$ by $\vy^0_{i\to j}=\vx^{0,N}_{i}$ and for $t\geq0$
%
\begin{eqnarray}\label{eq:rec-iid} y^{t+1}_{i\to j}(r) &=& \sum
_{\ell\in[N]\ex j} A^{t}_{\ell i} f_r^{\ell
}
\bigl(\vy_{\ell\to i}^{t},t\bigr),
\end{eqnarray}
and
%
\begin{eqnarray}\label{eq:defy} y^{t+1}_i(r) &=&\sum
_{\ell\in[N]} A^{t}_{\ell i} f_r^{\ell}
\bigl(\vy _{\ell
\to i}^{t},t\bigr).
\end{eqnarray}
The asymptotic analysis of $y^t$ is particularly simple because an
independent random matrix $A^t$ is used at each iteration. In
particular, it is easy to establish state evolution for $y^t$.
Our next result shows that $y^t$ provides a good approximation for~$z^t$.

\begin{proposition}\label{prop:mom-equ-iid}
Let $(A(N),\cF_N,x^{0,N})_{N\ge1}$ be a $(C,d)$-regular polynomial
sequence of
instances.
Let $\vz^t_{i}$ and $\vy^t_{i}$ be the sequences of vectors obtained by
iterating (\ref{eq:rec})--(\ref{eq:defz}) and (\ref
{eq:rec-iid})--(\ref
{eq:defy}), respectively.
Then for any $t\geq1$ and any $\vm=(m(1),\dots, m(q))\in\naturals
^q$, there
exists $K$ independent of $N$ such that, for any $i\in[N]$,
\[
\bigl\llvert \E \bigl[ \bigl(\vz_{i}^t
\bigr)^{\vm} \bigr]- \E \bigl[ \bigl(\vy_{i}^t
\bigr)^{\vm} \bigr]\bigr\rrvert \leq KN^{-1/2}.
\]
\end{proposition}

The proof of this proposition is provided in Section~\ref{sec:ProofMomIID}.

Finally, recall that we defined the sequences $(\vx^t_i)_{t\in
\naturals
}$ with $\vx^t_i\in\reals^q$, by $\vx^0_i$ and for $t\geq0$,
\[
x^{t+1}_i(r) = \sum_{\ell}A_{\ell i}f^{\ell}_r
\bigl(\vx^t_\ell,t\bigr)- \sum
_{\ell
}A_{\ell i}^2\sum
_s f^{i}_s\bigl(\vx^{t-1}_i,t-1
\bigr) \frac{\partial f^\ell
_r}{\partial x(s)}\bigl(\vx^{t}_\ell,t\bigr).
\]

\begin{proposition}\label{prop:amp}
Let $(A(N),\cF_N,x^{0,N})_{N\ge1}$ be a $(C,d)$-regular polynomial
sequence of
instances. Denote by $\{x^t\}_{t\geq0}$ the corresponding AMP
sequence 
and by $\{z^t\}_{t\geq0}$ the sequence defined
by (\ref{eq:defz}) while iterating (\ref{eq:rec}).
Then for any
$t\geq1$ and $m(1),\dots, m(q)\geq0$, there exists $K$ independent
of $N$ such that, for any $i\in[N]$,
\[
\bigl\llvert \E \bigl[ \bigl(\vx^t_i
\bigr)^{\vm} \bigr]-\E \bigl[ \bigl(\vz ^t_i
\bigr)^{\vm
} \bigr]\bigr\rrvert \leq K N^{-1/2}.
\]
\end{proposition}

The proof of this proposition is provided in Section~\ref{sec:PolyAMP}.

%
\subsection{Tree representation}\label{sec:TreeRep}

By assumption of Proposition~\ref{prop:mom-equ}, we have for each
$\ell
\in[N]$ and $r\in[q]$,
%
\begin{equation}\label{eq:fpoly} f_r^\ell(\vz,t) = \sum
_{i_1+\cdots+i_q\leq d}{c^\ell_{i_1,\dots,
i_q}(r,t)} \prod
_{s=1}^q {z(s)^{i_s}},
\end{equation}
where each coefficient $c^\ell_{i_1,\dots, i_q}(r,t)$ belongs to
$\reals
$ and has absolute value bounded by $C$ (uniformly in $\ell\in[N]$,
$i_1,\dots, i_q$, and $t\in\naturals$).

We now introduce families of finite rooted labeled trees that will
allow us to get a simple expression for the $z^t_{i\to j}(r)$'s and
$z^t_i(r)$; see Lemma~\ref{lem:treerep} below.
For a vertex $v$ in a rooted tree $T$ different from the root, we
denote by $\pi(v)$ the parent of $v$ in $T$.
We denote the root of $T$ by $\root$.
We consider that the edges of $T$ are directed towards the root and
write $(u\to v)\in E(T)$ if $\pi(u)=v$.
The unlabeled trees that we consider are such that the root and the
leaves have degree one; each other vertex has degree at most $d+1$,
that is, has at most $d$ children.
We now describe the possible labels on such trees. The label of the
root is in $[N]$, the label of a leaf is in $[N]\times[q]\times
\naturals^q$ and all other vertices have a label in $[N]\times[q]$.
For a vertex $v$ different from the root or a leaf, we denote its label
by $(\ell(v),r(v))$ and call $\ell(v)$ its type and $r(v)$ its mark.
The label (or type) of the root is also denoted by $\ell(\circ)$; the
label of a leaf $v$ is denoted by $(\ell(v),r(v), v[1],\dots, v[q])$.
For a vertex $u\in T$, we denote $\llvert u\rrvert $ its generation in the tree, that
is, its graph-distance from the root. Also for a vertex $u\in T$ (which
is not a leaf), we denote by $u[r]$ the number of children of $u$ with
mark $r\in[q]$ (with the convention $u[0]=0$). The children of such a
node are ordered with respect to their mark: the labels of the children
of $u$ are then $(\ell^1,1),\dots, (\ell^{u[1]},1),(\ell^{u[1]+1},2),\dots,(\ell^{u[1]+\cdots+u[q]},q)$, where each
$(\ell^{u[0]+\cdots+u[i]},\dots, \ell^{u[0]+\cdots+u[i+1]-1})$ is a
$u[i+1]$-tuple with coordinates in $[N]$.
We denote by $L(T)$ the set of leaves of a tree $T$, that is, the set of
vertices of $T$ with no children. For $v\in L(T)$, its label $(\ell
(v),r(v), v[1],\dots, v[q])$
is such that for all $i\in[q]$, $v[i]\in\naturals$ and $v[1]+\cdots
+v[q]\leq
d$. We will distinguish between two types of leaves: those with
maximal depth $t=\max\{\llvert v\rrvert, v\in L(T)\}$ and the remaining ones. If
$v\in L(T)$ and $\llvert v\rrvert \leq t-1$, then we impose
$v[1]=\cdots=v[q]=0$. This case corresponds to ``natural'' leaves, and
since they have no children, the notation is consistent with the
notation introduced for other nodes of the tree. For all other leaves,
we do not make this assumption so that $v[1]+\cdots+v[q]$ can take any
value in $[d]$. These leaves are ``artificial'' and can be thought of as
leaves resulting from cutting a larger tree after generation $t$ so
that the vector of the $v[r]$'s keeps the information on the number of
children with mark $r$ in the original tree.

\begin{definition}
We denote by $\cTall^t$ the set of labeled trees $T$ with $t$
generations as above.
We let $\cT^t\subseteq\cTall^t$ denote the subset of such trees
that satisfy the following additional condition:
\begin{longlist}[(1)]
\item[(1)] If $v_1=\circ,v_2,\dots,v_k$ is a path starting from the
root [i.e., with $\pi(v_{i+1})=v_i$ for $i\geq1$], then the
corresponding sequence of types $\ell(v_i)$ is nonbacktracking. That
is, for any $1\leq i\leq k-2$, the three labels $\ell(v_i),\ell
(v_{i+1})$ and $\ell(v_{i+2})$ are distinct.
\end{longlist}
We also let $\cU^t$ be the same set of trees from which marks have been
removed (i.e., we identify any two trees that differ in the marks but
not on type). Analogously, $\cUall^t$ is the set of trees in which
marks have been removed, but do not necessarily satisfy the
nonbacktracking condition 1.
\end{definition}

For a labeled tree $T\in\cT^t$ and a set of coefficients $\vc
=(c^\ell
_{i_1,\dots,i_q}(r,t))$, we define three weights:
\begin{eqnarray*}
A(T) &=& \prod_{(u\to v)\in E(T)}A_{\ell(u) \ell(v)},
\\
\Gamma(T,\vc,t) &=&\prod_{(u\to v)\in E(T)} c^{\ell(u)}_{u[1],\dots,u[q]}
\bigl(r(u),t-\llvert u\rrvert \bigr),
\\
x(T) &=& \prod_{v\in L(T)}\prod
_{s=1}^q{ \bigl(x^{0,N}_{\ell
(v)}(s)
\bigr)^{v[s]}}.
\end{eqnarray*}

We define:
\begin{longlist}[(b)]
\item[(a)] $\cT^t_{i\to j}(r)\subset\cT^t$ the family of trees
such that:
(i) The root has type $i$; (ii) The type of the child of the
root, denoted by $v$, is $\ell(v)\notin\{ i,j\}$ and its
mark is $r(v)=r$.
\item[(b)] $\cT^t_{i}(r)\subset\cT^t$ the family of trees such
that: (i) the root has type $i$; (ii) the type of the child of the root,
denoted by $v$, is $\ell(v)\neq i$, and
its mark is $r(v) = r$.
\end{longlist}
The sets of trees $\cU_i^t(r)$ and $\cU_{i\to j}^t(r)$ are obtained
from $\cT^t_{i}(r)$ and $\cT^t_{i\to j}(r)$ by removing marks.

\begin{lemma}\label{lem:treerep}
Let $(A(N),\cF_N,x^{0,N})_{N\ge1}$ be a polynomial sequence of AMP instances.
Denote by $\vz^t_i$ the orbit defined by
(\ref{eq:defz}) while iterating (\ref{eq:rec}) with matrix $A$.
Then
%
\begin{eqnarray}\label{eq:tree1}z^{t}_{i\to j}(r) &=& \sum
_{T\in\cT^t_{i\to j}(r)} A(T)\Gamma(T,\vc,t)x(T),
\\
\label{eq:tree2}z^{t}_{i}(r) &=& \sum
_{T\in\cT^t_{i}(r)} A(T)\Gamma (T,\vc,t)x(T).
\end{eqnarray}
\end{lemma}

\begin{pf}
We first prove (\ref{eq:tree1}) by induction on $t$.
For $t=1$ we have, by definition,
\begin{eqnarray*}
z^1_{i\to j}(r) &=& \sum_{\ell\in[N]\ex j}
\sum_{i_1+\cdots+i_q\leq
d}A_{\ell i} c^\ell_{i_1,\dots,i_q}(r,0)
\prod_{s=1}^q{ \bigl( x^{0,N}_{\ell\to i}(s)
\bigr)^{i_s}}.
\end{eqnarray*}
This expression corresponds exactly to equation (\ref{eq:tree1}) since
trees in $\cT^1_{i\to j}(r)$ have a root with label $i$ and with one
child with label $(\ell,r,i_1,\dots,i_q)$ for some $\ell\notin\{
i,j\}$
and $i_1+\cdots+i_q\leq d$.

To prove the induction, we start with equation (\ref{eq:rec}), which yields
\begin{eqnarray*}
z^{t+1}_{i\to j} (r) &=& \sum_{\ell\in[N]\ex j}A_{\ell i}
\sum_{i_1+\cdots+i_q\leq d}c^\ell_{i_1,\dots, i_q}(r,t)
\prod_{s=1}^q \bigl( z^t_{\ell\to i}(s)
\bigr)^{i_s}.
\end{eqnarray*}
Using the induction hypothesis, we get
\begin{eqnarray*}
\prod_{s=1}^q \bigl(
z^t_{\ell\to i}(s) \bigr)^{i_s} &=& \prod
_{s=1}^q \biggl(\sum_{T\in\cT^t_{\ell\to i}(s)}
A(T)\Gamma (T,\vc,t)x(T) \biggr)^{i_s}
\\
&=& \sum_{ [\cT^t_{\ell\to i}(s)]^{i_1+\cdots+i_q}} \prod
_{s=1}^q\prod
_{k=1}^{i_s} A\bigl(T^s_k
\bigr)\Gamma\bigl(T^s_k,\vc,t\bigr)x\bigl(T^s_k
\bigr),
\end{eqnarray*}
where the last expression is a sum over all
$(i_1+\cdots+i_q)$-tuples of trees with the first $i_1$ trees in
$\cT^t_{\ell\to i}(1)$, the
following $i_2$ in $\cT^t_{\ell\to i}(2)$, and so on.

Hence, we get
%
\begin{eqnarray}
\label{eq:indt+1}z^{t+1}_{i\to j} (r) = \sum
_{\ell\in[N]\ex
j}\sum_{i_1,\dots, i_q}\sum
_{[\cT^t_{\ell\to i}(s)]^{i_1+\cdots+i_q}}&& A_{\ell
i} c^\ell_{i_1,\dots, i_q}(r,t)\nonumber\\[-8pt]\\[-8pt]
&&{}\times
\prod_{s=1}^q\prod
_{k=1}^{i_s} A\bigl(T^s_k
\bigr)\Gamma\bigl(T^s_k,\vc,t\bigr)x\bigl(T^s_k
\bigr).\nonumber
\end{eqnarray}
The claim now follows by observing that the set of trees in
$\cT^{t+1}_{i\to j}(r)$ is in bijection with the set of pairs
constituted by a label $(\ell,r)$ with $\ell\notin\{i,j\}$ and a
$(i_1+\cdots+i_q)$-tuple of trees with exactly $i_s$ trees belonging
to $\cT^k_{\ell\to i}(s)$ for $s\in[q]$. Indeed, take a root with
label $i$ and one child, say $v$, with label $(\ell,r)$ for some
$\ell\notin\{ i,j\}$ and with a $(i_1+\cdots+i_q)$-tuple of trees
with exactly $i_s$ trees belonging to $\cT^t_{\ell\to i}(s)$ for
$s\in
[q]$. Now take $v$ as the root of these $(i_1+\cdots+i_q)$ trees, the
order in the tuple giving the order of the subtrees of $v$. Note that
the root of each subtree in $\cT^t_{\ell\to i}(s)$ has type $\ell$ and
in the resulting tree will get mark $r$.

The proof of (\ref{eq:tree2}) follows by the same argument, the only
change is that in the sum in (\ref{eq:indt+1}), we need now to include
$\ell=j$.
\end{pf}

%
\subsection{Proof of Proposition \texorpdfstring{\protect\ref{prop:mom-equ}}{1}}
\label{sec:ProofMom}

We are now in position to prove Proposition~\ref{prop:mom-equ}.

\begin{pf}
For notational simplicity, we consider the case $m(r) = m$, and
$m(s)=0$ for all $s\in[q]\setminus r$.
Thanks to Lemma~\ref{lem:treerep}, we have
%
\begin{equation}
\label{eq:mom} %
\E \bigl[ \bigl(z^t_{i}(r)
\bigr)^m \bigr] = \sum_{T_1,\dots, T_m\in\cT^t_{i}(r)} \Biggl[
\prod_{\ell=1}^{m}\Gamma(T_\ell,
\vc,t) \Biggr]\E \Biggl[\prod_{\ell
=1}^m
x(T_\ell) \Biggr]\E \Biggl[\prod_{\ell=1}^{m}A(T_\ell)
\Biggr] .\hspace*{-40pt} %
\end{equation}
Since $\vc$ is fixed in this section, we omit to write it in
$\Gamma(T,t)$. Notice that the general case $\vm=
(m(1),\dots,m(q))\in\naturals^q$ admits a very similar representation
whereby the sum over $T_1,\dots,T_m\in\cT^t_{i}(r)$ is replaced by
sums over $T_1,\dots,T_{m(1)}\in\cT^t_i(1)$,
$T_1,\dots,T_{m(2)}\in\cT^t_i(2)$,
\dots, $T_1,\dots,T_{m(q)}\in\cT^t_i(q)$. The argument goes through
essentially unchanged.

We have $\Gamma(T_\ell,t) \leq C^{d^{t+1}}$.
We first concentrate on the term $\E [\prod_{\ell
=1}^{m}A(T_\ell
) ]$.
Recall that, from sub-Gaussian property of entries of $A$: $\E (
e^{\lambda A_{ij}} )\leq e^{\fracc{C\lambda^2}{2N}}$.
Now using Lemma~\ref{lem:x_s_ineq} from Appendix \ref{app:calculus} we
get for all $i< j\in[N]$,
%
\begin{equation}
\label{eq:uppsubg} %
\E \bigl[ \llvert A_{ij}\rrvert
^s \bigr]\leq2 \biggl( \frac{s}{e} \biggr)^{s}
\lambda ^{-s}e^{\fracc{C\lambda^2}{2N}} \leq2C^{\fraca{s}{2}} \biggl(
\frac
{s}{e} \biggr)^{\fraca{s}{2}}N^{-\fraca{s}{2}}, %
\end{equation}
obtained by taking $\lambda= \sqrt{Ns/C}$.

For a labeled tree $T$, we define $\phi(T)=\{\phi(T)_{ij}\in
\naturals,
i\leq j\in[N]\}$ where $\phi(T)_{ij}$ is the number of occurrences in
$T$ of an edge $(u\to v)$ with endpoints having types $\ell(u),\ell
(v)\in\{i,j\}$.
Hence we have
%
\begin{eqnarray}
\label{eq:prodA}
 A(T) &=& \prod_{i < j\in[N]}A_{ij}^{\phi(T)_{ij}}
\quad\mbox{and}\nonumber\\[-8pt]\\[-8pt]
\E \Biggl[\prod_{\ell=1}^{m}A(T_\ell)
\Biggr]&=& \prod_{i<j\in[N]}\E \bigl[A_{ij}^{\sum_{\ell=1}^m\phi(T_\ell
)_{ij}}
\bigr].\nonumber
\end{eqnarray}

Since the mean of each entry of the matrix $A$ is zero, in equation
(\ref{eq:mom}), we can restrict the sum to $T_1,\dots, T_m$ such that
for all $i<j\in[N]$, $\sum_{\ell=1}^m\phi(T_\ell)_{ij}< 2$ implies
$\sum_{\ell=1}^m\phi(T_\ell)_{ij}=0$.
For such a $m$-tuple $T_1,\dots, T_m$, we denote $\mu=\mu(T_1,\dots,
T_m) =\sum_{i<j}\sum_{\ell=1}^m\phi(T_\ell)_{ij}$.
Using equation (\ref{eq:uppsubg}), we get
%
\begin{eqnarray}\label{eq:1exp}
\Biggl\llvert \E \Biggl[\prod_{\ell=1}^{m}A(T_\ell)
\Biggr]\Biggr\rrvert &\leq& \prod_{i<j\in[N]}\E \bigl[
\llvert A_{ij}\rrvert ^{\sum_{\ell=1}^m\phi(T_\ell
)_{ij}} \bigr]
\nonumber
\\[-8pt]\\[-8pt]
&\le& \biggl(2C^{\fraca{\mu}{2}} \biggl( \frac{\mu}{e}
\biggr)^{\fraca
{\mu
}{2}} \biggr)^{\mu/2}N^{-\fraca{\mu}{2}},\nonumber
\end{eqnarray}
since in the product on the right-hand side of (\ref{eq:prodA}), there
are at most $\mu/2$ terms different from one.

We now compute an upper bound on
\[
\sum^{(\mu)}_{T_1,\dots,T_m}\E \Biggl[ \Biggl\llvert
\prod_{\ell=1}^m x(T_\ell )\Biggr
\rrvert \Biggr] ,
\]
where the sum ${\sum}^{(\mu)}$
ranges
on $m$-tuple of trees in $\cT^t_{i}(r)$ such that\break
$\sum_{i<j}\sum_{\ell=1}^m\phi(T_\ell)_{ij}=\mu$, and moreover there\vspace*{1pt}
exists $i<j\in[N]$ such that $\sum_{\ell=1}^m\phi(T_\ell)_{ij}\geq
3$. Let $\vG$ be the graph obtained by taking the union of~the
$T_{\ell}$'s and identifying vertices $v$ with the same type
$\ell(v)$.
We define $e(T_1,\dots,\break T_m) = \sum_{i<j}\ind(\sum_{\ell=1}^m \phi
(T_\ell)_{ij}\geq1)$ which is the number of edges counted without
multiplicity in $\vG$. Since there exists $i<j$ with $\sum_{\ell
=1}^m\phi(T_{\ell})_{ij}\geq3$, we have $3+2(e(T_1,\dots,
T_m)-1)\leq
\mu$, that is, $e(T_1,\dots, T_m)\leq\frac{\mu-1}{2}$.

Now note that for any $\vx\in\reals^q$, we have for any $p\geq2$,
\[
\llVert \vx\rrVert _p^p\leq\llVert \vx\rrVert
_2^p\leq\max \bigl(\exp\bigl( \llVert \vx\rrVert
_2^2\bigr), p^p \bigr).
\]
Hence the condition $\frac{1}{N} \sum_{i=1}^N \exp(\llVert \vx^{0,N}_i\rrVert
_2^2/C)\leq
C$ ensures that for any $p\geq2$,
\[
\frac{1}{N} \sum_{i=1}^N\bigl
\llVert \vx^{0,N}_i\bigr\rrVert _p^p
\leq C_p.
\]
Therefore,
%
\begin{eqnarray}\label{eq:2exp}
\sum_{T_1,\dots, T_m}^{(\mu)}\prod
_{\ell=1}^m \bigl\llvert x(T_\ell)\bigr
\rrvert &\leq& \Biggl(q^m \sum_{j=1}^N
\sum_{s=1}^q \bigl(1+\bigl\llvert
x^{0,N}_j(s)\bigr\rrvert +\cdots +\bigl\llvert
x^{0,N}_j(s)\bigr\rrvert ^{md} \bigr)
\Biggr)^{\fracb{\mu-1}{2}}
\nonumber
\\
&=& \bigl(q^mN\bigr)^{\fracb{\mu-1}{2}} \Biggl(q+\sum
_{k=1}^{md}\frac{1}{N}\sum
_{j=1}^N\bigl\llVert \vx^{0,N}_j
\bigr\rrVert _k^k \Biggr)^{\fracb{\mu-1}{2}}
\\
&\leq& \Biggl(q^m\Biggl(q+\sum
_{k=1}^{md}C_k\Biggr)
\Biggr)^{\fracb{\mu
-1}{2}}N^{\fracb
{\mu-1}{2}},
\nonumber
\end{eqnarray}
where the last inequality is valid for $N\geq C$. To see why (\ref
{eq:2exp}) is true, note that
the graph $\vG$ is connected since all trees $T_1,\dots,T_m$ have the
same type $i$ at the root. Therefore,
the number of vertices in $\vG$ is at most
$e(T_1,\dots,T_m)+1\leq\frac{\mu-1}{2}+1$. Since all $T_{\ell}$'s have
the same root which has type $i$, $\vG$ has at most
$\frac{\mu-1}{2}$ distinct vertices which are distinct from the one
associated to the root. In particular, all trees $T_1,\dots,T_m$
together have at most
$\frac{\mu-1}{2}$ distinct types among their leaves. The factor $q^m$
comes from the fact that
for each type $j$ there are at most $q^m$ choices for its $m$
marks $r$ corresponding to the $m$ trees. Now each leaf with
type $j$ will contribute a factor
$\prod_{s=1}^q (x^{0,N}_j(s) )^{n_s}$ with $\sum_s n_s\leq md$.\vspace*{1pt}

It is now easy to conclude, since we can decompose the sum in (\ref
{eq:mom}) in two terms, the first term say $S_1(A)$ consists of the
contribution of the $m$-tuples $T_1,\dots, T_m$ such that for all
$i,j$, $\sum_{\ell=1}^m\phi(T_{\ell})_{ij}\in\{0,2\}$, while the
second term denoted by $S_2(A)$ consists of the remaining contribution.
We have $S_1(A)=S_1(\tilde{A})$ and, using (\ref{eq:1exp}) and (\ref
{eq:2exp}), we get
%
\begin{equation}\label{eq:S2(A)}\bigl\llvert S_2(A)\bigr\rrvert \leq\sum
_{\mu\leq md^{t+1}} C^{d^{t+1}} C'N^{\fracb{\mu-1}{2}}C''N^{-\fraca{\mu}{2}}
= O \bigl( N^{-\fraca
{1}{2}} \bigr),
\end{equation}
which completes the proof Proposition~\ref{prop:mom-equ}. Here we used\vspace*{1pt}
the fact that
all values $\mu, q$ and $\{C_k\}_{k=0}^{md}$ are independent of $N$.
\end{pf}

We end this section by showing that the term $S_1(A)$ can be further reduced.
This result will be useful in the sequel, and we state it as
the following lemma.

\begin{lemma}\label{lem:tree}
Recall that we denoted by $S_1(A)$ the term in the sum (\ref{eq:mom}),
consisting of the contribution of the $m$-tuples $T_1,\dots, T_m$
such that for all $i,j$, $\sum_{\ell=1}^m\phi(T_{\ell})_{ij}\in
\{0,2\}$. We further decompose $S_1(A)=T(A)+R(A)$ in two terms where
the first term $T(A)$ corresponds to the sum over trees $T_1,\dots,
T_m$ such that the resulting graph $\vG$ obtained by taking the union
of the $T_{\ell}$'s and identifying vertices $v$ with the same type
$\ell(v)$, is a tree (each edge having multiplicity two). Then there
exists $K$ (independent of $N$) such that
\begin{eqnarray*}
  \bigl\llvert \E \bigl[ z^t_i(r)^m
\bigr]-T(A)\bigr\rrvert &=& KN^{-1/2},
\\
  \bigl\llvert \E \bigl[ z^t_i(r)^m \bigr]
\bigr\rrvert &\leq& K ,\\ \bigl\llvert \E \bigl[ z^t_{i\to j}(r)^m
\bigr]\bigr\rrvert &\leq& K. %
\end{eqnarray*}
\end{lemma}

\begin{pf}
We have, by definition,
$\E [ (z^t_i(r)^m ) ]=T(A)+R(A)+S_2(A)$, so that
thanks to (\ref{eq:S2(A)}), we need only to show that
$R(A)=O (N^{-1/2} )$.

For any $m$-tuple $T_1,\dots, T_m$ such that for all $i,j$, $\sum_{\ell
=1}^m\phi(T_{\ell})_{ij}\in
\{0,2\}$, we have with the same notation as above, $e(T_1,\dots, T_m)=
\frac{\mu}{2}$. The number of vertices in $\vG$ is at most
$1+e(T_1,\dots, T_m)$ with equality if and only if $\vG$ is a tree
(remember that $\vG$ is always connected as all trees $T_\ell$'s share
the same root). Hence for the cases that $\vG$ is not a tree it has at
most $\frac{\mu}{2}-1$ vertices that
serve as leaves of a tree among $T_1,\dots, T_m$.
By the same argument as above we get
%
\begin{eqnarray}
\label{eq:T(A)}\bigl\llvert T(A)\bigr\rrvert &\leq&\sum
_{\mu\leq md^{t+1}} KN^{\fraca{\mu
}{2}}N^{-\fraca{\mu}{2}}=O(1),
\\
\label{eq:R(A)}\bigl\llvert R(A)\bigr\rrvert &\leq& \sum
_{\mu\leq md^{t+1}} KN^{\fraca{\mu
}{2}-1} N^{-\fraca{\mu}{2}}=O
\bigl(N^{-1}\bigr), %
\end{eqnarray}
and the claim follows.
\end{pf}

%
\subsection{Proof of Proposition \texorpdfstring{\protect\ref{prop:mom-equ-iid}}{2}}
\label{sec:ProofMomIID}

The proof follows the same approach as for Proposition~\ref
{prop:mom-equ}. For notational simplicity, we consider the case $m(r) =
m$ and $m(s)=0$ for all $s\in[q]\setminus r$. The general case follows
by the same argument.
For $\vy$, we are using a different matrix at each iteration and we
need to define a new weight associated to trees $T\in\cT^t$ as follows:
%
\begin{eqnarray}
\label{eq:defA}\overline{A}(T,t) &=& \prod_{(u\to v)\in
E(T)}A^{t-\llvert u\rrvert }_{\ell(u) \ell(v)}.
\end{eqnarray}
In the particular case where the sequence $\{A^t\}_{t\in\naturals}$ is
constant (i.e., equals to $A$), this expression reduces to $A(T)$
defined previously.
Similar to Lemma~\ref{lem:treerep} for $\vx$, we have now
\begin{eqnarray*}
\label{eq:tree-iid1}y^{t}_{i\to j}(r) &=& \sum
_{T\in\cT^t_{i\to j}(r)} \overline{A}(T,t)\Gamma(T,\vc,t)x(T),
\\
\label{eq:tree-iid2}y^{t}_{i}(r) &=& \sum
_{T\in\cT^t_{i}(r)} \overline {A}(T,t)\Gamma(T,\vc,t)x(T),
\end{eqnarray*}
so that we get
%
\begin{eqnarray}
\label{eq:mom-iid}&&\E \bigl[ \bigl(y^t_{i}(r)
\bigr)^m \bigr] \nonumber\\[-8pt]\\[-8pt]
&&\qquad= \sum_{T_1,\dots, T_m\in\cT^t_{i}(r)} \Biggl[
\prod_{\ell=1}^{m}\Gamma (T_\ell,
\vc,t) \Biggr]\E \Biggl[\prod_{\ell=1}^{m}x(T_\ell)
\Biggr]\E \Biggl[\prod_{\ell=1}^{m}
\overline{A}(T_\ell,t) \Biggr].\nonumber
\end{eqnarray}
For a labeled tree $T$, we define $\varphi(T)=\{\varphi(T)^g_{ij}\geq
0, i\leq j\in[N], d\geq1\}$ where $\varphi(T)^g_{ij}$ is the number
of occurrences in $T$ of an edge $(u\to v)$ with endpoints having
labels $\ell(u),\ell(v)\in\{i,j\}$ and with generation $\llvert u\rrvert =g$. In
particular, we have $\sum_g \varphi(T)^g_{ij}=\phi(T)_{ij}$ which was
defined in the proof of Proposition~\ref{prop:mom-equ}.
Hence we have with $\mu=\sum_{i<j}\sum_{\ell=1}^m\phi(T_\ell)_{ij}$,
%
\begin{eqnarray}\label{eq:prodA-iid}
\Biggl\llvert \E \Biggl[\prod_{\ell=1}^{m}
\overline{A}(T_\ell,t) \Biggr]\Biggr\rrvert &\stackrel{\mathrm{(a)}} {=}&\prod
_{i<j\in[N]}\prod_{g}\bigl
\llvert \E \bigl[A_{ij}^{\sum
_{\ell=1}^m\varphi(T_\ell)^g_{ij}} \bigr]\bigr\rrvert
\nonumber
\\
&\le&\prod_{i<j\in
[N]}\prod
_{g}\E \bigl[|A_{ij}|^{\sum_{\ell=1}^m\varphi(T_\ell
)^g_{ij}} \bigr]
\\
&\stackrel{\mathrm{(b)}} {\leq}&
\biggl(2C^{\fraca{\mu}{2}} \biggl(
\frac{\mu}{e} \biggr)^{\fraca{\mu}{2}} \biggr)^{(\mu-1)/2}N^{-\fraca
{\mu
}{2}},
\nonumber
%
\end{eqnarray}
%
%
%
where (a) holds since $\{A^t\}_{t\in\naturals}$ is an i.i.d. sequence
with the same distribution as $A(N)$, and (b) follows by the same
argument as in (\ref{eq:1exp}).
Inequality \eqref{eq:prodA-iid} implies that bounds (\ref{eq:S2(A)})
and (\ref{eq:R(A)}) are still valid
with the weight of a tree given by (\ref{eq:defA}) (the term $\E
[\prod_{\ell=1}^{m}x(T_\ell) ]$ can be treated as in previous section).

As in the proof of Proposition~\ref{prop:mom-equ}, we define the graph
$\vG$ obtained by taking the union of the $T_\ell$'s and identifying
vertices $v$ with the same type $\ell(v)$.
By Lemma~\ref{lem:tree}, we need only to concentrate on the term
$T(\overline{A})$ corresponding to $m$-tuples $T_1,\dots, T_m$ such
that each edge in $\vG$ has multiplicity $2$ and such that $\vG$ is a
tree. Indeed, the proposition will follow once we prove
%
\begin{equation}
\label{eq:equal-iid}T(A)=T(\overline{A}),
\end{equation}
where $T(A)$ was defined in Lemma~\ref{lem:tree}, and $T(\overline{A})$
is the corresponding term with the weight of a tree given by (\ref{eq:defA}).
First note that for any $T_1,\dots, T_m$ such that $\E [\prod_{\ell
=1}^{m}\overline{A}(T_\ell,t) ]\neq0$, we have
\[
\E \Biggl[\prod_{\ell=1}^{m}
\overline{A}(T_\ell,t) \Biggr]=\E \Biggl[\prod
_{\ell=1}^{m}{A}(T_\ell) \Biggr].
\]
Now suppose that we have $\E [\prod_{\ell=1}^{m}{A}(T_\ell
)
]\neq0=\E [\prod_{\ell=1}^{m}\overline{A}(T_\ell,t) ]$.
This can only happen, if an edge in $\vG$ connecting types say $i$ and
$j$ has multiplicity $2$ but appears at different generations in the
original trees $T_\ell$'s. Suppose this edge appears twice in say $T_1$
at on the same branch and at different generations; that is, there
exists $(a\to b)$ and $(c\to d)\in E(T_1)$ with $\{\ell(a),\ell(b)\}
=\{
\ell(c),\ell(d)\}=\{i,j\}$, $\llvert a\rrvert < \llvert c\rrvert $, and the edge $(a\to b)$ is on
the path that
connects $c,d$ to the root. Thanks to the nonbacktracking property,
these two edges cannot be adjacent,
that is, $a\neq d$. But then these edges create a cycle in $\vG$,
contradiction.
Suppose now that these edge appears in $T_{1}$ and $T_{2}$ in different
generations,
that is, there exists $(a\to b)\in E(T_1)$ and $(c\to d)\in E(T_2)$
with $\{\ell(a),\ell(b),\ell(c),\ell(d)\}=\{i,j\}$ and $\llvert a\rrvert < \llvert c\rrvert $.
Then the same reasoning shows that they will create a cycle in $\vG$
since $b$ and $d$ are connected to the roots of $T_1$ and $T_2$
respectively which both identify to a single vertex in $\vG$. The latter
argument can be used for the case where both edges belong to the same
tree $T_1$, but they lie in different branches. Hence we obtain again a
contradiction.

%
\subsection{Proof of Proposition \texorpdfstring{\protect\ref{prop:amp}}{3}}
\label{sec:PolyAMP}

As in the proof of Proposition~\ref{prop:mom-equ}, we will rely on a
representation of $x^t_i(r)$ based on labeled trees defined as in
Section~\ref{sec:TreeRep}. In the present case, it is, however, more
convenient to work with trees from which marks have been removed,
that is, we identify any two trees in which the vertex marks are
different, but the types are the same.
Notice that equations (\ref{eq:tree1}), (\ref{eq:tree2}) imply
%
\begin{eqnarray}\label{eq:tree1BIS}
z^{t}_{i\to j}(r) &=& \sum_{T\in\cU^t_{i\to j}(r)}
A(T)\Gamma '(T,\vc,t)x(T),
\\
\label{eq:tree2BIS}
z^{t}_{i}(r) &=& \sum_{T\in\cU^t_{i}(r)}
A(T)\Gamma'(T,\vc,t)x(T),
\end{eqnarray}
where $\Gamma'(T,\vc,t)$ is obtained by summing $\Gamma(T,\vc,t)$ over
all trees $T$ that coincide up to marks. In the following, with a
slight abuse of notation, we will write $\Gamma(T,\vc,t)$ instead of
$\Gamma'(T,\vc,t)$.

In a directed labeled graph, we define a backtracking path of length
$3$ as a path $a\to b\to c\to d$ such that $\ell(a)=\ell(c)$ and
$\ell
(b)=\ell(d)$. We define a backtracking star as a set of vertices $a\to
b\to c$ and $a'(\neq a)\to b$ such that $\ell(a)=\ell(a')=\ell(c)$.
We define ${\cB}^t$ as the set of rooted labeled trees $T$ in $\cUall
^t$, that satisfy the following conditions:
\begin{itemize}
\item If $u\to v \in E(T)$, then $\ell(u)\neq\ell(v)$ and there exists
in $T$ at least one backtracking path of length $3$ or one backtracking star.
\end{itemize}
Then, we define $\cB^t_i$ as the subset of trees in $\cB^t$ with root
having type $i$ and only one child with type $\ell$ with $\ell\neq i$.

\begin{lemma}\label{lem:treeAMP}
Under the same assumptions as in Proposition~\ref{prop:amp}, we have
\begin{eqnarray*}
x^t_i(r) &=& z^t_i(r) +\sum
_{T\in\cB^t_i}A(T) \widetilde{\Gamma}(T,t,r)x(T),
\end{eqnarray*}
for some $\widetilde{\Gamma}(T,t,r)$ which is bounded uniformly as
$\llvert \widetilde{\Gamma}(T,t,r)\rrvert \le K(d,C,t)$.
\end{lemma}

\begin{pf}
Following the same argument as in Lemma~\ref{lem:treerep}, we first
prove by induction on $t$ that we can find $\widetilde{\Gamma}(T,t,r)$
such that
%
\begin{equation}\label{eq:x=sumTrees}
x^t_i(r) = \sum_{T\in\cUall_i^t}A(T)
\widetilde{\Gamma }(T,t,r)x(T),
\end{equation}
with $\llvert \widetilde{\Gamma}(T,t,r)\rrvert \le K(d,C,t)$.
The terms $A_{i\ell}f^\ell_r(\vx^t_\ell,t)$ can be handled exactly as
in Lemma~\ref{lem:treerep}. Concerning the terms
$A_{i\ell}^2f^{i}_s(\vx^{t-1}_i,t-1) \frac{\partial f^\ell
_r}{\partial
x(s)}(\vx^{t}_\ell,t)$, it can be interpreted as a sum on the
following trees in $\cUall$: the type of the root is $i$ and the root
has one child with type $\ell$. This child has at most $d-1$ subtrees
in $\cUall^t$ coming from the term $\frac{\partial f^\ell
_r}{\partial
x(s)}(\vx^{t}_\ell,t)$ (which is a polynomial with degree at most
$d-1$) and one child say $u$ with type $i$. This child $u$ is the root
of at most $d$ subtrees in $\cUall^{t-1}$ coming from the term
$f^{i}_s(\vx^{t-1}_i,t-1)$. We see that the resulting tree is in
$\cUall
^{t+1}$. Now to see that $\llvert \widetilde{\Gamma}(T,t,r)\rrvert \le K(d,C,t)$,
note that each polynomial $f^\ell_r( \cdot,t)$ [resp., $\frac
{\partial
f^\ell_r}{\partial x(s)}( \cdot,t)$] has coefficients bounded by $C$
(resp., $dC$) so that taking into account the contribution of each term
in decomposition (\ref{eq:x=sumTrees}), we easily get
\[
\bigl\llvert \widetilde{\Gamma}(T,t+1,r)\bigr\rrvert \le dC^2
\bigl[K(d,C,t)^d+K(d,C,t)^{d-1}K(d,C,t-1)^{d}
\bigr] .
\]

It remains to prove that $\widetilde{\Gamma}(T,t,r)$ agrees with the
expression in Lemma~\ref{lem:treerep},
[cf. equations (\ref{eq:tree1BIS}), (\ref{eq:tree2BIS})], for $T\in
\cU
^{t}_i(r)$ and is zero for trees in $\cUall^t\ex\cB^t_i$.
The proof of this fact will proceed by induction on $t$. The cases
$t=0,1$ are clear since $\cB^t_i=\emptyset$.
For $t\geq1$, we define
\begin{eqnarray*}
z^t_{\ell,i}(r) &=& A_{i\ell}f^{i}_r
\bigl(\vz^{t-1}_{i\to\ell},t-1\bigr),\qquad e^t_\ell(r)
= \sum_{T\in\cB^t_\ell}A(T)\widetilde{\Gamma}(T,t,r)x(T),
\\
d^t_{\ell,i}(r) &=& z^t_{\ell,i}(r)+e^t_\ell(r)
\end{eqnarray*}
so that we have by the induction hypothesis, $\vx^t_\ell=\vz^t_{\ell
\to i}+\vz^t_{\ell,i}+\ve^t_\ell=\vz^t_{\ell\to i}+\mathbf
{d}^t_{\ell,i}$.

Since $f^{\ell}_r( \cdot,t)$ is a polynomial, we have
\begin{eqnarray*}
f^{\ell}_r\bigl(\vx^t_\ell,t\bigr)
&=& f^\ell_r\bigl(\vz^t_{\ell\to i},t
\bigr)+\sum_s \bigl( z^t_{\ell,i}(s)+e^t_\ell(s)
\bigr)\frac{\partial f^\ell
_r}{\partial
x(s)}\bigl(\vz^{t}_{\ell\to i},t\bigr)
\\
&&{}+ \sum_{n_1+\cdots+n_q\geq2}\prod_{s=1}^q
\frac{ (d^t_{\ell,i}(s) )^{n_s}}{n_s!}\frac{\partial^{n_1+\cdots+n_q}f^\ell
_r}{\partial x(1)^{n_1}\cdots\partial z(q)^{n_q}}\bigl(\vz^t_{\ell\to i},t
\bigr),
\end{eqnarray*}
where the last sum contains a finite number of nonzero terms.

Multiplying by $A_{i\ell}$ and summing over $\ell\in[N]$, the first
term on the right-hand side gives exactly $z^{t+1}_{i}(r)$. The second
term gives
\[
\sum_\ell A_{\ell i}^2\sum
_s f^i_s\bigl(
\vz^{t-1}_{i\to
\ell},t-1\bigr)\frac{\partial f^\ell_r}{\partial x(s)}\bigl(
\vz^{t}_{\ell\to
i},t\bigr)+\sum_\ell
A_{\ell i}\sum_s e^t_{\ell}(s)
\frac{\partial
f^\ell_r}{\partial x(s)}\bigl(\vz^{t}_{\ell\to i},t\bigr).
\]
From now on and to lighten the notation, we omit the second argument of
the functions $f^\ell_r$.
Hence we have
%
\begin{eqnarray}\label{eq:zx}
x^{t+1}_i(r) &=& z^{t+1}_i(r)
-\sum_\ell A_{\ell i}^2\sum
_s \biggl(f^i_s\bigl(
\vx^{t-1}_{i}\bigr)\frac{\partial f^\ell_r}{\partial
x(s)}\bigl(\vx
^{t}_{\ell}\bigr)-f^i_s\bigl(
\vz^{t-1}_{i\to\ell}\bigr)\frac{\partial f^\ell
_r}{\partial x(s)}\bigl(
\vz^{t}_{\ell\to i}\bigr) \biggr)\nonumber
\\
&&{}+\sum_\ell A_{\ell i}\sum
_s e^t_{\ell}(s)
\frac
{\partial
f^\ell_r}{\partial x(s)}\bigl(\vz^{t}_{\ell\to i}\bigr)
\\
&&{}+\sum_\ell A_{\ell i}\sum
_{n_1+\cdots+n_q\geq2}\prod_{s=1}^q
\frac{ (d^t_{\ell,i}(s) )^{n_s}}{n_s!}\frac
{\partial
^{n_1+\cdots+n_q}f^\ell_r}{\partial x(1)^{n_1}\cdots\partial
x(q)^{n_q}}\bigl(\vz^t_{\ell\to i}\bigr).
\nonumber
\end{eqnarray}
We now show that each contribution on the right-hand side [except
$z^{t+1}_i(r)$] can be
written as a sum of terms $A(T)\widetilde{\Gamma}(T,t+1,r,x^0)$ over
trees $T\in\cB^{t+1}_i$ that we construct explicitly.

First consider the terms of the form $A_{\ell i} e^t_\ell(s)\frac
{\partial f^\ell_r}{\partial x(s)}(\vz^{t}_{\ell\to i})$. By
definition, $e^t_\ell(s)$ can be written as a sum over trees in $\cB
^t_{\ell}$, and by Lemma~\ref{lem:treerep}, the $r$th component of
$\vz
^t_{\ell\to i}$ can be written as a sum over trees in $\cU^t_{\ell
\to
i}(r)$. Hence by the same argument as in the proof of Lemma~\ref
{lem:treerep}, we see that $A_{\ell i} e^t_\ell(s)\frac{\partial
f^\ell
_r}{\partial x(s)}(\vx^{t}_{\ell\to i})$ can be written as a sum over
trees with root having type $i$, one child say $v$ with type $\ell$.
This vertex $v$ is the root of a tree in $\cB^t_\ell$ [corresponding to
the factor $e^t_\ell(s)$] and a set of trees in $\cU^t_{\ell\to
i}(1),\dots,\cU^t_{\ell\to i}(q)$ [corresponding to the factor
$\frac
{\partial f^\ell_r}{\partial x(s)}(\vz^{t}_{\ell\to i})$]. This tree
clearly belongs to $\cB^{t+1}_i$.

We now treat the terms in the first line.
Again, we have
\[
f^i_s\bigl(\vx^{t-1}_{i}\bigr)
\frac{\partial f^\ell_r}{\partial x(s)}\bigl(\vx ^{t}_{\ell}\bigr)=f^i_s
\bigl(\vz^{t-1}_{i\to\ell}\bigr)\frac{\partial f^\ell
_r}{\partial x(s)}\bigl(
\vz^{t}_{\ell\to i}\bigr)+g\bigl(\vd^{t-1}_{i,\ell},
\vd ^t_{\ell,i},\vz^{t-1}_{i\to\ell},
\vz^{t}_{\ell\to i}\bigr),
\]
where $g$ is a polynomial with either a positive power of a component
of $\vd^{t-1}_{i,\ell}$ or of $\vd^t_{\ell,i}$.
Hence, we only need to construct trees in $\cB^{t+1}_i(r)$
corresponding to terms of the following form: for $\sum_s(
a_s+b_s)\geq1$,
\[
A_{\ell i}^2\prod_s \bigl(
d^{t-1}_{i,\ell}(s) \bigr)^{a_s} \bigl(
d^{t}_{\ell,i}(s) \bigr)^{b_s} \bigl(z^{t-1}_{i\to\ell}(s)
\bigr)^{c_s} \bigl(z^{t}_{\ell\to i}(s)
\bigr)^{d_s}.
\]
Let us first consider the term $A_{\ell i}^2\prod_s
(z^{t-1}_{i\to
\ell}(s) )^{c_s} (z^{t}_{\ell\to i}(s) )^{d_s}$.
It can be interpreted as a sum on the following family of trees: the
type of the root is $i$, and the root has one child with type $\ell$.
This child has $d_s$ subtrees in $\cU^{t}_{\ell\to i}(s)$ and one
child denoted $u$ with type $i$. This child $u$ has $c_s$ subtrees in
$\cU^{t-1}_{i\to\ell}(s)$. Note that the only backtracking path in
such a tree is the path from $u$ to the root with types $i,\ell, i$. In
particular such a tree does not belong to $\cB^t_i(r)$.

We assume now that there exists $s$ with $a_s\geq1$. We need to
interpret the multiplication by
$d^{t-1}_{i,\ell}(s)=z^{t-1}_{i,\ell}(s)+e^{t-1}_i(s)$. First consider
the case of $e^{t-1}_i(s)$, this corresponds to add a subtree in
$\cB^{t-1}_i$ to the vertex $u$. As in previous analysis, we clearly
obtain a tree in $\cB^{t+1}_i$. The term $z^{t-1}_{i,\ell}(s)$
corresponds to adding a child of type $\ell$ to the vertex $u$ which
is the root of a subtree in $\cU^{t-2}_{\ell\to i}(s)$, in particular
we introduce a backtracking path of length $3$ so that again the
resulting tree is in $\cB^{t+1}_i$.
Similarly if $b_s\geq1$, the multiplication by $d^{t}_{\ell,i}(s)$
will correspond to add a subtree to the child of the root, resulting in
either adding a backtracking path of length $3$ or adding a
backtracking star.

The last term of the form
\[
A_{\ell i}\prod_{s=1}^q
\frac{ (d^t_{\ell,i}(s)
)^{n_s}}{n_s!}\frac{\partial^{n_1+\cdots+n_q}f^\ell_r}{\partial
x(1)^{n_1}\cdots\partial x(q)^{n_q}}\bigl(\vz^t_{\ell\to i}\bigr),
\]
with $n_1+\cdots+n_q\geq2$ can be analyzed by the same kind of
argument by noticing that the factor $A_{i\ell}z^t_{\ell,i}(s)z^t_{\ell,i}(s')$ corresponds to a backtracking star.
\end{pf}

The proof of Proposition~\ref{prop:amp} follows from the same
arguments as in the proof of Proposition~\ref{prop:mom-equ}. Once
more, for simplicity, we only consider the case $m(r)=m$ and $m(s)=0$
for $s\neq r$, the general case of $\vm=
(m(1),m(2),\dots,m(q))\in\naturals^q$ being completely analogous.
We represent both moments $\E[x_i^t(r)^m]$ and $\E[z_i^t(r)^m]$ using
Lemma~\ref{lem:treerep} [of the form given in equations (\ref{eq:tree1BIS}),
(\ref{eq:tree2BIS})] and Lemma~\ref{lem:treeAMP}.
The expectation $\E[x_i^t(r)^m]$ is represented as a sum over trees
$T_1,\dots, T_m\in\cU_i^t(r)\cup\cB_i^t(r)$, while $\E
[z_i^t(r)^m]$ is
given by a sum over trees $T_1,\dots, T_m\in\cU_i^t(r)$.
In order to complete the proof we need to show that the contribution
of terms that have at least one tree in $\cB_i^t(r)$ vanishes as
$N\to\infty$.

The factor $\prod_{\ell=1}^m \widetilde{\Gamma}(T_\ell,t,r)$ is
bounded by
$K(d,C,t)^m$, which is independent of $N$.
Hence, we only need to prove that
%
\begin{equation}
\label{eq:backt} \qquad\sum_{T_1\in\cB^t_i(r)}\sum
_{T_j\in\cT^t_{i}(r_j)\cup\cB^t_{i}(r_j),
j\in[2,m]}\E \Biggl[ \prod_{\ell=1}^m
A(T_\ell)x(T_\ell) \Biggr]=O \bigl(N^{-\fraca{1}{2}}
\bigr).
\end{equation}
This statement directly follows from previous analysis, since in the
graph $\vG$ obtained by taking the union of the $T_{\ell}$'s and
identifying vertices $v$ with the same type $\ell(v)$, there is at
least one edge with multiplicity $3$, due to the backtracking path of
length $3$ or the backtracking star in $T_1$. So that previous analysis
shows that the term in (\ref{eq:backt}) is of order $O (N^{-\fraca
{1}{2}} )$.

%

%
\subsection{Proof of Theorem \texorpdfstring{\protect\ref{thm:Universality}}{3}}
\label{sec:ProofUniversalityPolynom}

Let $\{p_{N,i}\}_{N\ge0, 1\le i\le N}$ be a collection of
multivariate polynomials $p_{N,i}\dvtx \reals^q\to\reals$ with degrees
bounded by $D$, and coefficients bounded in magnitude by $B$,
%
\begin{equation}
p_{N,i}(\vx) = \sum_{m(1)+\cdots+m(q)\le D}c_{m(1),\dots,m(q)}^{N,i}
x(1)^{m(1)}\cdots x(q)^{m(q)}. %
\end{equation}
By Propositions \ref{prop:mom-equ} and \ref{prop:amp}, we have
%
\begin{eqnarray}
\hspace*{10pt}\bigl\llvert \E p_{N,i}\bigl(\vx^t_i
\bigr) - \E p_{N,i}\bigl(\tvx^t_i\bigr)\bigr
\rrvert &\le&\sum_{m(1)+\cdots+m(q)\le D} \bigl\llvert
c^{N,i}_{m(1),\dots,m(q)}\bigr\rrvert \bigl\llvert \E\bigl[\bigl(
\vx^t_i\bigr)^\vm\bigr] - \E\bigl[\bigl(
\tvx^t_i\bigr)^{\vm}\bigr]\bigr\rrvert\nonumber\\[-8pt]\\[-8pt]
\hspace*{10pt}& \le& K
D^qB N^{1/2} \nonumber%
\end{eqnarray}
whence the thesis follows.
%
\subsection{Proof of Theorem \texorpdfstring{\protect\ref{thm:PolySE}}{5}}
\label{sec:ProofPolySE}

An important simplification is provided by the following.

\begin{remark}\label{remark:T=S}
It is sufficient to prove Theorem~\ref{thm:PolySE} for $t=s$.
(Hence, Theorem~\ref{thm:SE} implies Theorem~\ref{thm:PolySE}.)
\end{remark}

\begin{pf}
Indeed, consider a converging sequence $\{(A(N),\cF_N,x^{0,N})\}_{N\ge
1}$, and
fix $h=t-s>0$. For the sake of simplicity, and in view of Remark~\ref{remark:NonRandom}, we can assume $\cF_N$ to be given by
the polynomial function $g\dvtx \reals^q\times\reals^{\tq}\times
[k]\times\naturals\to\reals^q,
(\vx,Y,a,t)\mapsto g(\vx,Y,a,t)$ that does
not depend on the random variable $Y$. With an abuse of notation we
will write $g(\vx,a,t)$ in place of $g(\vx,Y,a,t)$.

We will construct a new converging sequence of instances
$\{(A(N),\tcF_N,\allowbreak\tx^{0,N})\}_{N\ge1}$ with variables
$\tvx_i^t\in\reals^{2q}$
and such that, letting $\tvx_i^t = (\vu_i^t, \vv^t_i)$,
$\vu_i^t,\break \vv^t_i\in\reals^q$, the pair $(\vu_i^t,\vv^t_i)$ is
distributed as $(\vx^t_i,\vx^{t-h}_i)$ asymptotically as $N\to\infty$.

The new sequence of initial conditions is constructed as follows:
\begin{longlist}[(3)]
\item[(1)] The initial condition is given by $\tvx^0_i =
(0,0)$.
\item[(2)] The independent randomness is given by $Y(i)=\vx^0_i$.
Notice that, for $i\in C_a^N$, we have $Y(i)\sim_{\mathrm{i.i.d.}} Q_a$, and
hence we let $P_a=Q_a$.
\item[(3)] The partitions $C^N_a$, $a\in[k]$ and matrices $A(N)$ are kept
unchanged.
\item[(4)] The collection of functions in $\tcF_N$
is determined by the polynomial function $\tg\dvtx \reals^{2q}\times
\reals^{\tq}\times[k]\times\naturals\to\reals^{2q}$,
$(\tvx,Y,a,t)\mapsto\tg(\tvx,Y,a,t)$. Writing
$\tg( \cdot)= [\tg^{(1)}( \cdot),\tg^{(2)}( \cdot)]$, with
$\tg^{(1)}( \cdot), \tg^{(2)}( \cdot)\in\reals^q$, we let, for
$\vu,\vv\in\reals^q$.
%
\begin{eqnarray}
g^{(1)} \bigl((\vu,\vv),Y,a,t \bigr) & = & \cases{
g(Y,a,t) &\quad if $t=0$,
\cr
g(\vu,a,t)&\quad if $t>0$,}
\\
g^{(2)} \bigl((\vu,\vv),Y,a,t \bigr) & = & \cases{
g(Y,a,t) &\quad if $t\le h$,
\cr
g(\vv,a,t)&\quad if $t>h$. }
\end{eqnarray}
\end{longlist}
As a consequence of this construction, $\vu^t_i = \vx^t_i$ for all
$i\in[N]$, $t\ge1$, and $\vv^t_i=\vx^{t-h}_i$ for all $t\ge h+1$.
This completes the reduction.
\end{pf}

As a consequence of this remark, it is sufficient to prove Theorem~\ref{thm:SE}, and by Remark~\ref{remark:NonRandom}, we can limit ourselves to the case in which
$g\dvtx (\vx,Y,a,t)\mapsto g(\vx,Y,a,t)$
does not depend on $Y$, and hence this argument will be dropped.
We begin by considering the expectation of moments of $\vx_i^t$.

\begin{proposition}\label{prop:SE}
Let $(A(N),\cF_N,x^0)_{N\ge0}$ be a polynomial and converging sequence
of AMP
instances, and denote by $\{x^t\}_{t\ge0}$ the corresponding AMP
orbit.
Then we have for any $i=i(N)\in C^N_a$, $t\geq1$, $\vm= (m(1),\dots,
m(q))\in\naturals^q$,
\[
\lim_{N\to\infty}\E \bigl[ \bigl(\vx_{i}^t
\bigr)^{\vm} \bigr] = \E \bigl[ \bigl(Z_{a}^t
\bigr)^{\vm} \bigr] ,
\]
where $Z^t_a\sim\normal (0, \Sigma^t_a  )$.
\end{proposition}

\begin{pf}
By Propositions \ref{prop:mom-equ-iid} and \ref{prop:amp}, we need
only to prove the statement for the AMP orbit $y^t$.
We will indeed prove by induction on $t$ that for any $i\in C^N_a$ and
any $j\neq i$,
%
\begin{eqnarray}
\label{eq:momyZ} \lim_{N\to\infty}\E \bigl[ \bigl(
\vy_{i\to
j}^t \bigr)^{\vm} \bigr] & = &\E \bigl[
\bigl(Z_{a}^t \bigr)^{\vm} \bigr] ,
\\
\label{eq:probZ} \lim_{N\to\infty}\frac{1}{\llvert C_a^N\rrvert }\sum
_{i\in C_a^N} \bigl(\vy _{i\to
j}^t
\bigr)^{\vm} & = &\E \bigl[ \bigl(Z_{a}^t
\bigr)^{\vm} \bigr]\qquad \mbox{in probability}. %
\end{eqnarray}
For $t\geq1$, let $\mathfrak{F}_t$ be the $\sigma$-algebra generated
by $A^0,\dots,A^{t-1}$.
We will show, using the central limit theorem, that the random vector
$(y_{i \to j}^{t+1}(1),\dots, y_{i \to j}^{t+1}(q))$ given $\mathfrak
{F}_t$ converge in distribution to a centered Gaussian random vector.
More precisely,
by (\ref{eq:rec-iid}) and the induction hypothesis, the following limit
holds in probability:
\begin{eqnarray*}
&&\lim_{N\to\infty}\E \bigl[ y_{i \to j}^{t+1}(r)y_{i \to
j}^{t+1}(s)
|\mathfrak{F}_t \bigr] \\
& &\qquad=\lim_{N\to\infty} \mathop{\sum_{\ell\in[N]\ex j}}_{\ell\in C_b^N}\E \bigl[ \bigl(A_{\ell
i}^t
\bigr)^2 \bigr]g_r\bigl(\vy^t_{\ell\to
i},b,t
\bigr)g_s\bigl(\vy^t_{\ell\to i},b,t\bigr)
\\
&&\qquad= \sum_{b=1}^k
c_b W_{ab} \E \bigl[ g_r
\bigl(Z^t_{b},b,t\bigr)g_s
\bigl(Z^t_{b},b,t\bigr) \bigr] = \Sigma
_a^{t+1}(r,s). %
\end{eqnarray*}
Since for all $r\in[q]$ from (\ref{eq:rec-iid}) we have $\E[ y_{i
\to
j}^{t+1}(r)]=0$,
from the central limit theorem, it follows that $\vy_{i\to j}^{t+1}$
converges to a centered Gaussian vector with covariance
$\Sigma_a^{t+1}$. Since all the moments of $\vy_{i\to j}^{t+1}$ are
bounded uniformly in $N$ by Proposition~\ref{prop:mom-equ-iid} and
Lemma~\ref{lem:tree}, the induction claim, equation (\ref{eq:momyZ})
follows, for iteration $t+1$.

In the base case $t=0$ the same conclusion holds because
\begin{eqnarray*}
\lim_{N\to\infty}\E \bigl[ y_{i \to j}^{1}(r)y_{i \to
j}^{1}(s)
\bigr] & =&\lim_{N\to\infty} \mathop{\sum_{\ell\in[N]\ex j}}_{\ell\in C_b^N}
\E \bigl[ \bigl(A_{\ell
i}^0\bigr)^2
\bigr]g_r\bigl(\vy^0_{\ell\to
i},b,0
\bigr)g_s\bigl(\vy^0_{\ell\to i},b,0\bigr)
\\[-3pt]
&=& \sum_{b=1}^k
c_b W_{ab}\hSigma_b^0(r,s),
\end{eqnarray*}
where the second identity holds by assumption.

Next consider the induction claim, equation (\ref{eq:probZ}).
Recall the representation introduced in Section~\ref{sec:ProofMomIID},
\begin{eqnarray*}
y^{t}_{i\to j}(r) &=& \sum_{T\in\cT^t_{i\to j}(r)}
\overline {A}(T,t)\Gamma(T,\vc,t)x(T),
\\[-2pt]
\overline{A}(T,t) &=& \prod_{(u\to v)\in E(T)}A^{t-\llvert u\rrvert }_{\ell(u)
\ell(v)}.
\end{eqnarray*}
Using this representation of $\vy^{t}_{i\to j}$, $\vy^{t}_{k\to j}$
it is easy to show that, for $i\neq k$, $i,k\in C_a^N$\vspace*{-1pt}
%
\begin{equation}\label{eq:BoundCorrelation}
\bigl\llvert \E \bigl[ \bigl(\vy_{i\to j}^t
\bigr)^{\vm} \bigl(\vy_{k\to
j}^t
\bigr)^{\vm} \bigr]- \E \bigl[ \bigl(\vy_{i\to j}^t
\bigr)^{\vm} \bigr]\E \bigl[ \bigl(\vy_{k\to
j}^t
\bigr)^{\vm} \bigr] \bigr\rrvert \le\eps(N),%
\end{equation}
for some function $\eps(N)\to0$ as $N\to\infty$ ar $\vm,C,d,t$ fixed.
Indeed, the above expectations can be represented as sums over
$m=m(1)+m(2)+\cdots+m(q)$ trees $T_1,\dots,T_m\in\cT^t_{i\to j}$ and
$m$ trees
$T'_1,\dots,T'_m\in\cT^t_{k\to j}$. Let $\vG$ be the simple graph
obtained by
identifying vertices of the same type in $T_1,\dots,T_m,T'_1,\dots,\break T'_m$.

By Lemma~\ref{lem:tree} and the argument in
the proof of Proposition~\ref{prop:mom-equ}, all the terms in which
$\vG$ has cycles, or an edge of $\vG$ correspond to more than 2 edges
in the union of $T_1,\dots,T_m$, $T'_1,\dots,T'_m$ add up to a
vanishing contribution in the $N\to\infty$ limit. Further, all the
terms in which $\vG$ is the union of two disconnected components (one
containing $i$, and the other containing $k$) are
identical in $\E [  (\vy_{i\to j}^t )^{\vm} (\vy
_{k\to
j}^t )^{\vm} ]$ and $\E [  (\vy_{i\to j}^t
)^{\vm
} ]\E [ (\vy_{k\to
j}^t )^{\vm} ]$ and hence cancel out.
We are therefore left with the sum over trees
$T_1,\dots,T_m,T'_1,\dots,T'_m$ such that $\vG$ is itself a connected
tree, with edges covered exactly twice. Assume, to be definite, that
$\vG$ has $\mu$ vertices and hence $\mu-1$ edges. The weight of such
a term
is bounded by
\[
K\E \Biggl\{\prod_{i=1}^m
\overline{A}(T_i,t) \prod_{i=1}^m
\overline{A}\bigl(T'_i,t\bigr) \Biggr\}\le
KN^{-\mu+1}. %
\]
On the other hand, the number of such terms is bounded by
$K N^{\mu-2}$ (because the type has to be assigned to $\mu$ vertices,
but $2$ of these are fixed to $i$ and $k$), and hence the overall
contribution of these terms vanishes as well.

From equation (\ref{eq:BoundCorrelation}), and using the fact that
$\E[(\vy_{i\to j}^t)^{2\vm}]\le K$ (because of Lemma~\ref{lem:tree}
and Proposition~\ref{prop:mom-equ-iid}), we have
\begin{eqnarray*}
\hspace*{-4pt}&&\lim_{N\to\infty}\Var \biggl\{\frac{1}{\llvert C_a^N\rrvert }\sum
_{i\in C_a^N}  \bigl(\vy _{i\to
j}^t
\bigr)^{\vm} \biggr\}
\\
\hspace*{-4pt}&&\qquad\le \lim_{N\to\infty}\frac{1}{\llvert C_a^N\rrvert ^2}\sum
_{i,k\in C_a^N} \bigl\llvert \E \bigl[ \bigl(\vy_{i\to j}^t
\bigr)^{\vm} \bigl(\vy_{k\to
j}^t
\bigr)^{\vm} \bigr]- \E \bigl[ \bigl(\vy_{i\to j}^t
\bigr)^{\vm} \bigr]\E \bigl[ \bigl(\vy_{k\to
j}^t
\bigr)^{\vm} \bigr] \bigr\rrvert = 0 . %
\end{eqnarray*}
Equation (\ref{eq:probZ}) follows for iteration $t+1$ by applying Chebyshev's
inequality to the sequence
\[
\biggl\{\frac{1}{\llvert C_a^N\rrvert }\sum_{i\in C_a^N} \bigl(
\vy_{i\to
j}^t\bigr)^{\vm} \biggr\}_{N\ge0},
\]
and using \eqref{eq:momyZ}.
\end{pf}

We are now ready to prove Theorem~\ref{thm:PolySE} in the case in
which $\psi\dvtx \reals^q\to\reals$ is a polynomial.

\begin{proposition}\label{propo:InProbMoment}
Let $(A(N),\cF_N,x^0)_{N\ge0}$ be a polynomial and converging sequence
of AMP
instances, and denote by $\{x^t\}_{t\ge0}$ the corresponding AMP
orbit.
Then we have for any $t\geq1$, $\vm= (m(1),\dots, m(q))\in\naturals^q$,
%
\begin{equation}\label{eq:VarianceClaim}
\lim_{N\to\infty}\Var \biggl\{\frac{1}{\llvert C_a^N\rrvert }\sum
_{i\in C_a^N} \bigl(\vx_{i}^t
\bigr)^{\vm} \biggr\} = 0 . %
\end{equation}
%
\end{proposition}

\begin{pf}
In order to prove (\ref{eq:VarianceClaim}), we fix $t\ge1$ and $a\in
[k]$, and construct a
modified sequence of AMP instances as follows. The new sequence has
$N'=2N$ and $k' = k+1$. The new partition of the variable indices $\{
1,\dots,N\}$ is the
same as in the original instances, with the addition of
$C_{k+1}^N=\{N+1,\dots,2N=N'\}$.
Further we set, for $\varphi\dvtx \reals^q\to\reals$ a polynomial:
\begin{longlist}[(2)]
\item[(1)] for $i,j\le N$: $A'_{ij} = A_{ij}$ and when $i> N$ or $j>N$
define $A_{ij}'\sim\normal(0,1/N)$ independently;
\item[(2)] $g'(\vx,b,t')= g(\vx,b,t')$ for $b\in[k]$, $t'\le t-1$;
$g'(\vx,b,t) = 0$ for $b\in[k]\setminus a$;
$g'_1(\vx,a,t) = \varphi(\vx)$, $g'_r(\vx,a,t) = 0$, for $r\ge
2$; $g'(\vx,k+1,t')=0$ for all
$t'$.

The definition of $g'(\vx,a,t')$ for $t'>t$ is irrelevant for our purposes.
\end{longlist}
Since $g'(\vx,k+1,t')=0$ for all $t'$, the orbit $(\vx_i^{t'}\dvtx  i\le
N, t'\le t)$ is not affected by the new variables.
Further, by the general AMP equation (\ref{eq:AMPGeneralDef_bis}), we
have, for $i\in C_{k+1}^N$,
%
\begin{equation}
x_i^{t+1}(1) = \sum
_{j\in C_a^N} A_{ij}\varphi\bigl(\vx_j^t
\bigr). %
\end{equation}
Notice that the $\{A_{ij}\}_{j\in C_a^N}$ in this equation are
independent of $\vx_j^t$. Hence
%
\begin{eqnarray}
\E\bigl\{x_i^{t+1}(1)^4\bigr\}
&=& \sum_{j_1,\dots,j_4\in C_a^N} \E\{ A_{ij_1}A_{ij_2}A_{ij_3}A_{ij_4}
\}\nonumber\\[-8pt]\\[-8pt]
&&\hspace*{45pt}{}\times \E\bigl\{\varphi\bigl(\vx_{j_1}^t\bigr) \varphi\bigl(
\vx_{j_2}^t\bigr) \varphi\bigl(\vx_{j_3}^t
\bigr) \varphi\bigl(\vx_{j_4}^t\bigr)\bigr\}\nonumber
\\\label{eq:FourthMoment}
&=& \frac{3}{N^2}\sum_{j_1,j_2\in C_a^N} \E\bigl\{\varphi
\bigl(\vx_{j_1}^t\bigr)^2 \varphi\bigl(
\vx_{j_2}^t\bigr)^2\bigr\}.
\end{eqnarray}
On the other hand, using Proposition~\ref{prop:SE} (once for
iteration $t+1$ and $i\in C_{k+1}^N$, and another time for iteration $t$
and $i\in C_{a}^N$) we get
%
\begin{eqnarray}
\lim_{N\to\infty} \E\bigl\{x_i^{t+1}(1)^4
\bigr\} &=& \E\bigl\{ \bigl(Z_{k+1}^{t+1}(1)
\bigr)^4\bigr\}
= 3\bigl(\Sigma_{k+1}^{t+1}(1,1)
\bigr)^2\nonumber\\[-8pt]\\[-8pt]
 &=& 3c^2_a\E\bigl\{\varphi
\bigl(Z^t_a\bigr)^2\bigr\}^2 ,\qquad i
\in C_{k+1}^N,\nonumber
\\
\lim_{N\to\infty} \E\bigl\{\varphi\bigl(\vx_i^t
\bigr)^2\bigr\} &=& \E\bigl\{\varphi \bigl(Z^t_a
\bigr)^2\bigr\},\qquad i\in C_{a}^N, %
\end{eqnarray}
where $Z^t_a\sim\normal(0,\Sigma^t_a)$.
Comparing these equations with equation (\ref{eq:FourthMoment}) we conclude
that
%
\begin{equation}
\hspace*{25pt}\lim_{N\to\infty}\frac{1}{N^2}\sum
_{j_1,j_2\in C_a^N} \E\bigl\{\varphi \bigl(\vx _{j_1}^t
\bigr)^2 \varphi\bigl(\vx_{j_2}^t
\bigr)^2\bigr\} = \biggl\{ \lim_{N\to\infty}
\frac{1}{N}\sum_{j\in C_a^N} \E\bigl[\varphi\bigl(
\vx_{j}^t\bigr)^2\bigr] \biggr\}^2
. %
\end{equation}
Equivalently,
%
\begin{equation}
\lim_{N\to\infty}\Var \biggl\{\frac{1}{\llvert C_a^N\rrvert }\sum
_{i\in C_a^N} \varphi\bigl(\vx_{i}^t
\bigr)^2 \biggr\} = 0 . %
\end{equation}
Taking $\varphi(\vx)=\vx^{\vk}$, we obtain equation (\ref
{eq:VarianceClaim})
for $\vm$ even.
In order to establish equation (\ref{eq:VarianceClaim}) for general
$\vm$
we take, for instance, $\varphi(\vx) = 1+\eps\vx^\vm$ and use the
fact that the
limit must vanish for all $\eps$.
\end{pf}

At this point we can prove Theorem~\ref{thm:PolySE}.

\begin{pf*}{Proof of Theorem~\ref{thm:PolySE}}
By Remark~\ref{remark:NonRandom} and Remark~\ref{remark:T=S}, we
reduced ourselves to the case $t=s$, and
$Y(i)=0$ [equivalently, $Y(i)$, is absent].

Consider the empirical measure on $\reals^q$ given by
\[
\mu^N_a=\frac{1}{\llvert C_a^{N}\rrvert } \sum
_{i\in C_a^{N}}\delta_{\vx_{i}^t}.
\]
Proposition~\ref{prop:SE} shows the convergence of expected the
moments of
$\mu^N_a$ to moments that determine the Gaussian
distribution. Proposition~\ref{propo:InProbMoment} combined with Chebyshev
inequality implies
\[
\lim_{N\to\infty}\mu^N_a \bigl( \bigl(
\vx_{i}^t \bigr)^{\vm}\bigr) = \E \bigl[
\bigl(Z_{a}^t \bigr)^{\vm} \bigr] ,
\]
in probability.
The proof follows using the relation between convergence in
probability and convergence almost sure along subsequences, together
with the moment method.
\end{pf*}

%
%
\section{Nonsymmetric matrices}\label{sec:NonSymmetric}

In this section we consider a slightly different setting that turns
out to be a special case of the one introduced in Section~\ref
{sec:StateEvolutionResults}.

\begin{definition}
A converging sequence of (polynomial) \emph{bipartite AMP instances}
$\{(A(n),f,h,x^{0,n})\}_{n\ge1}$ is defined by giving for each $n$:
\begin{longlist}[(3)]
\item[(1)] A matrix $A(n)\in\reals^{m\times n}$ with $m=m(n)$ such that
$\lim_{n\to\infty}m(n)/n= \delta>0$. Further, $A(n) =
(A_{ij})_{i\le
m,j\le n}$ is a matrix with the entries $A_{ij}$ independent sub-Gaussian
random variables with common scale factor $C/n$ and first two
moments $\E\{A_{ij}\}=0$, $\E\{A_{ij}^2\} = 1/m$.
\item[(2)] Two functions $f\dvtx  \reals^{q}\times\reals^{\tq}\times
\naturals
\to\reals^q$ and $h\dvtx  \reals^{q}\times\reals^{\tq}\times\naturals
\to\reals^q$ such that, for each $t\ge0$, $f( \cdot, \cdot,t)$ and $h( \cdot, \cdot,t)$ are polynomials.
\item[(3)] An initial condition $x^{0,n} = (\vx_1^0,\dots,\vx
_n^0)\in
\mathcal{V}_{q,n}\simeq(\reals^q)^{n}$,
with $\vx_i^0\in\reals^q$, such that, in probability,
%
\begin{eqnarray}
  \sum_{i=1}^n\exp\bigl\{
\bigl\llVert \vx_i^{0,n}\bigr\rrVert _2^2/C
\bigr\}&\le& nC ,
\\\label{eq:BipartiteInitial}
  \lim_{n\to\infty} \frac{1}{m(n)}\sum
_{i=1}^nf\bigl(\vx^0_i,Y(i),0
\bigr) f\bigl(\vx ^0_i,Y(i),0\bigr)^{\sT} &=&
\Xi^0 .%
\end{eqnarray}
\item[(4)] Two collections of i.i.d. random variables $(Y(i), i\in[n])$
and $(W(j), j\in[m])$ with $Y(i)\sim_{\mathrm{i.i.d.}} Q$ and
$W(j)\sim_{\mathrm{i.i.d.}} P$.
Here $Q$ and $P$ are finite mixture of Gaussians on $\reals^{\tq}$.
\end{longlist}
\end{definition}

Throughout this section, we will refer to nonbipartite AMP instances
as per
Definition~\ref{def:Converging}, as to \emph{symmetric instances}.
With these ingredients, we define the AMP orbit
as follows.

\begin{definition}
The \emph{approximate message
passing orbit} corresponding to the bipartite instance $(A,f,h,x^0)$
is the sequence of vectors $\{x^t, z^t\}_{t\ge0}$, $x^t\in\mathcal{V}_{q,n}$
$z^t\in\mathcal{V}_{q,m}$
defined as follows, for $t\ge0$,
%
\begin{eqnarray}\label{eq:BipartiteAMP1}
z^{t} &= & A f\bigl(x^t,Y;t\bigr) -
\Ons_t h\bigl(z^{t-1},W;t-1\bigr),
\\
\label{eq:BipartiteAMP2}
x^{t+1} & = &A^{\sT} h\bigl(z^t,W;t\bigr) -
\POns_t f\bigl(x^t,Y;t\bigr),%
\end{eqnarray}
where $f( \cdot)$, $h( \cdot)$ are applied componentwise; see
below for an explicit formulation.
Here $\Ons_t\dvtx  \mathcal{V}_{q,m}\to\mathcal{V}_{q,m}$ is the linear
operator that maps $v$
to $v' = \Ons_t v$,
and for any $j\in[m]$ satisfies,
%
\begin{equation}
\vv'_j = \biggl(\sum_{k\in[n]}A_{jk}^2
\frac{\partial f}{\partial
\vx
}\bigl(\vx^t_k,Y(k);t\bigr) \biggr)
\vv_j.
\end{equation}
Analogously $\POns_t\dvtx  \mathcal{V}_{q,n}\to\mathcal{V}_{q,n}$ is the
linear operator
defined by
letting, for $v' = \POns_t v$, and any $j\in[n]$,
%
\begin{equation}
\vv'_i = \biggl(\sum_{l\in[m]}A_{li}^2
\frac{\partial h}{\partial
\vz
}\bigl(\vz^t_l,W(l);t\bigr) \biggr)
\vv_i.
\end{equation}
\end{definition}

For the sake of clarity, it is useful to rewrite the iteration
(\ref{eq:BipartiteAMP1}), (\ref{eq:BipartiteAMP2}) explicitly, by
components
\begin{eqnarray*}
\vz_i^{t} &= & \sum
_{j\in[n]}A_{ij} f\bigl(\vx_j^t,Y(j);t
\bigr) \\
&&{}- \sum_{k\in
[n]}A_{jk}^2
\frac{\partial f}{\partial\vx}\bigl(\vx^t_k,Y(k);t\bigr) h\bigl(
\vz_i^{t-1},W(i);t-1\bigr)\qquad \mbox{for all } i\in[m],
\\
\vx_j^{t+1} & = &\sum_{i\in[m]}A_{ij}
h\bigl(\vy_i^t,W(i);t\bigr) \\
&&{}- \sum
_{l\in
[m]}A_{lj}^2\frac{\partial h}{\partial\vz}\bigl(
\vz^t_l,W(l);t\bigr) f\bigl(x_j^t,Y(j);t
\bigr)\qquad \mbox{for all } j\in[n]. %
\end{eqnarray*}
We will state and prove a state evolution result that is analogous to
Theorem~\ref{thm:PolySE} for the present case. Since the proof is by
reduction to the symmetric
case, the same argument also implies a universality statement of the
type of Theorem~\ref{thm:Universality}. However, we will not state explicitly any
universality statement in this case.
We begin by introducing the appropriate state evolution recursion.
In analogy with equation (\ref{eq:GeneralSE}), we introduce two sequences
of positive semidefinite matrices $\{\Sigma^t\}_{t\ge0}$,
$\{\Xi^t\}_{t\ge0}$ by letting $\Xi^0$ be given as per
equation (\ref{eq:BipartiteInitial}) and defining, for all
$t\ge1$,
%
\begin{eqnarray}\label{eq:SEBipartite1}
\Sigma^{t} & =& \E \bigl\{ h\bigl(Z^{t-1},W,t-1
\bigr) h\bigl(Z^{t-1},W,t-1\bigr)^{\sT
} \bigr\},\nonumber\\[-8pt]\\[-8pt]
Z^{t-1}&\sim&\normal\bigl(0,\Xi^{t-1}\bigr),\qquad W\sim P
,\nonumber
\\
\label{eq:SEBipartite2}
\Xi^{t} &=&\frac{1}{\delta} \E \bigl\{ f\bigl(X^t,Y,t
\bigr) f\bigl(X^t,Y,t\bigr)^{\sT
} \bigr\},\nonumber\\[-8pt]\\[-8pt]
 X^t
&\sim&\normal\bigl(0,\Sigma^t\bigr), \qquad Y\sim Q .\nonumber
\end{eqnarray}
We also define a two-times recursion analogous to
equations (\ref{eq:GeneralstSE}), (\ref{eq:GeneralstSE2}).
Namely, we introduce the boundary condition
%
\begin{eqnarray}
\hspace*{25pt}\Xi^{0,0} = \pmatrix{ \Xi^0 & \Xi^0
\cr
\Xi^0 & \Xi^0 },\qquad \Xi^{t,0} = \pmatrix{
\Xi^t & 0
\cr
0& \Xi^0 }, \qquad\Xi^{0,t} = \pmatrix{
\Xi^0 & 0
\cr
0& \Xi^t }, %
\end{eqnarray}
with $\Xi^t$ defined per equations (\ref{eq:SEBipartite1}), (\ref
{eq:SEBipartite2}).
For any $s,t\ge1$, we set recursively
%
\begin{eqnarray}
\Sigma^{t,s}&=&\E \bigl\{ \cZ_{t-1,s-1}
\cZ_{t-1,s-1}^{\sT} \bigr\},
\\
 \cZ_{t-1,s-1}&\equiv&\bigl[h\bigl(Z^{t-1},W,t-1\bigr),h
\bigl(Z^{s-1},W,s-1\bigr)\bigr] ,
\\
\Xi^{t,s}&=&\E \bigl\{ \cX_{t,s} \cX_{t,s}^{\sT}
\bigr\},
\\
 \cX_{t,s}&\equiv&\bigl[f\bigl(X^{t},Y,t\bigr),f
\bigl(X^{s},Y,s\bigr)\bigr] . %
\end{eqnarray}
(Recall that $[u,v]$ denotes the column vector obtained by
concatenating $u$ and $v$.)

\begin{theorem}\label{thm:BipartiteSE}
Let $\{(A(n),f,h,x^{0,n})\}_{n\ge1}$ be a polynomial and converging
sequence of \emph{bipartite AMP instances},
and denote by $\{x^t,z^t\}_{t\ge0}$ the corresponding AMP
orbit.

Fix $s,t\ge1$. If $s\neq t$, further assume that the initial condition
$x^{0,n}$ is
obtained by letting $\vx_i^{0,n}\sim R$ independent and
identically distributed, with $R$ a finite mixture of Gaussians.
Then, for each locally Lipschitz function
$\psi\dvtx \reals^q\times\reals^q\times\reals^{\tq}\to\reals$ such that
$\llvert \psi(\vx,\vx',y)\rrvert \le K(1+\llVert y\rrVert _2^2+\llVert \vx\rrVert _2^2+\llVert \vx'\rrVert _2^2)^K$, we
have, in
probability,
%
\begin{eqnarray}
\lim_{n\to\infty} \frac{1}{n}\sum
_{j\in[n]} \psi\bigl(\vx^t_j,
\vx^s_j,Y(j)\bigr) &=& \E \bigl[ \psi\bigl(X^t,X^s,Y
\bigr) \bigr],
\\
\lim_{N\to\infty} \frac{1}{m(n)}\sum
_{j\in[m]} \psi\bigl(\vz^t_j,
\vz^s_j,W(j)\bigr) &=& \E \bigl[ \psi\bigl(Z^t,Z^s,W
\bigr) \bigr], %
\end{eqnarray}
where $(X^t,X^s)\sim\normal(0,\Sigma^{t,s})$ is independent of
$Y\sim Q$, and $(Z^t,Z^s)\sim\break\normal(0,\Xi^{t,s})$ is independent of
$W\sim P$.
\end{theorem}

\begin{pf}
The proof follows by constructing a suitable polynomial and
converging sequence of symmetric instances, recognizing that a suitable
subset of the resulting orbit corresponds to the orbit $\{x^t,z^t\}$
of interest, and applying Theorem~\ref{thm:PolySE}.

Specifically, given a converging sequence of bipartite instances
$(A(n),f,h,\allowbreak x^{0,n})$, we
construct a symmetric instance $(A_s(N),g,x_s^{0,N})$ with
(below we use the subscript $s$ to refer to the symmetric instance):
\begin{longlist}[(6)]
\item[(1)] The symmetric instance has dimensions
$N=n+m$ and $q_s=q$, $\tq_s=\tq$.
\item[(2)] We partition the index set in $k=2$ subsets: $[N] =
C^N_1\cup
C_2^N$, with $C^N_1=\{1,\dots,m\}$ and $C^N_2 =\{m+1,\dots,m+n\}$.
In particular $c_1= \delta/(1+\delta)$ and $c_2=1/(1+\delta)$.
\item[(3)] The symmetric random matrix $A'$ is given by
\begin{eqnarray*}
A_s &=& \lleft( %
\matrix{ 0&A
\cr
A^{\sT}&0 } %
\rright).
\end{eqnarray*}
In particular $W_{11} =W_{22}=0$ and $W_{12}= W_{21} =
(1+\delta)/\delta$.
\item[(4)] The vertex labels are $Y_s(i) = W(i)$ for $i\le m$ and
$Y_s(i) =
Y(i-m)$ for $i>m$. In particular, these are independent random
variables with distribution $Y_s(i)\sim P_1= Q$ if $i\in C_1^N$ and
$Y_s(i)\sim P_2= P$ if $i\in C_2^N$.
\item[(5)] The initial condition is given by $\vx_{s,i}^{0,N} =0$ for
$i\in C_1^N$ and $\vx_{s,i}^{0,N} = \vx^{0,n}_{i-m}$ for
$i\in C_2^N$.
\item[(6)] Finally, for any $\vx\in\reals^q$, $Y\in\reals^{\tq
}$, $t\ge0$,
we let
%
\begin{eqnarray}
g(\vx,Y,a=1,2t) &=& f(\vx,Y,t),
\\
g(\vx,Y,a=2,2t+1)& = &h(\vx,Y,t). %
\end{eqnarray}
The definition of $g(\vx,Y,a=1,2t+1)$ and $g(\vx,Y,a=2,2t)$ is
irrelevant for our purposes.
\end{longlist}
The proof is concluded by recognizing that, for all $t\ge0$,
\begin{eqnarray*}
\vx^{2t+1}_{s,i}& =& \vz^t_{i},\qquad
\mbox{for $i\in C_1^N$},
\\
\vx^{2t}_{s,i}& =& \vx^t_{i-m},\qquad
\mbox{for $i\in C_2^N$}. %
\end{eqnarray*}
\upqed\end{pf}

We finish this section with a lemma that establishes continuity of
the AMP trajectories with respect to Gaussian perturbations of the
matrix $A$.
This fact will be used in the next section.
(Notice that an analogous lemma holds by the same argument for
converging, nonbipartite,
instances.)

\begin{lemma}\label{lemma:Continuity}
Let $\{(A(n),f,h,x^{0,n})\}_{n\ge1}$ be a polynomial converging
sequence of \emph{bipartite AMP instances} and denote by
$\{x^t,z^t\}_{t\ge0}$ the corresponding AMP orbit. For each $n$, let
$G(n)\in\reals^{m(n)\times n}$ be a random matrix with i.i.d. entries
$G(n)_{ij}\sim\normal(0,1/m(n))$, independent of $A(n)$. Consider the
perturbed sequence
$\{(\tA(n) = A(n)+\nu G(n),f,h,x^{0,n})\}_{n\ge1}$, with
$\nu\in\reals^+$, and denote by
$\{\tx^t,\tilde{z}^t\}_{t\ge0}$ the corresponding AMP orbit.
Then for any $t$ there exists a constant $K$ independent of $n$ such
that
\[
\E\bigl\{\bigl\llVert \vx_i^t-
\tvx_i^t\bigr\rrVert ^2_2\bigr\}
\le K \bigl(\nu^2 + n^{-1/2} \bigr),\qquad \E\bigl\{\bigl\llVert
\vz_i^t-\tvz_i^t\bigr\rrVert
^2_2\bigr\}\le K \bigl(\nu^2 +
n^{-1/2} \bigr). %
\]
\end{lemma}

\begin{pf}
Consider the difference $[\vx_i^t(r)-\tvx_i^t(r)]$. By the tree
representation in Section~\ref{sec:TreeRep} and Lemma~\ref{lem:treeAMP}, this difference can be written as a polynomial in
$A$ and $G$ whereby each monomial has the form
%
\begin{equation}
\Gamma(T,t)x(T) \biggl\{\prod_{(u\to v)\in E(T)}
\tA_{\ell(u)\ell(v)}-\prod_{(u\to v)\in E(T)} A_{\ell(u)\ell(v)}
\biggr\}. %
\end{equation}
Enumerating the edges in $T$ as $(u_1,v_1),\dots, (u_k,v_k)$ the
quantity in parenthesis reads
%
\begin{equation}
\sum_{i=1}^k\prod
_{j=1}^{i-1}A_{\ell(u_j),\ell(v_j)}\cdot\nu
G_{\ell(u_i),\ell(v_i)}\cdot\prod_{j=i+1}^{k}
\tA_{\ell(u_j),\ell
(v_j)}. %
\end{equation}
In other words, the sum over trees $T$ is replaced by a sum over
trees with one distinguished edge, and the edge carries weight
$\nu G_{\ell(u_i),\ell(v_i)}$.
The expectation $\E\{\llVert \vx_i^t-\tvx_i^t\rrVert ^2_2\}$ is given by a sum
over pairs of such marked trees.
Using the fact that the entries of the matrix $\tA(n)$ are still
independent sub-Gaussian with scale factor $C/(n+\nu^2Cm(n))\leq C'/n$,
it is easy to see that the argument in Lemma~\ref{lem:tree} and (\ref
{eq:backt}) are still
valid. Hence,
up to errors bounded by $K n^{-1/2}$ the only terms that contribute
to this sum are those over pair of trees such that the graph $\vG$
obtained by identifying vertices of the same type has only double
edges.
In particular for the distinguished edge, we can use the following
upper bound instead of (\ref{eq:uppsubg}): $\E [ \llvert \nu
G_{ij}\rrvert ^2 ]=\frac{\nu^2}{m(n)}\leq K \frac{\nu^2}{n}$ and this
yields a factor $\nu^2$ [by the same argument as in the proof of Lemma~\ref{lem:tree}
to get (\ref{eq:T(A)})].
%
\end{pf}

%
\section{Proof of universality of polytope neighborliness}\label{sec:Polytope}

In this section we prove Theorem~\ref{thm:Polytope}, deferring several
technical steps to the Appendices.

\begin{hypothesis}\label{hyp1}
 Throughout this section $\{A(n)\}_{n\ge0}$ is a
sequence of random
matrices whereby $A(n)\in\reals^{m\times n}$ has independent entries
that satisfy $\E\{A(n)_{ij}\} = 0$, $\E\{A(n)_{ij}^2\} =1/m$ and are
sub-Gaussian with scale factor $s/m$, with $s$ independent of $m$, $n$.
\end{hypothesis}

Notice that these matrices differ by a
factor $1/\sqrt{m}$ from the matrices in the statement of Theorem~\ref{thm:Polytope}. Since neighborliness is invariant
under scale transformations, this change is immaterial.

The approach we will follow is based on the equivalence between weak
neighborliness and compressed sensing reconstruction developed
in \cite{Donoho2005a,Donoho2005b,DoTa05a,DoTa05b}.
Within compressed sensing, one considers the problem of reconstructing
a vector $x_0\in\reals^n$ from a vector of linear ``observations'' $y=
Ax_0$ with $y\in\reals^m$ and $m\le n$. The measurement matrix
$A\in\reals^{m\times n}$ is assumed to be known. An interesting
approach toward reconstructing $x_0$ from the linear observations
$y$ consists in solving a convex program
%
\begin{equation}\label{eq:Ell1Min}
\hx(y) = \arg\min \bigl\{\llVert x\rrVert _1 \mbox{
such that } x\in\reals^n, y = Ax  \bigr\}.
\end{equation}
Hence one says that $\ell_1$ minimization \emph{succeeds} if the above
$\arg\min$ is uniquely defined and $\hx(y) = x_0$.
Remarkably, this event
only depends on the support of $x_0$, $\supp(x_0) = \{i\in[n]\dvtx
x_{0,i} \neq0\}$ \cite{Donoho2005a}.
This motivates the following
abuse of terminology. We say that, for a given matrix $A$, $\ell_1$
minimization \emph{succeeds} for a fraction $f$ of vectors $x_0$
with\footnote{As customary in this domain, we denote by $\llVert v\rrVert _0$ the
number of nonzero entries in $v\in\reals^q$ (which of course is not
a norm).} $\llVert x_0\rrVert _0\le k$ if it does succeed for at least $f\binom{n}{k}$
choices of $\supp(x_0)$ out of the $\binom{n}{k}$ possible ones.
Analogously, that $\ell_1$
minimization \emph{fails} for a fraction $f$ of vectors $x_0$ if it
does succeed at most for $(1-f)\binom{n}{k}$
choices of $\supp(x_0)$.

Success of $\ell_1$
minimization turns out to be intimately related to the neighborliness
properties of
the polytope $AC^n$.

\begin{theorem}[(Donoho \cite{Donoho2005a})]\label{thm:CSPolytope}
Fix $\delta\in(0,1)$. For each $n\in\naturals$, let $m(n) =
\lfloor n\delta\rfloor$ and $A(n)\in\reals^{m(n)\times n}$ be a random
matrix. Then the sequence\break $\{A(n)C^n\}_{n\ge0}$ has weak
neighborliness $\rho$ in probability if and only if the following
happens:
\begin{longlist}[(2)]
\item[(1)] For any $\rho_-<\rho$, there exists $\eps_n\downarrow0$ such
that, for a fraction larger than $(1-\eps_n)$ of vectors
$x_0$ with $\llVert x_0\rrVert _0 = m(n) \rho_-$
the $\ell_1$ minimization \emph{succeeds}
with high probability [with respect to the choice of the random matrix
$A(n)$].
\item[(2)] Vice versa, for any $\rho_+>\rho$, there exists
$\eps_n\downarrow0$ such that,
for a fraction larger than $(1-\eps_n)$ of vectors
$x_0$ with $\llVert x_0\rrVert _0 = m(n) \rho_+$
the $\ell_1$ minimization \emph{fails}
with high probability [with respect to the choice of the random matrix
$A(n)$].
\end{longlist}
\end{theorem}

This is indeed a rephrasing of Theorem~2 in \cite{Donoho2005a}.

In view of this result, Theorem~\ref{thm:Polytope} follows from the
following result on compressed sensing with random sensing matrices.

\begin{theorem}\label{thm:CS}
Fix $\delta\in(0,1)$. For each $n\in\naturals$, let $m(n) =
\lfloor n\delta\rfloor$ and define $A(n)\in\reals^{m(n)\times n}$ to
be a random matrix with independent sub-Gaussian entries,
with mean $0$, variance $1/m$ and common scale factor $s/m$.
Further assume $A_{ij}(n) = \tA_{ij}(n)+\nu_0 G_{ij}(n)$ where $\nu_0>0$
is independent of $n$ and $\{G_{ij}(n)\}_{i\in[m],j\in[n]}$ is a
collection of
i.i.d. $\normal(0,1/m)$ random variables independent of $\tA(n)$.

Consider either of the following two cases:
\begin{longlist}[(II)]
\item[(I)] The matrix $A(n)$ has i.i.d. entries, and $\{x_0(n)\}
_{n\ge1}$
is any fixed sequence of vectors with $\lim_{n\to\infty}\llVert x_0(n)\rrVert
_0/m(n) = \rho$.
\item[(II)] The matrix $A(n)$ has independent but not identically
distributed entries. The vectors $x_0(n)$ have i.i.d. entries
independent of $A(n)$, with\break $\prob\{x_{0,i}(n)\neq0\} =\rho\delta$.
\end{longlist}
Then the following holds.
If $\rho<\rho_*(\delta)$, then $\ell_1$ minimization succeeds with high
probability. Vice versa, if $\rho>\rho_*(\delta)$, then $\ell_1$
minimization fails with high probability.
[Here probability is with respect to the realization of the random
matrix $A(n)$ and, eventually, $x_0(n)$.]
\end{theorem}

The rest of this section is devoted to the proof of Theorem~\ref{thm:CS}.
Indeed, as shown below, this immediately implies Theorem~\ref{thm:Polytope}.

\begin{pf*}{Proof of Theorem~\ref{thm:Polytope}}
Take $x_0(n)$ to be a sequence of independent vectors with
independent entries such that $\prob_{\rho}\{x_0(n)_i = 1\} = \rho
\delta
$ and\break
$\prob_{\rho}\{x_0(n)_i = 0\} = 1-\rho\delta$. Then, by the law of
large numbers
we have\break $\lim_{n\to\infty}\llVert x_0(n)\rrVert _0/m(n) = \rho$ almost surely.
Let $A(n)\in\reals^{m(n)\times n}$ be a matrix with i.i.d. entries as
per Hypothesis~\ref{hyp1}
above, with $m(n)=\lfloor n\delta\rfloor$ and $y(n) =A(n)x_0(n)$.
Applying Theorem~\ref{thm:CS}, we have, for any $\rho_-<\rho
_*(\delta)$
and $\rho_+>\rho_*(\delta)$,
%
\begin{eqnarray}\label{eq:CSProbabilityMinus}
\lim_{n\to\infty}\prob_{\rho_-} \bigl\{\hx
\bigl(y(n)\bigr) = x_0(n) \bigr\} &=& 1 ,
\\
\label{eq:CSProbabilityPlus}
\lim_{n\to\infty}\prob_{\rho_+} \bigl\{\hx\bigl(y(n)\bigr) =
x_0(n) \bigr\} &=& 0  ,
\end{eqnarray}
where $\prob_{\rho_{\pm}} \{ \cdot \}$ denotes probability with
respect to the law just described when $\rho=\rho_{\pm}$.
Let $V(\rho;m,n)$ be the fraction of vectors $x_0$ with $\llVert x_0\rrVert
=\lfloor m\rho\rfloor$ on which $\ell_1$ reconstruction
succeeds. Since in equations (\ref{eq:CSProbabilityMinus}), (\ref
{eq:CSProbabilityPlus}),
support of $x_0(n)$ is uniformly random given its size, and the
probability of success is monotone decreasing in the support size
\cite{Donoho2005a},
the above equations imply
%
\begin{eqnarray}\label{eq:CSProbabilityMinus_BIS}
\lim_{n\to\infty}\E \bigl\{ V(\rho_-;m,n) \bigr\} &=& 1
,
\\
\label{eq:CSProbabilityPlus_BIS}
\lim_{n\to\infty}\E \bigl\{ V(\rho_+;m,n) \bigr\} &=& 0.
\end{eqnarray}
Using Markov's inequality, equations (\ref{eq:CSProbabilityMinus_BIS}),
(\ref{eq:CSProbabilityPlus_BIS})
coincide (resp.) with assumptions~(1) and~(2) in Theorem~\ref{thm:CSPolytope}.
The claim follows by applying this theorem.
\end{pf*}

Let us now turn to the proof of Theorem~\ref{thm:CS}.
The following lemma provides a useful sufficient condition for successful
reconstruction.
Here and below, for a convex function $F\dvtx \reals^q\to\reals$,
$\partial
F(x)$ denotes the subgradient of $F$ at $x\in\reals^q$.
In particular $\partial\llVert x\rrVert _1$ denotes the subgradient of the $\ell_1$
norm at $x$. Further, for
$R\subseteq[n]$, $A_R$ denotes the submatrix of $A$ formed by columns
with index in $R$. The singular values of a matrix
$M\in\reals^{d_1\times d_2}$ are denoted by $\sigma_{\max}(M)\equiv
\sigma_{1}(M)\ge
\sigma_{2}(M)\ge\cdots\ge\sigma_{\min(d_1,d_2)}(M)\equiv\sigma
_{\min}(M)$.

\begin{lemma}\label{5551555515555}
For any $c_1,c_2,c_3>0$, there exists $\eps_0(c_1,c_2,c_3)>0$ such that the following happens.
If $x_0\in\reals^n$, $A\in\reals^{m\times n}$,
$y=Ax_0\in\reals^m$, are such
that:

\renewcommand\thelonglist{(\arabic{longlist})}
\renewcommand\labellonglist{\thelonglist}

\begin{longlist}
\item\label{H:Grad} There exists $v\in\partial\llVert x_0\rrVert _1$ and $z\in
\reals^m$ with $v = A^{\sT}z+w$ and
$\llVert w\rrVert _2\le\sqrt{n} \eps$,
with $\eps\le\eps_0(c_1,c_2,c_3)$.
\item\label{H:Key} For $c\in(0,1)$, let $S(c) \equiv\{i\in[n]\dvtx  \llvert v_i\rrvert \ge
1-c\}
$. Then,
for any $S'\subseteq[n]$, $\llvert S'\rrvert \le c_1n$, the minimum singular
value of $A_{S(c_1)\cup S'}$ satisfies
$\sigma_{\min}(A_{S(c_1)\cup S'})\ge c_2$.
\item\label{H:SVD} The maximum singular value
of $A$ satisfies $c_3^{-1}\le\sigma_{\max}(A)^2\le
c_3$.
\end{longlist}
Then $x_0$ is the unique minimizer of $\llVert x\rrVert _1$ over
$x\in\reals^n$ such that $y=Ax$.
\end{lemma}

The proof of this lemma is deferred to Appendix \ref{app:ProofSubgradient}.

The proof of Theorem~\ref{thm:CS} consists of two parts. For
$\rho>\rho_*(\delta)$, we shall exhibit a vector $x$ with
$\llVert x\rrVert _1<\llVert x_0\rrVert _1$ and $y=Ax$. For $\rho<\rho_*(\delta)$ we will
show that assumptions of Lemma~\ref{5551555515555} hold. In particular,
we will construct a subgradient $v$ as per assumption \ref{H:Grad}.
In both tasks, we will use an iterative message
passing algorithm analogous to the one in Section~\ref
{sec:NonSymmetric}. The algorithm is defined by the
following recursion initialized with $x^0=0$:
%
\begin{eqnarray}\label{eq:AMPell1_1}
x^{t+1} & = & \eta\bigl(x^t+ A^{\sT}
z^t; \alpha\sigma_t\bigr),
\\
\label{eq:AMPell1_2}
z^t & = & y-Ax^t+\mathsf{b}_t
z^{t-1}, %
\end{eqnarray}
where $\eta(u;\theta) = \sign(u) (u-\theta)_+$, $\alpha$ is a
nonnegative constant and
$\mathsf{b}_t$ is a diagonal matrix whose precise definition is immaterial
here and will be given in the proof of Proposition~\ref{1616161616} below.
Notice two important differences with respect to the treatment in
Section~\ref{sec:NonSymmetric}:
\begin{itemize}
\item The iteration in equations (\ref{eq:AMPell1_1}), (\ref
{eq:AMPell1_2}) does not take immediately the
form in equations (\ref{eq:BipartiteAMP1}), (\ref{eq:BipartiteAMP2}). For
instance the nonlinear mapping $\eta( \cdot;\alpha\sigma_t)$ is
applied \emph{after} multiplication by $A^{\sT}$. This mismatch can be
resolved by a simple change of variables.
\item The nonlinear mapping $\eta( \cdot;\alpha\sigma_t)$ is not a
polynomial. This point will be addressed by constructing suitable
\emph{polynomial approximations} of $\eta$.
\end{itemize}
We refer to Appendix \ref{app:SEell1} for further details.

For $t\ge0$, $\sigma_t$ is defined by the
one-dimensional recursion
%
\begin{equation}\label{eq:ell1SE}
\sigma_{t+1}^2 = \frac{1}{\delta}\E\bigl\{
\bigl[\eta(X+\sigma_t Z;\alpha\sigma_t)-X
\bigr]^2\bigr\},%
\end{equation}
where expectation is with respect to the independent random variables
$Z\sim\normal(0,1)$,
$X\sim p_X$, and $\sigma_0^2 = \E\{X^2\}/\delta$.

\begin{proposition}\label{1616161616}
Let
$\{(x_0(n),A(n),y(n))\}_{n\ge0}$ be a sequence of
triples with $A(n)$ random as per Hypothesis~\ref{hyp1}, $\{x_{0,i}(n)\dvtx  i\in
[n]\}$ independent and identically distributed with $x_{0,i}(n)\sim
p_X$ a finite mixture of Gaussians on $\reals$, and $y(n) =
A(n)x_0(n)$.

Then, for each $n$ there exist a
sequence of vectors
$\{x^t(n),z^t(n)\}_{t\ge0}$, with $x^t(n)=x^t\in\reals^n$,
$z^t(n)=z^t\in\reals^m$, such that the following happens for every~$t$.
\begin{longlist}[(3)]
\item[(1)] There exists a diagonal matrix
$\mathsf{b}_t = \mathsf{b}_t(n)$ such that
%
\begin{eqnarray}\label{eq:LemmaSE_1}
z^t &=& y-Ax^t+\mathsf{b}_t
z^{t-1},
\\
\lim_{n\to\infty}\max_{i\in[m]} (
\mathsf{b}_t)_{ii} &=&\lim_{n\to
\infty}\min
_{i\in[m]} (\mathsf{b}_t)_{ii} =
\frac{1}{\delta} \prob \bigl\{\llvert X+\sigma_{t-1}Z\rrvert \ge\alpha
\sigma_{t-1} \bigr\}, %
\end{eqnarray}
where the limit holds in probability.
\item[(2)] In probability,
%
\begin{equation}\label{eq:LemmaSE_2}
\lim_{n\to\infty} \frac{1}{n} \bigl\llVert
x^{t+1}-\eta\bigl(x^t+A^{\sT}z^t;\alpha
\sigma_t\bigr)\bigr\rrVert _2^2 = 0 .
\end{equation}
\item[(3)] For any locally Lipschitz function
$\psi\dvtx \reals\times\reals\to\reals$, $\llvert \psi(x,y)\rrvert \le
C(1+x^2+y^2)$, in
probability
%
\begin{equation}\label{eq:LemmaSE_3}
\lim_{n\to\infty}\frac{1}{n}\sum
_{i=1}^n\psi \bigl(x_{0,i},x^t_i+
\bigl(A^{\sT}z^t\bigr)_i\bigr) =\E\psi(X,X+
\sigma_tZ). %
\end{equation}
\item[(4)]
There exist  two functions $o(a;c)$ and $o(a,b;c)$, with
$o(a;c)\to0$,\break $o(a,b;c)\to0$ as $c\to0$ at $a,b$ fixed, such that
the following
holds.\break Assume
{\spaceskip=0.2em plus 0.05em minus 0.02em
$A_{ij}(n) = \tA_{ij}(n)+\nu G_{ij}(n)$ where $\nu>0$
is independent of $n$ and\break $\{G_{ij}(n)\}_{i\in[m],j\in[n]}$} is a
collection of
i.i.d. $\normal(0,1/m)$ random variables independent of $\tA(n)$.
Then there exists a sequence of vectors $\{\tx^t,\tilde{z}^t\}_{t\ge
0}$ that is
independent of $G$ such that, for any $t\ge0$,
%
\begin{eqnarray}
\qquad \frac{1}{n}\sum_{i=1}^n
\E \bigl\{ \bigl(\bigl(x^t+A^{\sT}z^t
\bigr)_i-\bigl(\tx ^t+\tA^{\sT
}
\tilde{z}^t\bigr)_i \bigr)^2 \bigr\}&\le&
o(t;\nu)+o\bigl(t,\nu;n^{-1}\bigr),
\\
\frac{1}{m}\sum_{i=1}^m\E \bigl\{
\bigl(z^t_i-\tilde{z}^t_i
\bigr)^2 \bigr\}&\le& o(t;\nu)+o\bigl(t,\nu;n^{-1}\bigr).
\end{eqnarray}
\end{longlist}
\end{proposition}

The proof is deferred to Appendix \ref{app:SEell1}.

We also need a generalization of the last proposition for functions of
the estimates $x^t$, $x^s$ at two distinct iteration numbers $t\neq
s$. To this objective, we introduce the generalization of the state
evolution equation (\ref{eq:ell1SE}). Namely, we define
$\{R_{s,t}\}_{s,t\ge0}$
recursively for all $s,t\ge0$ by letting
%
\begin{equation}\label{eq:TwoTimesRecursion}
R_{s+1,t+1} = \frac{1}{\delta} \E \bigl\{\bigl[
\eta(X+Z_s;\alpha\sigma_s)-X\bigr] \bigl[
\eta(X+Z_t;\alpha\sigma _t)-X\bigr] \bigr\}.
\end{equation}
Here the expectation is with respect to $X\sim p_X$ and the independent
Gaussian vector $[Z_s,Z_t]$ with zero mean and covariance given by $\E
\{Z_s^2\}=R_{s,s}$,
$\E\{Z_t^2\}=R_{t,t}$ and $\E\{Z_tZ_s\}= R_{t,s}$. The boundary
condition is fixed by letting $R_{0,0} =\E\{X^2\}/\delta$ and
defining, for each $t\ge0$,
%
\begin{equation}
R_{0,t+1} = \frac{1}{\delta} \E \bigl\{\bigl[
\eta(X+Z_t;\alpha\sigma_t)-X\bigr] [-X] \bigr\},
\end{equation}
with $Z_t\sim\normal(0,R_{t,t})$. This uniquely determines the doubly
infinite array $\{R_{t,s}\}_{t,s\ge0}$. Notice in particular that
$R_{t,t} =\sigma_t^2$ for all $t\ge0$. (This is easily checked by
induction over $t$.)

\begin{proposition}\label{171717171717}
Under the assumptions of Proposition~\ref{1616161616} the
sequence $\{x^t(n),z^t(n)\}_{t\ge0}$ constructed there further
satisfies the following.
For any fixed $t,s\ge0$, and any Lipschitz continuous functions
$\psi\dvtx \reals\times\reals\times\reals\to\reals$,
$\phi\dvtx \reals\times\reals\to\reals$,
in probability
%
\begin{eqnarray}
&&\lim_{n\to\infty}\frac{1}{n}\sum
_{i=1}^n\psi \bigl(x_{0,i},x^s_i+
\bigl(A^{\sT
}z^s\bigr)_i,x^t_i+
\bigl(A^{\sT}z^t\bigr)_i \bigr)  \nonumber\\[-8pt]\\[-8pt]
&&\qquad=  \E
\psi(X,X+Z_s,X+Z_t),\nonumber
\\
&&\lim_{n\to\infty}\frac{1}{m}\sum
_{i=1}^n\phi\bigl(z^s_i,z^t_i\bigr)  =  \E\phi(Z_s,Z_t), %
\end{eqnarray}
where expectation is with respect to $X\sim p_X$, and the independent
Gaussian vector $(Z_s,Z_t)$ with zero mean and covariance given by $\E
\{Z_s^2\}=R_{s,s}$,
$\E\{Z_t^2\}=R_{t,t}$ and $\E\{Z_tZ_s\}= R_{t,s}$.
\end{proposition}

The proof of this proposition is in Appendix \ref{app:SEell1}.

Finally, we need some analytical estimates on the recursions
(\ref{eq:ell1SE}) and (\ref{eq:TwoTimesRecursion}). Part of these
estimates were already proved in
\cite{DMM09,NSPT,BayatiMontanariLASSO},
but we reproduce them here for the reader's convenience. Proofs of the
others are provided in Appendix~\ref{app:AsympSE}.

\begin{lemma}\label{6661666616666166616661}
Let $p_X$ be a probability measure on the real line such that
$p_X(\{0\}) = 1-\eps$ and $\E_{p_X}\{X^2\}<\infty$, fix $\delta\in
(0,1)$ and set $\rho=\delta\eps$. For this choice of parameters,
consider the
sequences $\{\sigma_t^2\}_{t\ge0}$, $\{R_{s,t}\}_{s,t\ge0}$ defined
as per equations (\ref{eq:ell1SE}), (\ref{eq:TwoTimesRecursion}).

If $\rho<\rho_*(\delta)$, then:
\begin{longlist}[(a3)]
\item[(a1)] There exists $\alpha_1(\eps,\delta)$,
$\alpha_2(\eps,\delta)$, $\alpha_*(\eps)$ with
$0<\alpha_1(\eps,\delta)<\alpha_*(\eps)<\alpha_2(\eps,\allowbreak\delta
)<\infty$,
and $\omega_*(\eps,\delta)\in
(0,1)$ such that the following happens. For each $\alpha\in
(\alpha_1,\alpha_2)$, $\sigma_t^2 = B \omega^t(1+o_t(1))$
as $t\to\infty$, with $\omega\in(0,1)$.

Further, for each $\omega\in
[\omega_*(\eps,\delta),1)$ there exists $\alpha_-\in
(\alpha_1,\alpha_*]$ and $\alpha_+\in[\alpha_*,\alpha_2)$ (distinct
as long as $\omega>\omega_*$) such that, letting $\alpha\in\{\alpha
_-,\alpha_+\}$, $\sigma_t^2 = B
\omega^t(1+o_t(1))$.

Finally, for all $\alpha\in[\alpha_*,\alpha_2)$, we have $\eps
+2(1-\eps)\Phi(-\alpha)<\delta$.
\item[(a2)] For any $\alpha\in[\alpha_*(\eps),\alpha_2(\eps,\delta
))$, we have
$\lim_{t\to\infty}R_{t,t-1}/(\sigma_t\sigma_{t-1})=1$.
\item[(a3)] Assume $p_X$ to be such that $\max(p_X((0,a)),
p_X((-a,0)))\le
Ba^{b}$ for some $B,b>0$ (in particular, this is the case if $p_X$
has an atom at $0$ and is absolutely continuous in a neighborhood of
$0$). Fixing again
$\alpha\in[\alpha_*(\eps),\alpha_2(\eps,\delta))$, and $c\in
\reals_+$,
%
\begin{equation}\label{eq:PointA3}
\lim_{t_0\to\infty}\sup_{t,s\ge t_0} \prob
\bigl\{\llvert X+Z_s\rrvert \ge c \sigma_s; \llvert
X+Z_t\rrvert <c \sigma_t \bigr\}=0 ,
\end{equation}
where $(Z_s,Z_t)$ is a Gaussian vector with $\E\{Z_s^2\} =
\sigma_s^2$, $\E\{Z_s^2\} = \sigma_s^2$, $\E\{Z_sZ_t\} =
R_{s,t}$.
\end{longlist}
Vice versa, if $\rho>\rho_*(\delta)$, then there exists $\alpha_{0}
(\delta,p_X)>\alpha_{\min}(\delta)>0$ such that:
\begin{longlist}[(b3)]
\item[(b1)] For any $\alpha>\alpha_{\min}(\delta)$, we have
$\lim_{t\to\infty}\sigma_t^2 =
\sigma_*^2>0$ and, for $\alpha\ge\alpha_{0}$, $\lim_{t\to\infty
}[R_{t,t}-2R_{t,t-1}+R_{t-1,t-1}]=0$.
\item[(b2)] Letting $\alpha= \alpha_{0}(\delta,p_X)$, we have
$\prob\{\llvert X+\sigma_*Z\rrvert \ge\alpha\sigma_*\} = \delta$.
\item[(b3)] Consider the probability distribution $p_X =
(1-\eps)\delta_0+\eps\gamma$ with $\gamma(\de x) =
\exp(-x^2/2)/\sqrt{2\pi}\, \de x$ the standard Gaussian measure. Then,
setting
$\alpha= \alpha_{0}(\delta,\allowbreak p_X)$, we have $\lim_{t\to\infty}
\E\{\llvert \eta(X+\sigma_t Z;\alpha\sigma_t)\rrvert \} < \E\{\llvert X\rrvert \}$, where
$Z\sim
\normal(0,1)$ independent of $X$.
\end{longlist}
\end{lemma}

We are now in position to prove Theorem~\ref{thm:CS}.
For greater convenience of the reader, we distinguish the cases
$\rho<\rho_*(\delta)$ and $\rho>\rho_*(\delta)$. Before considering
these cases, we will establish some common simplifications.

%
%
\subsection{Proof of Theorem \texorpdfstring{\protect\ref{thm:CS}}{8}: Common simplifications}

Consider first case (I).
By exchangeability of the columns of $A(n)$, it is sufficient to prove
the claim
for the sequence of random vectors obtained by permuting the entries
of $x_0(n)$ uniformly at random.
Hence $x_0(n)$ is a vector with a uniformly random
support $\supp(x_0(n)) = S_n$, with deterministic size $\llvert S_n\rrvert $ such
that $\llvert S_n\rrvert /n\to\eps$.
Further, the success of $\ell_1$ minimization is an event that is
monotone decreasing in the
support $\supp(x_0(n))$ \cite{Donoho2005a}.
Therefore we can replace the deterministic
support size, with a random size $\llvert S_n\rrvert \sim\operatorname
{Binom}(n,\eps
)$ (which concentrates tightly around
$n\eps$).

Finally, since success of $\ell_1$
minimization only depends on the support of $x_0(n)$ \cite
{Donoho2005a}, we can replace
the nonzero entries by arbitrary values. We will take advantage of
this fact and assume that all the nonzero entries of $x_0(n)$ are
i.i.d. $\normal(0,1)$.
We conclude that it is sufficient to prove that
$\ell_1$ minimization succeeds/fails with high probability if the vectors
$x_0(n)$ have i.i.d. entries with distribution $p_X=
(1-\eps)\delta_0+\eps\gamma$, where $\gamma(\de x) =
\exp(-x^2/2)/\sqrt{2\pi} \,\de x$.\vadjust{\goodbreak}

Consider next case (II), in which the entries of $x_0(n)$ are
i.i.d. with $\prob\{x_{0,i}(n)\neq0\}=\rho\delta=\eps$. Again,
exploiting the fact that the success of $\ell_1$ minimization depends
only on the support of $x_0(n)$, we can assume that its entries have
common distribution $p_X= (1-\eps)\delta_0+\eps\gamma$.

Summarizing this discussion, in order to prove the Theorem both in
case (I) and case (II), it will be sufficient to do so for the following
setting:

\begin{remark}
In the proof of Theorem~\ref{thm:CS}, we can assume the vectors
$x_0(n)$ to be random
with i.i.d. entries with common distribution $p_X=
(1-\eps)\delta_0+\eps\gamma$.\looseness=-1
\end{remark}

%
%
\subsection{Proof of Theorem \texorpdfstring{\protect\ref{thm:CS}}{8}, \texorpdfstring{$\rho<\rho_*(\delta)$}{$rho<rho_*(delta)$}}

Fix $\rho<\rho_*(\delta)$. We will prove that
the hypotheses \ref{H:Grad}, \ref{H:Key}, \ref{H:SVD} of Lemma~\ref{5551555515555} hold with high probability for fixed $c_1,
c_2, c_3>0$, and $\eps$ arbitrarily small. This implies the claim
(i.e., that $\ell_1$ minimization succeeds) by applying the lemma.
Notice that hypothesis \ref{H:SVD} holds with high probability for
some $c_3=c_3(\delta)$ by classical estimates on the
extreme eigenvalues of sample covariance matrices
\cite{BaiSilversteinPaper,BaiSilverstein}.

We next consider hypothesis \ref{H:Grad} of Lemma~\ref{5551555515555}. In order to construct the subgradient $v$
used there, we consider the sequence of vectors $\{x^t,z^t\}_{t\ge0}$
defined by as per Proposition~\ref{1616161616}. We fix $\alpha
\in(
\alpha_1(\eps),\alpha_2(\eps))$ as per Lemma~\ref{6661666616666166616661}(a) so that $\sigma_t^2 = A\omega^t(1+o(1))$ with
$\omega\in(0,1)$ to be chosen close enough to $1$. Also, we introduce
the notation $\theta_t\equiv\alpha\sigma_t$.
We let $v^t\in\reals^n$ be defined by
%
\begin{eqnarray}
v^t_i &=& \cases{\displaystyle \sign(x_{0,i})
&\quad if $i\in S$,
\cr
\displaystyle\frac{1}{\theta_{t-1}} \bigl(x^{t-1}+A^{\sT}z^{t-1}-
\hx^t \bigr)_i&\quad otherwise, }
\\
\hx^t & \equiv& \eta\bigl(x^{t-1}+A^{\sT}z^{t-1};
\theta_{t-1}\bigr). %
\end{eqnarray}
Notice that, by definition of the function $\eta( \cdot;
\cdot)$ we have $\llvert x^{t-1}_i-(A^{\sT}z^{t-1})_i-\hx^t_i\rrvert \le\theta_{t-1}$,
and hence $v^t\in\partial\llVert x_0\rrVert _1$. We can write
%
\begin{eqnarray}
v^t &= &\frac{1}{\theta_{t-1}} A^{\sT}z^{t}
+\xi^t +\beta^t+\zeta ^t,
\\
\xi^t&\equiv &\frac{1}{\theta_{t-1}} \bigl(x^{t-1}+A^{\sT}z^{t-1}-x^t-A^{\sT
}z^t
\bigr),
\\
\beta^t & \equiv& \frac{1}{\theta_{t-1}} \bigl(x^t-
\hx^t \bigr),
\\
\zeta^t & \equiv& \cases{ \displaystyle\sign(x_{0,i})-
\frac{1}{\theta_{t-1}} \bigl(x^{t-1}+A^{\sT
}z^{t-1}-\hx
^t \bigr)_i &\quad if $i\in S$,
\cr
0 &\quad otherwise. }
\end{eqnarray}
This part of the proof is completed by showing that there exists
$h(t)$ with $\lim_{t\to\infty}h(t) =0$ such that, for each $t$,
with high probability we have $\llVert \xi^t\rrVert _2^2/n\le
(1-\sqrt{\omega})^2/\alpha^2+h(t)$,
$\llVert \beta^t\rrVert _2^2/n\le h(t)$ and
$\llVert \zeta^t\rrVert _2^2/n\le h(t)$. Indeed, if this is true, we can
then choose $t$ sufficiently large and $\alpha\in(\alpha_*(\eps
),\alpha
_2(\eps,\delta))$
so that $\llVert \xi^t+\beta^t+\zeta^t\rrVert _2^2$ is small enough as to satisfy
the condition
\ref{H:Grad} of Lemma~\ref{5551555515555}.

First consider $\xi^t$. Applying Proposition~\ref{171717171717}
to $\psi(x,y_1,y_2) =
(y_1-y_2)^2$, we have, in probability
\begin{eqnarray*}
\lim_{n\to\infty}\frac{1}{n}\bigl\llVert
\xi^t\bigr\rrVert _2^2 &=&\lim
_{n\to\infty} \frac{1}{n\alpha^2\sigma_{t-1}^2} \bigl\llVert x^{t}+A^{\sT}z^t-x^{t-1}-A^{\sT
}z^{t-1}
\bigr\rrVert _2^2
\\
& =& \frac{1}{\alpha^2\sigma_{t-1}^2} [R_{t,t}-2R_{t,t-1}+R_{t-1,t-1}]
\\
&=& \frac{1}{\alpha^2\sigma_{t-1}^2} \bigl[\sigma_{t}^2-2\sigma
_{t}\sigma _{t-1}+\sigma^2_{t-1} \bigr]
+2\frac{\sigma_t}{\sigma_{t-1}} \biggl[1-\frac{R_{t,t-1}}{\sigma
_t\sigma
_{t-1}} \biggr]
\\
&=& \frac{1}{\alpha^2}(1-\sqrt{\omega})^2 + h(t). %
\end{eqnarray*}
Here the last equality follows from the fact that
$\sigma^2_t/\sigma^2_{t-1}\to\omega$ by Lemma~\ref{6661666616666166616661}(a1) and $R_{t,t-1}/(\sigma_t\sigma_{t-1})\to1$
by Lemma~\ref{6661666616666166616661}(a2). This implies the claim for $\xi^t$.

Next, consider $\beta^t$. By Proposition~\ref{1616161616}(2),
%
\begin{equation}
\hspace*{20pt}\lim_{n\to\infty} \frac{1}{n} \bigl\llVert
x^{t}-\hx^t\bigr\rrVert _2^2 =
\lim_{n\to\infty} \frac{1}{n} \bigl\llVert x^{t}-
\eta\bigl(x^{t-1}+A^{\sT
}z^{t-1};\alpha
\sigma_{t-1}\bigr)\bigr\rrVert _2^2 = 0 ,
\end{equation}
and hence $\llVert \beta^t\rrVert _2^2/n\le h(t)$ with high probability for any $h(t)>0$.

Finally consider $\zeta^t$, and define $R(y;\theta) =
y-\eta(y;\theta)$. We have
\begin{eqnarray*}
R(y;\theta) = \cases{ +1 &\quad for $y\ge\theta$,
\cr
y/\theta&\quad for $-
\theta<y<\theta$,
\cr
-1 &\quad for $y\le-\theta$. } %
\end{eqnarray*}
Using Proposition~\ref{1616161616}(3), we can show that
%
\begin{equation}\label{eq:LimitZetaT}
\lim_{n\to\infty}\frac{1}{n}\bigl\llVert
\zeta^t\bigr\rrVert _2^2 = \E\bigl\{\bigl[
\sign(X) - R(X+\sigma_{t-1}Z;\alpha\sigma_{t-1})
\bigr]^2\ind _{X\neq
0}\bigr\}. %
\end{equation}
Notice that this apparently requires applying Proposition~\ref{1616161616} to the function $\psi(x,y) =
[\sign(x)-R(y;\theta)]^2\ind_{x\neq0}$ which is non-Lipschitz in
$x$. However, we can define a Lipschitz approximation, with parameter
$r>0$,
%
\begin{eqnarray}
\psi_r(x,y) = \cases{ \displaystyle\bigl[x/r-R(y;\theta)
\bigr]^2 \llvert x\rrvert /r &\quad for $\llvert x\rrvert \le r$,\vspace*{2pt}
\cr
\displaystyle\bigl[1-R(y;\theta)\bigr]&\quad for $\llvert x\rrvert >r$. } %
\end{eqnarray}
Notice that $\psi_r$ is bounded and Lipschitz continuous.
We further have\break $\llvert \psi_{r}(x,y)-\psi(x,y)\rrvert \le4 \ind(x\neq0;
\llvert x\rrvert \le r)$, whence
%
\begin{eqnarray}\label{eq:BoundPsiR_1}
&&\lim\sup_{n\to\infty} \biggl|\frac{1}{n} \bigl\llVert
\zeta^t\bigr\rrVert _2^2-\frac
{1}{n}
\sum_{i=1}^n\psi_{r}
\bigl(x_{0,i},x^{t-1}_i+A^{\sT}z^{t-1}
\bigr) \biggr|\nonumber\\[-8pt]\\[-8pt]
&&\qquad\le\lim\sup_{n\to\infty} \frac{4}{n}\sum
_{i=1}^n \ind \bigl(x_{0,i}\neq0; \llvert
x_{0,i}\rrvert \le r\bigr)\le8 r.\nonumber%
\end{eqnarray}
The last inequality holds almost surely by the law of large numbers
using $\gamma([-r,r])< 2r$. Analogously,
%
\begin{eqnarray}\label{eq:BoundPsiR_2}
&& \bigl\llvert \E\psi(X,X+\sigma_{t-1}Z)- \E
\psi_r(X,X+\sigma_{t-1}Z) \bigr\rrvert 
\nonumber\\[-8pt]\\[-8pt]\nonumber
&&\qquad \le 4 \prob \bigl(X
\neq0; \llvert X\rrvert \le r\bigr) \le8r.%
\end{eqnarray}
Hence the claim (\ref{eq:LimitZetaT}) follows by applying Proposition~\ref{1616161616}(3)
to $\psi_r(x,y)$, using equations (\ref{eq:BoundPsiR_1}),
(\ref{eq:BoundPsiR_2}), and letting $r\to0$.

We conclude by noting that the right-hand side of
equation (\ref{eq:LimitZetaT}) converges to $0$ as $t\to\infty$ by
dominated convergence, since $\sigma_t\to0$. Therefore,
\[
\lim_{n\to\infty}\frac{1}{n}\bigl\llVert
\zeta^t\bigr\rrVert _2^2 \le\frac{h(t)}{2}. %
\]
This completes our proof of assumption \ref{H:Grad} of Lemma~\ref{5551555515555}.

We finally consider hypothesis \ref{H:Key}. Let $S_t(c)$ be defined as
there, for the subgradient $v^t$, namely
\begin{eqnarray*}
S_t(c) &\equiv& \bigl\{i\in[n] \dvtx  \bigl\llvert
v_i^t\bigr\rrvert \ge1-c \bigr\}
\\
& =& S\cup \bigl\{i\in[n]\setminus S \dvtx  \bigl\llvert x^{t-1}+A^{\sT}z^{t-1}
\bigr\rrvert \ge (1-c)\theta_{t-1} \bigr\}. %
\end{eqnarray*}
Recall that by assumption $A_{ij} = \tA_{ij}+\nu G_{ij}$ whereby
$G_{ij}\sim\normal(0,1/m)$ and (eventually redefining $\tA_{ij}$), we
can freely choose $\nu\in[0,\nu_0]$.
Let $\{\tx^t,\tilde{z}^t\}_{t\ge0}$ be a sequence of vectors defined
as per
Proposition~\ref{1616161616}(4), and define $\tv^t$ as $v^t$,
but
replacing $x^t,z^t, A$ by $\tx^t,\tilde{z}^t, \tA$
%
\begin{eqnarray}
\tv^t_i &=& \cases{\displaystyle \sign(x_{0,i})
&\quad if $i\in S$,
\cr
\displaystyle\frac{1}{\theta_{t-1}} \bigl(\tx^{t-1}+
\tA^{\sT}z^{t-1}-\hx^t \bigr)_i&\quad
otherwise, }
\\
\hx^t & \equiv& \eta\bigl(\tx^{t-1}+\tA^{\sT}
\tilde{z}^{t-1};\theta _{t-1}\bigr). %
\end{eqnarray}
We further define
\begin{eqnarray*}
\tS_t(c) & \equiv &\bigl\{i\in[n] \dvtx  \bigl\llvert
\tv_i^t\bigr\rrvert \ge1-c \bigr\}
\\
& =& S\cup \bigl\{i\in[n]\setminus S \dvtx  \bigl\llvert \tx^{t-1}+
\tA^{\sT}\tilde {z}^{t-1}\bigr\rrvert \ge (1-c)
\theta_{t-1} \bigr\}. %
\end{eqnarray*}
We claim that the following two claims hold for some $t_*\ge0$
independent of~$n$:

\begin{cla}\label{cl1}
There exists $c_1,\hc_2>0$ (independent of
$\nu
$) such that for
all $S'\subseteq[n]$, $\llvert S'\rrvert \le2c_1n$, the minimum singular value of
$A_{\tS_{t_*}(2c_1)\cup S'}$, satisfies $\sigma_{\min}(A_{\tS
_{t_*}(2c_1)\cup
S'})\ge\hc_2\nu$ with probability converging to $1$ as $n\to\infty$.
\end{cla}

\begin{cla}\label{cl2} For all $t\ge t_*$,
\[
\prob \bigl\{\bigl\llvert S_t(c_1)\setminus
\tS_{t_*}(2c_1)\bigr\rrvert \ge n c_1 \bigr\}
= o_1(t_*;\nu)+o_2\bigl(t_*,\nu;n^{-1}\bigr),
\]
where $o_1(t_*,\nu)$ vanishes as $\nu\to0$ at $t_*$,
$c_1$, $c_2$ fixed, and $o_2(t_*,\nu;n^{-1})$ vanishes as $n^{-1}\to
0$ at $t_*$, $\nu$,
$c_1$, $c_2$ fixed.
\end{cla}

These claims immediately imply that hypothesis \ref{H:Key} of Lemma~\ref{5551555515555} holds with probability converging to one as
$n\to\infty$. Indeed, if $\llvert S'\rrvert \le nc_1$, then (by Claim~\ref{cl2})
$S_t(c_1)\cup S'
\subseteq\tS_{t_*}(2c_1)\cup S''$ where $\llvert S''\rrvert \le2nc_1$ with
probability larger than $1-o_1(t_*;\nu)-o_2(t_*,\nu;n^{-1})$. By Claim~\ref{cl1}, we hence have
$\sigma_{\min}(A_{S_{t}(c_1)\cup S'})\ge c_2\equiv\hc_2\nu$. The
thesis follows
since $\nu$ can be chosen as small as we want.
[Notice that once $t_*$ is fixed to satisfy these claims, we can still
choose $t\ge t_*$ arbitrarily to satisfy hypothesis \ref{H:Grad} of Lemma~\ref{5551555515555}, as per the argument above.]

In order to prove Claim~\ref{cl1}, above first notice that, for any $b\ge0$
%
\begin{eqnarray}\label{eq:BigBoundSingValue}
&&\prob \Bigl\{\mathop{\min_{S'\subseteq[n]}}_{\llvert S'\rrvert \le2c_1n}\sigma_{\min
}(A_{\tS_{t_*}(2c_1)\cup
S'})<
\hc_2\nu \Bigr\}
\nonumber
\\
&&\qquad\le\prob \Bigl\{\mathop{\min_{S'\subseteq[n]}}_{\llvert S'\rrvert \le2c_1n}\sigma _{\min
}(A_{\tS_{t_*}(2c_1)\cup
S'})<
\hc_2\nu; \bigl\llvert \tS_{t_*}(2c_1)\bigr
\rrvert \leq bn \Bigr\}\nonumber\\
&&\quad\qquad{} + \prob \bigl\{\bigl\llvert \tS_{t_*}(2c_1)
\bigr\rrvert > bn \bigr\}
\\
&&\qquad\le e^{nH(2c_1)}\mathop{\max_{S'\subseteq[n]}}_{\llvert S'\rrvert \le2c_1n}\prob \bigl\{
\sigma_{\min}(A_{\tS_{t_*}(2c_1)\cup
S'})< \hc_2\nu; \bigl\llvert
\tS_{t_*}(2c_1)\bigr\rrvert \le bn \bigr\}
\nonumber\\
&&\quad\qquad{}+ \prob \bigl\{
\bigl\llvert \tS_{t_*}(2c_1)\bigr\rrvert > bn \bigr\},\nonumber%
\end{eqnarray}
where in the last
line $H(c)$ denotes the binary entropy of $b$, and we used
$\binom{n}{nc}\le\exp\{nH(c)\}$.
We want to show that $t_*$, $b$,
$c_1$, $c_2$, $\nu$ can be chosen so that both contributions vanish as
$n\to\infty$.

Consider any $b\in(0,\delta)$ and restrict $c_1\in
(0,(\delta-b)/2)$. Then the matrix $A_{\tS_{t_*}(2c_1)\cup S'}$ has
$n\delta$ rows and $n\delta-\Theta(n)$ columns. Further $A = \tA
+\nu
G$ with $\tS_{t_*}(2c_1)$ [and hence $\tS_{t_*}(2c_1)\cup S'$]
independent of $G$. We can therefore use an upper bound on the
condition number of randomly perturbed deterministic matrices proved
by Buergisser and Cucker \cite{Cucker} (see also Appendix
\ref{app:calculus}) to show that
%
\begin{equation}
\hspace*{15pt}\prob \bigl\{\sigma_{\min}(A_{\tS_{t_*}(2c_1)\cup
S'})<
\hc_2\nu; \bigl\llvert \tS_{t_*}(2c_1)\bigr
\rrvert \le bn \bigr\}\le (a_1\hc_2)^{n(\delta-b-2c_1)+1}
\end{equation}
with $a_1 = a_1((b+2c_1)/\delta)$ bounded as long as $(b+2c_1)/\delta<1$.
We can therefore select $\hc_2 = 1/(2a_1)$ and select $c_1$ small
enough so that $H(2c_1)\le(1/2)(\delta-b-2c_1)\log2$. This ensures
that the first term in equation (\ref{eq:BigBoundSingValue}) vanishes as
$n\to\infty$.

We are left with the task of selecting $b\in(0,\delta)$, $t_*\ge0$,
so that the
second term vanishes as well, since then we can take $c_1\in
(0,(\delta-b)/2)$. To this hand notice that by Proposition~\ref{1616161616} (and using the fact that $X+\sigma_{t-1}Z$ has
a density) we have, in probability,
\[
\lim_{n\to\infty}\frac{1}{n} \bigl\llvert
S_{t_*}(c)\bigr\rrvert = \prob \bigl\{\llvert X+\sigma_{t_*-1}Z
\rrvert \ge(1-c)\theta_{t_*-1}\bigr\}, %
\]
and further, since $\sigma_t\to0$ as $t\to\infty$ [cf. Lemma~\ref{6661666616666166616661}(a1)] and $\theta_t =
\alpha\sigma_t$, we have
\[
\lim_{t_*\to\infty}\prob \bigl\{\llvert X+
\sigma_{t_*-1}Z\rrvert \ge (1-c)\theta_{t_*-1}\bigr\} =\eps+ 2(1-
\eps)\Phi\bigl(-(1-c)\alpha\bigr). %
\]
On the other hand, by Lemma~\ref{6661666616666166616661}(a1), and since
$\alpha\in[\alpha_*,\alpha_2)$, we have $\eps+ 2(1-\eps)\Phi
(-\alpha
)<\delta$. Hence there exist $b_0\in
(0,\delta)$ and $c_1>0$ so that for all $t_*$ large enough
$\llvert S_{t_*}(3c_1)\rrvert \le nb_0$ with high
probability. Taking $b\in(b_0,\delta)$ and using Markov's inequality
(with $t_*'=t_*-1$)
\begin{eqnarray*}
&&\prob \bigl\{\bigl\llvert \tS_{t_*}(2c_1)
\bigr\rrvert > bn \bigr\} \\
&&\qquad\le \frac{1}{(b-b_0)n}\E\bigl\{\bigl\llvert
\tS_{t_*}(2c_1)\setminus S_{t_*}(3c_1)
\bigr\rrvert \bigr\}+ \prob \bigl\{\bigl\llvert S_{t_*}(3c_1)
\bigr\rrvert > b_0n \bigr\}
\\
&&\qquad \le\frac{1}{(b-b_0) c_1^2\theta_{t_*-1}^2
n}\sum_{i=1}^n\E
\bigl\{ \bigl(\bigl(x^{t_*'}+A^{\sT}z^{t_*'}
\bigr)_i-\bigl(\tx ^{t_*'}+\tA ^{\sT}
\tilde{z}^{t_*'}\bigr)_i \bigr)^2\ge
c_1^2\theta_{t_*'}^2 \bigr\}\\
&&\quad\qquad{}+ \prob
\bigl\{\bigl\llvert S_{t_*}(3c_1)\bigr\rrvert >
b_0n \bigr\}
\\
&&\qquad\le o_1(t_*;\nu)+o_2\bigl(t_*,\nu;n^{-1}
\bigr)+ \prob \bigl\{\bigl\llvert S_{t_*}(3c_1)\bigr\rrvert
> b_0n \bigr\}, %
\end{eqnarray*}
where the last inequality follows from Proposition~\ref{1616161616}(4).
These terms can be made arbitrarily small by choosing $\nu$ small and
$n$ large enough.

In order to complete the proof, we need to show that Claim~\ref{cl2} holds for
eventually larger $t_*$.
First notice that, applying again Proposition~\ref
{1616161616}(4), we get
%
\begin{eqnarray}\label{eq:StildeS}
&&\prob \bigl\{\bigl\llvert S_{t_*}(c_1)
\setminus\tS_{t_*}(2c_1)\bigr\rrvert \ge
nc_1/2 \bigr\}\nonumber\\
&&\qquad\le \frac{2}{nc_1}\E \bigl\{\bigl\llvert
S_{t_*}(c_1)\setminus \tS_{t_*}(2c_1)
\bigr\rrvert \bigr\}
\nonumber\\[-8pt]\\[-8pt]
&&\qquad\le\frac{2}{nc_1} \sum_{i=1}^n\E
\bigl\{ \bigl(\bigl(x^{t_*'}+A^{\sT}z^{t_*'}
\bigr)_i-\bigl(\tx ^{t_*'}+\tA ^{\sT}
\tilde{z}^{t_*'}\bigr)_i \bigr)^2\ge
c_1^2\theta_{t_*'}^2 \bigr\} \nonumber\\
&&\qquad\le
o_1(t_*;\nu)+o_2\bigl(t_*,\nu;n^{-1}\bigr).\nonumber%
\end{eqnarray}
By
Proposition~\ref{171717171717}, and using the fact that the
vector $(X+Z_{t_*},X+Z_t)$ has a density, we have, in probability,
\begin{eqnarray*}
&&\lim_{n\to\infty} \frac{1}{n}\bigl\llvert
S_t(c_1)\setminus S_{t_*}(c_1)
\bigr\rrvert \\
&&\qquad= \prob \bigl\{\llvert X+Z_{t_*-1}\rrvert
\ge(1-c_1)\sigma_{t_*-1}; \llvert X+Z_{t-1}\rrvert <
(1-c_1)\sigma_{t-1} \bigr\}\le h(t_*), %
\end{eqnarray*}
where, by Lemma~\ref{6661666616666166616661}(a3), $h(t_*)$ vanishes as
$t_*\to\infty$.
Given any $c_1>0$, we can therefore choose $t_*$ so that, with high
probability $\llvert S_t(c_1)\setminus S_{t_*}(c_1)\rrvert  \le nc_1/2$. Combining
with equation (\ref{eq:StildeS}), we obtain the desired claim.

%
%
\subsection{Proof of Theorem \texorpdfstring{\protect\ref{thm:CS}}{8}, \texorpdfstring{$\rho>\rho_*(\delta)$}{$rho>rho_*(delta)$}}

Fix a small number $h>0$. By Lemma \ref{6661666616666166616661}(b),
there exists $\Delta= \Delta(\delta,\eps)>0$ independent of $h$, such
that, for
$\alpha= \alpha_{0}(\delta,p_X)$ and $t$ large enough,
%
\begin{eqnarray}\label{eq:LargeRho1}
\biggl\llvert \frac{1}{\delta} \prob \{|X+\sigma_tZ|>
\alpha\sigma _t \} - 1\biggr\rrvert &\le&h,
\\
\label{eq:LargeRho2}
\llvert R_{t,t}-2R_{t,t-1}+R_{t-1,t-2}\rrvert &
\le&h^2 ,
\\
\label{eq:LargeRho3}
\E\bigl\{\bigl\llvert \eta(X+\sigma_t Z;\alpha\sigma_t)
\bigr\rrvert \bigr\} &<& \E\bigl\{\llvert X\rrvert \bigr\}-2\Delta
,%
\end{eqnarray}
as well as $\sigma_{t-1}^2\le2\sigma_*^2$.
By Propositions \ref{1616161616}, \ref{171717171717}
(and noting that $X+\sigma_tZ$ has a distribution that is absolutely
continuous with respect to Lebesgue measure), we have, with high
probability,
%
\begin{eqnarray}\label{eq:LargeRho1B}
\max_{i\in[m]}\bigl\llvert (\mathsf{b}_t-
1)_{ii}\bigr\rrvert &\le&2h,
\\
\label{eq:LargeRho2B}
\bigl\llVert z^t-z^{t-1}\bigr\rrVert _2&\le&2h
\sqrt{n},
\\
\label{eq:SmallerEll1Norm}
\bigl\llVert x^t\bigr\rrVert _1&\le& \llVert
x_0\rrVert _1-n\Delta,
\\
\label{eq:LargeRho4B}
\bigl\llVert z^t\bigr\rrVert _2&\le& 2\sigma_*\sqrt{n}.%
\end{eqnarray}
Namely equation (\ref{eq:LargeRho1}) implies (\ref{eq:LargeRho1B}),
equation (\ref{eq:LargeRho2}) implies (\ref{eq:LargeRho2B}),
equation (\ref{eq:LargeRho3}) implies (\ref{eq:SmallerEll1Norm}) and the
assumption $\sigma_{t-1}^2\le2\sigma_*^2$ implies (\ref{eq:LargeRho4B}).

Using equation (\ref{eq:LemmaSE_1}) together with the above, we get
%
\begin{eqnarray}
\bigl\llVert y-Ax^t\bigr\rrVert _2&\le&\bigl
\llVert z^t-z^{t-1}\bigr\rrVert _2+\max
_{i\in[m]}\bigl\llvert (\mathsf{b}_t)_{ii}-1
\bigr\rrvert \bigl\llVert z^{t-1}\bigr\rrVert _2\nonumber\\[-8pt]\\[-8pt]
&\le& 2h
\sqrt{n} (1+2\sigma_*). \nonumber%
\end{eqnarray}
Define $\tx= x^t+ A^{\sT}(AA^{\sT})^{-1}(y-Ax^t)$ (notice that the
sample covariance matrix $AA^{\sT}$ has full rank with high
probability \cite{BaiSilversteinPaper,BaiSilverstein}).
Notice that, by construction \mbox{$A\tx=y$}.
Then, with high probability,
%
\begin{equation}
\hspace*{20pt}\bigl\llVert \tx-x^t\bigr\rrVert _2\le
\sigma_{\max}(A)\sigma_{\min}(A)^{-2}\bigl\llVert
y-Ax^t\bigr\rrVert _2\le C(\delta) (1+2\sigma_*) h
\sqrt{n},
\end{equation}
where $\sigma_{\max}(A)$, $\sigma_{\min}(A)$ are the maximum and
minimum nonzero singular values of $A$. The second inequality
holds with high probability for $\delta\in(0,1)$ by standard estimates
on the singular values of random matrices \cite
{BaiSilversteinPaper,BaiSilverstein}.
Using equation (\ref{eq:SmallerEll1Norm}) together with triangular
inequality and $\llVert \tx-x^t\rrVert _1\le\sqrt{n} \llVert \tx-x^t\rrVert _2$, we finally
get
%
\begin{eqnarray}
\llVert \tx\rrVert _1&\le& \llVert x_0\rrVert
_1-n\Delta+C(\delta) (1+2\sigma_*) h n <\llVert x_0
\rrVert _1,
\end{eqnarray}
where the second inequality follows from the fact that $h>0$ can be
taken arbitrarily small (by letting $t$ large) while $\Delta$, $C$ and
$\sigma_*$ are fixed. We conclude that $x_0$ cannot be the solution of
the $\ell_1$ minimization problem (\ref{eq:Ell1Min}).

\begin{appendix}
\section{Proofs of Propositions \texorpdfstring{\protect\ref{1616161616}}{6} and
\texorpdfstring{\protect\ref{171717171717}}{7}}\label{app:SEell1}

In this Appendix we prove Propositions~\ref{1616161616} and
\ref{171717171717} by a suitable application of Theorem~\ref
{thm:BipartiteSE}.
Before passing to these proofs, we establish a corollary of Theorem~\ref{thm:BipartiteSE} that allows us to control iterations of the form
(\ref{eq:AMPell1_1}), (\ref{eq:AMPell1_2}), with
$\eta( \cdot; \cdot)$ replaced by a general polynomial.

\subsection{A general corollary}\label{app:CorollaryBipartiteSE}

For $x_0 = x_0(n)\in\reals^n$ and $A=A(n)\in\reals^{m\times n}$ as per
Hypothesis~\ref{hyp1} in Section~\ref{sec:Polytope}, we define
$y=y(n)\in\reals^m$ by
%
\begin{equation}
y = Ax_0 .
\end{equation}
Let $D\in\reals^{n\times n}$ be the diagonal matrix with
diagonal entries equal to the square column norms of $A$, that is,
$D_{ii} = \sum_{j\in[m]}A_{ji}^2$, and $D_{ij}=0$ for $i\neq
j$. Further define $u_0=u_0(n)\in\reals^n$ as follows:
%
\begin{equation}
u_{0,i}= (D_{ii}-1)x_{0,i} = \biggl(
\sum_{j\in[m]}A_{ji}^2-1 \biggr)
x_{0,i}. %
\end{equation}
Let $x^0=(I-D^{-1})x_0$ (notice that $D$ is invertible with high
probability) and define iteratively
%
\begin{eqnarray}\label{eq:AMP-CS1}
z^t & = & y-Ax^t + \mathsf{b}_t
z^{t-1},\nonumber\\[-8pt]\\[-8pt]
 (\mathsf{b}_t)_{ii} &=& \sum
_{j\in[n]}A_{ij}^2\eta'_{t-1}
\bigl(D_{jj}x^{t-1}_j +\bigl(A^{\sT}z^{t-1}
\bigr)_j-u_{0,j} \bigr),\nonumber
\\
\label{eq:AMP-CS2}
x^{t+1} & = &\eta_t\bigl(Dx^t+A^{\sT}z^t-u_0
\bigr),%
\end{eqnarray}
where, for each $t$, $\eta_t\dvtx \reals\to\reals$ is a polynomial and,
for $v\in\reals^n$, $\eta_t(v) = (\eta_t(v_1),\break\dots,\eta_t(v_n))$.
Further $\mathsf{b}_t\in\reals^{m\times m}$ is a diagonal matrix with
entries given as in equation (\ref{eq:AMP-CS1}).

We next introduce the corresponding state evolution recursion. Namely,
we define $\{\tR_{s,t}\}_{s,t\ge0}$ recursively for all $s,t\ge0$ by letting
%
\begin{equation}\label{eq:TwoTimesPolyCS}
\tR_{s+1,t+1} = \frac{1}{\delta} \E \bigl\{\bigl[
\eta_s(X+Z_s)-X\bigr] \bigl[\eta_t(X+Z_t)-X
\bigr] \bigr\}. %
\end{equation}
Here expectation is with respect to $X\sim p_X$ and the independent
Gaussian vector $[Z_s,Z_t]$ with zero mean and covariance given by $\E
\{Z_s^2\}=\tR_{s,s}$,
$\E\{Z_t^2\}=\tR_{t,t}$\vspace*{1pt} and $\E\{Z_tZ_s\}= \tR_{t,s}$. The boundary
condition is fixed by letting $\tR_{0,0} =\E\{X^2\}/\delta$ and
defining, for each $t\ge0$,
%
\begin{equation}
\tR_{0,t+1} = \frac{1}{\delta} \E \bigl\{\bigl[
\eta_t(X+Z_t)-X\bigr] [-X] \bigr\}, %
\end{equation}
{\spaceskip=0.2em plus 0.05em minus 0.04em with $Z_t\sim\normal(0,\tR_{t,t})$. This uniquely determines the doubly
infinite array\break $\{\tR_{t,s}\}_{t,s\ge0}$.}

\begin{corollary}\label{coro:SE}
Let
$\{(x_0(n),A(n),y(n))\}_{n\ge0}$ be a sequence of
triples with $A(n)$ having independent sub-Gaussian entries with
$\E\{A_{ij}\}=0$, $\E\{A_{ij}^2\}=1/m$, $\{x_{0,i}(n)\dvtx  i\in[n]\}$
independent and identically distributed with $x_{0,i}(n)\sim
p_X$, and $p_X$ a finite mixture of Gaussians.
Define $\{x^t,z^t\}_{t\ge0}$ as per equations (\ref{eq:AMP-CS1}),
(\ref{eq:AMP-CS2}).

Then, for any fixed $t,s\ge0$, and any Lipschitz continuous functions
$\psi\dvtx \reals\times\reals\times\reals\to\reals$,
$\phi\dvtx \reals\times\reals\to\reals$,
in probability
%
\begin{eqnarray}\label{eq:CoroSE_A}
&&\lim_{n\to\infty}\frac{1}{n}\sum
_{i=1}^n\psi \bigl(x_{0,i},x^s_i+
\bigl(A^{\sT
}z^s\bigr)_i,x^t_i+
\bigl(A^{\sT}z^t\bigr)_i \bigr) \nonumber\\[-8pt]\\[-8pt]
&&\qquad= \E
\psi(X,X+Z_s,X+Z_t),\nonumber
\\
\label{eq:CoroSE_B}
&&\lim_{n\to\infty}\frac{1}{m}\sum
_{i=1}^n\phi\bigl(z^s_i,z^t_i\bigr) = \E\phi(Z_s,Z_t), %
\end{eqnarray}
where expectation is with respect to $X\sim p_X$ and the independent
Gaussian vector $[Z_s,Z_t]$ with zero mean and covariance given by $\E
\{Z_s^2\}=\tR_{s,s}$,
$\E\{Z_t^2\}=\tR_{t,t}$ and $\E\{Z_tZ_s\}= \tR_{t,s}$.
\end{corollary}

\begin{pf}
Define $\tx^{t+1} = A^{\sT}z^t+Dx^t-Dx_0$. Then equations (\ref{eq:AMP-CS1}),
(\ref{eq:AMP-CS2}) read
%
\begin{eqnarray}
z^t &=& A f\bigl(\tx^t,x_0;t
\bigr) +\mathsf{b}_t h\bigl(z^{t-1};t-1\bigr),
\\
\tx^{t+1} & = & A^{\sT} h\bigl(x^t;t\bigr) +
\mathsf{d}_t f\bigl(\tx^t,x_0;t\bigr),
\end{eqnarray}
where, for $i\in[m]$, $j\in[n]$,
%
\begin{eqnarray}
f(x,y;t) &=& y-\eta_{t-1}(y+x),\qquad h(z;t) = z,
\\
(\mathsf{b}_t)_{ii} &=& {-}\sum
_{j\in[n]} A_{ij}^2f'\bigl(
\tx_j^t,x_{0,i};t\bigr),
\\
(\mathsf{d}_t)_{jj} &=& {-}\sum
_{j\in[n]} A_{ij}^2h'(z;t).
\end{eqnarray}
[Here $f'(x,y;t)$, $h'(x;t)$ denote derivatives with respect to the
first argument.]
The iteration takes the same form as in equations (\ref{eq:BipartiteAMP1}),
(\ref{eq:BipartiteAMP2})
with $Y(i)= x_{0,i}$, and $W(i) = 0$, $\Ons_t=-\mathsf{b}_t$ and $\POns
_t =
-\mathsf{d}_t$.
Further, the initial condition $x^0$ implies $\tx^0= -x_0$. Notice
that this is dependent on $Y=x_0$, but we can easily set the initial
condition at $\tx^{-1} = 0$ and define $f(x,y;t=0)=-y$.
We can therefore apply Theorem~\ref{thm:BipartiteSE} and conclude
that, in probability,
%
\begin{eqnarray}
&&\lim_{n\to\infty}\frac{1}{n}\sum
_{i=1}^n\psi \bigl(x_{0,i},D_{ii}
\bigl(x^s_i-x_{0,i}\bigr)+\bigl(A^{\sT}z^s
\bigr)_i,D_{ii}\bigl(x^t_i-x_{0,i}
\bigr)+\bigl(A^{\sT}z^t\bigr)_i \bigr)\hspace*{-36pt}  \nonumber\\[-8pt]\\[-8pt]
&&\hspace*{25pt}\qquad= \E
\psi(X,Z_s,Z_t),\nonumber
\\
&&\lim_{n\to\infty}\frac{1}{m}\sum
_{i=1}^n\phi\bigl(z^s_i,z^t_i\bigr) = \E\phi(Z_s,Z_t),
\end{eqnarray}
where expectations are defined as in the statement of the corollary.
The second of these equations coincides with equation (\ref{eq:CoroSE_B}).
For the first one, note that $\E\{D_{ii}\}=1$ and, by a
standard Chernoff bound
%
\begin{eqnarray}
\lim_{n\to\infty}\max \bigl\{D_{ii}\dvtx  i\in[n]
\bigr\} &=& 1 ,
\\
\lim_{n\to\infty}\min \bigl\{D_{ii}\dvtx  i\in[n] \bigr\} &=&
1 . %
\end{eqnarray}
We therefore get
%
\begin{eqnarray}
&&\lim_{n\to\infty}\frac{1}{n}\sum
_{i=1}^n\psi \bigl(x_{0,i},
\bigl(x^s+A^{\sT
}z^s\bigr)_i-x_{0,i},\bigl(x^t_i+A^{\sT}z^t
\bigr)_i-x_{0,i} \bigr) \nonumber\\[-8pt]\\[-8pt]
&&\qquad= \E\psi(X,Z_s,Z_t),\nonumber %
\end{eqnarray}
which coincides with equation (\ref{eq:CoroSE_A}) after a redefinition of
the function $\psi$.
\end{pf}

\subsection{Proofs of Propositions \texorpdfstring{\protect\ref{1616161616}}{6} and \texorpdfstring{\protect\ref{171717171717}}{7}}

We will start by proving Proposition~\ref{1616161616}. Since
Proposition~\ref{171717171717} follows from the same
construction, we will
only point to the necessary modifications.
Before presenting the proof, we recall a basic result in weighted
polynomial approximation (here stated for a specific case); see, for
example, \cite{Lubinsky}.

\begin{theorem}\label{thm:WeightedApproximation}
Let $f\dvtx \reals\to\reals$ be a continuous function. Then for any
$\kappa,\xi>0$ there exists a polynomial $p\dvtx \reals\to\reals$ such
that, for all $x\in\reals$,
%
\begin{equation}
\bigl\llvert f(x)-p(x)\bigr\rrvert \le\xi e^{\kappa x^2/2}.
\end{equation}
\end{theorem}

\begin{pf*}{Proof of Propositions \ref{1616161616}}
Since the proposition holds as $n\to\infty$ at $t$ fixed, we shall
assume throughout that $t\in\{0,1,\dots,t_{\max}\}$ for some fixed
arbitrarily large $t_{\max}$.\vadjust{\goodbreak}

We claim that, for each $\beta, t_{\max}>0$, we can construct
an orbit $\{x^{\beta,t},z^{\beta,t}\}_{t\ge0}$ obeying equations
(\ref
{eq:AMP-CS1}),
(\ref{eq:AMP-CS2}) for suitable functions $\eta_t=\eta_t^{(\beta)}$
such that the
following holds (with a slight abuse of notation we will drop the
parameter $\beta$ from $x^{\beta,t},z^{\beta,t}$). For all $0\le
t\le
t_{\max}$, and
all functions $\psi$ as in the statement, we have $z^t = y-Ax^t+\mathsf{b}_t
z^{t-1}$ by construction. Further, in probability,
%
\begin{eqnarray}
\label{eq:FirstBeta}
\lim_{n\to\infty} \max_{i\in[m]}\biggl
\llvert (\mathsf{b}_t)_{ii} - \frac
{1}{\delta} \prob
\bigl\{\llvert X+\sigma_{t-1}Z\rrvert \ge\alpha\sigma_{t-1}
\bigr\}\biggr\rrvert &\le& \beta,
\\
\label{eq:SecondBeta}
\lim_{n\to\infty} \frac{1}{n} \bigl\llVert x^{t+1}-
\eta\bigl(x^t+A^{\sT}z^t;\alpha
\sigma_t\bigr)\bigr\rrVert _2^2 &\le&\beta,
\\
\label{eq:ThirdBeta}
\lim_{n\to\infty}\Biggl\llvert \frac{1}{n}\sum
_{i=1}^n\psi \bigl(x_{0,i},x^t_i+
\bigl(A^{\sT}z^t\bigr)_i\bigr) -\E\psi(X,X+
\sigma_tZ) \Biggr\rrvert &\le&\beta.%
\end{eqnarray}
Assuming this claim holds, let $\{\beta_{\ell}\}_{\ell\ge0}$ be a sequence
such that $\lim_{\ell\to\infty}\beta_\ell=0$. Denote by $\{
x^{\ell,t},z^{\ell,t}\}_{t\ge0}$ the orbit satisfying
equations (\ref{eq:FirstBeta}), (\ref{eq:SecondBeta}), (\ref{eq:ThirdBeta})
with $\beta= \beta_\ell$.
Let $\eta^{\ell}_t = \eta^{(\beta_{\ell})}_t$ be the corresponding
polynomial and
$\mathsf{b}_t^\ell$ be given per equation (\ref{eq:AMP-CS1}).
Fix an increasing sequence of instance sizes $n_1<n_2<n_3<\cdots$\,,
and let $x^t(n) = x^{\ell,t}(n)$, $z^t(n) = z^{\ell,t}(n)$ for all
$n_\ell\le n< n_{\ell+1}$.
Choosing $\{n_\ell\}_{\ell\ge0}$ that increases rapidly enough we
can ensure that, for all $n\ge n_{\ell}$,
%
\begin{eqnarray}\label{eq:FirstBetak}
\max_{i\in[m]}\biggl\llvert \bigl(
\mathsf{b}^{\ell}_t\bigr)_{ii} - \frac{1}{\delta}
\prob \bigl\{\llvert X+\sigma_{t-1}Z\rrvert \ge\alpha
\sigma_{t-1} \bigr\}\biggr\rrvert &\le& 2\beta_\ell,
\\
\label{eq:SecondBetak}
\frac{1}{n} \bigl\llVert x^{\ell,t+1}-\eta\bigl(x^{\ell,t}+A^{\sT}z^{\ell,t};
\alpha\sigma _t\bigr)\bigr\rrVert _2^2 &\le&2
\beta_\ell,
\\
\label{eq:ThirdBetak}
\Biggl\llvert \frac{1}{n}\sum_{i=1}^n
\psi\bigl(x_{0,i},x^{\ell,t}_i+\bigl(A^{\sT
}z^{\ell,t}
\bigr)_i\bigr) -\E\psi(X,X+\sigma_tZ) \Biggr\rrvert &\le&2
\beta_\ell,%
\end{eqnarray}
with probability larger than $1-\beta_{\ell}$.
Points
1, 2, 3 in the
proposition then follow since $\beta_\ell\to0$.

In order to prove equations (\ref{eq:FirstBeta}) to (\ref
{eq:ThirdBeta}) we
proceed as follows. It is easy to check that $\sigma_t>0$ for all
$t$; cf. equation (\ref{eq:ell1SE}). We use Theorem~\ref{thm:WeightedApproximation} to construct polynomials $\eta_t$ such
that
%
\begin{equation}\label{eq:EtaApprox}
\bigl\llvert \eta(x;\alpha\sigma_t)-
\eta_t(x)\bigr\rrvert \le\xi \exp \biggl\{\frac{x^2}{16\max(\sigma_t^2,s^2)} \biggr\},%
\end{equation}
for all $x\in\reals$. Here $\xi>0$ is a small parameter to be chosen
below, and $s^2$ is the smallest variance of the Gaussians that are
combined in $p_X$. Let $\tsigma_t$ be defined by
%
\begin{equation}\label{eq:ell1SEApprox}
\tsigma_{t+1}^2 = \frac{1}{\delta}\E\bigl\{
\bigl[\eta_t(X+\tsigma_t Z)-X\bigr]^2\bigr\},
\end{equation}
with $Z\sim\normal(0,1)$ independent from
$X\sim p_X$, and $\tsigma_0^2 = \E\{X^2\}/\delta$. Notice that
$\tsigma_t^2= \tR_{tt}$.
From equations (\ref{eq:ell1SE}),
(\ref{eq:EtaApprox}) and (\ref{eq:ell1SEApprox}), it is then
straightforward to show that
$\llvert \sigma_t^2-\tsigma_t^2\rrvert \le C \xi$ for some $C= C(t)$.

Given polynomials as defined by (\ref{eq:EtaApprox}), we define $\{
x^t,z^t\}_{t\ge0}$ as per
equations (\ref{eq:AMP-CS1}), (\ref{eq:AMP-CS2}), with the initial
condition given there. Equation (\ref{eq:ThirdBeta}) follows
immediately from Corollary~\ref{coro:SE} for $\xi$ sufficiently
small. Equation (\ref{eq:SecondBeta}) also follows from the same
corollary, by taking
%
\begin{equation}
\psi(x_1,x_2,x_3) = \bigl\{
\eta_t(x_3)-\eta(x_3;\alpha
\sigma_t) \bigr\}^2 , %
\end{equation}
and then using once again equation (\ref{eq:EtaApprox}) on the resulting
expression.

Finally, consider equation (\ref{eq:FirstBeta}). For economy of notation,
we write
%
\begin{equation}
\hspace*{30pt}(\mathsf{b}_t)_{ii} = \sum
_{j\in[n]}A_{ij}^2\varphi_j,\qquad
\varphi_i=\eta'_{t-1}\bigl(D_{jj}x^{t-1}_j
+\bigl(A^{\sT}z^{t-1}\bigr)_j-u_{0,j}
\bigr), %
\end{equation}
and further define
%
\begin{equation}
\mathsf{b}_t^{\mathrm{av}} =\frac{1}{m} \sum
_{j\in[n]}\varphi_j. %
\end{equation}
Then we have
\begin{eqnarray*}
&&\E \bigl\{ \bigl((\mathsf{b}_t)_{ii}-
\mathsf{b}_t^{\mathrm{av}} \bigr)^4 \bigr\} \\[-2pt]
&&\qquad= \sum
_{j_1,j_2,j_3,j_4\in[n]} \E \biggl\{ \biggl(A_{ij_1}^2-
\frac
{1}{m} \biggr) \biggl(A_{ij_2}^2-
\frac{1}{m} \biggr) \biggl(A_{ij_3}^2-
\frac{1}{m} \biggr) \\[-2pt]
&&\hspace*{142pt}{}\times\biggl(A_{ij_4}^2-
\frac{1}{m} \biggr) \varphi_{j_1}\varphi_{j_2}\varphi
_{j_3}\varphi_{j_3} \biggr\}
\\[-2pt]
&&\qquad=\sum_{j_1,j_2,j_3,j_4\in[n]} E(j_1,j_2,j_3,j_4).
\end{eqnarray*}
Using the tree representation in Section~\ref{sec:TreeRep}, it is not
hard to prove that the expectation on the right-hand side is bounded
as follows:
\begin{eqnarray*}
E(p,q,r,s) &\le&\frac{K}{n^6},\qquad p,q,r,s \mbox{ distinct,}
\\[-1pt]
E(q,q,r,s) &\le&\frac{K}{n^5},\qquad q,r,s \mbox{ distinct,}
\\[-1pt]
E(r,r,s,s) &\le&\frac{K}{n^4},\qquad r,s \mbox{ distinct,}
\\[-1pt]
E(r,r,r,s) &\le&\frac{K}{n^4},\qquad r,s \mbox{ distinct,}
\\[-1pt]
E(r,r,r,r) &\le&\frac{K}{n^3}. %
\end{eqnarray*}
Consider, for instance, the first case, $p$, $q$, $r$, $s$ distinct.
Using Lemma~\ref{lem:treeAMP}, each of $\varphi_p$, $\varphi_q$,
$\varphi_r$, $\varphi_s$ can be
represented as a sum over trees with root\vadjust{\goodbreak} type respectively at $p$,
$q$, $r$,
$s$. The weight of these trees is as in Lemma~\ref{lem:treeAMP}, times
the prefactor $(A_{ip}^2-m^{-1})\cdots(A_{is}^2-m^{-1})$. Let $\mu$
be the total number of edges in these trees, plus $8$ (two for each of
the additional factors). Then any nonvanishing contribution is of
order $n^{-\mu/2}$. Let $\vG$
be the graph obtained by identifying the vertices of the same type in
these trees, and $e(\vG)$ the number of its edges. Since each edge in
$\vG$ must be covered at least twice by the trees to get a nonzero
expectation, and the edges in $(i,p),\dots,(i,s)$ at least once, we
have $2e(\vG)+4\le\mu$.
The number of vertices in $\vG$ is
at most $e(\vG)+1$ (note that $\vG$ is connected because it includes
type $i$ connected to $p,q,r,s$). Of these vertices all but $5$
(whose type is $i$, $p$, $q$, $r$, $s$) can take an arbitrary type,
yielding a combinatorial factor of order $n^{e(\vG)+1-5}\le
n^{\mu/2-6}$. Hence the sum over trees is of order
$n^{-\mu/2}n^{\mu/2-6}= n^{-6}$ as claimed.

Summing over $j_1$, \dots, $j_4$ the above bounds we obtain $\E \{
((\mathsf{b}_t)_{ii}-\mathsf{b}_t^{\mathrm{av}}
)^4 \} \le K/n^2$ and therefore, by Markov's inequality
%
\begin{equation}
\lim_{n\to\infty} \prob \Bigl\{\max_{i\in[m]}
\bigl\llvert (\mathsf{b}_t)_{ii}-\mathsf{b}_t^{\mathrm
{av}}
\bigr\rrvert \ge n^{-1/5} \Bigr\} = 0 . %
\end{equation}
Since by standard concentration bounds $\max_{i\in[n]}D_{ii}$,
$\min_{i\in[n]}D_{ii}\to1$, we obtain, in probability,
\begin{eqnarray*}
\lim_{n\to\infty} \max_{i\in[m]} (
\mathsf{b}_t)_{ii} &=& \lim_{n\to
\infty} \min
_{i\in[m]} (\mathsf{b}_t)_{ii} = \lim
_{n\to\infty}\mathsf{b}_t^{\mathrm{av}}
\\
&=& \lim_{n\to\infty} \frac{1}{m}\sum
_{j\in[n]}\eta'_{t-1}\bigl(x^{t-1}_j
+\bigl(A^{\sT}z^{t-1}\bigr)_j\bigr)
\\
&=&\frac{1}{\delta}\E \bigl\{\eta'_{t-1}(X+
\tsigma_{t-1} Z) \bigr\},%
\end{eqnarray*}
where, in the last step, we applied Corollary~\ref{coro:SE} to the polynomials $\eta'_{t-1}$, and $X\sim p_X$,
$Z\sim\normal(0,1)$ are independent.
We are left with the task of showing that, by taking $\xi$ small
enough in equation (\ref{eq:EtaApprox}), we can ensure that
%
\begin{equation}
\bigl\llvert \E \bigl\{\eta'_{t-1}(X+
\tsigma_{t-1} Z) \bigr\}-\prob\bigl\{\llvert X+\sigma_{t-1} Z
\rrvert \ge \alpha\sigma_{t-1}\bigr\}\bigr\rrvert \le\beta\delta.
\end{equation}
Indeed integrating by parts with respect to $Z$ the above difference
can be written as (for $K$ a finite constant that can depend on $t$ and
change from line to line)
\begin{eqnarray*}
&&\biggl\llvert \frac{1}{\tsigma_{t-1}}\E \bigl\{Z\eta_{t-1}(X+
\tsigma_{t-1} Z) \bigr\}-\frac{1}{\sigma_{t-1}}\E \bigl\{Z\eta(X+
\tsigma_{t-1} Z;\alpha\sigma_{t-1}) \bigr\}\biggr\rrvert
\\
&&\qquad\le K \E\bigl\llvert Z\eta_{t-1}(X+\sigma_{t-1} Z)-Z\eta(X+
\sigma_{t-1} Z;\alpha \sigma_{t-1})\bigr\rrvert + K\llvert
\sigma_{t-1}-\tsigma_{t-1}\rrvert
\\
&&\qquad\le K \xi\E \biggl\{\exp \biggl\{\frac{X^2+\sigma_{t-1}^2X^2}{4\max
(\sigma_t^2,s^2)} \biggr\} \biggr\} +K\llvert
\sigma_{t-1}-\tsigma_{t-1}\rrvert
\\
&&\qquad\le K\xi+K\llvert \sigma_{t-1}-\tsigma_{t-1}\rrvert.
\end{eqnarray*}
The claim follows by
noting that, as argued above $\llvert \sigma_{t-1}-\tsigma_{t-1}\rrvert  \le K'\xi$.\vadjust{\goodbreak}

Consider finally point 4. First recall that we constructed the
vectors $\{x^t,z^t\}_{t\ge0}$, using a sequence of orbits $\{x^{\ell,t},z^{\ell,t}\}_{t\ge
0}$, indexed by $\ell\in\naturals$, that obey equations (\ref{eq:AMP-CS1}),
(\ref{eq:AMP-CS2}), and letting
%
\begin{equation}\label{eq:ConstructionXt}
\hspace*{35pt}x^t(n) = x^{\ell,t}(n),\qquad z^t(n) =
z^{\ell,t}(n) \qquad \mbox{for all $n$, with $n_{\ell}\le
n<n_{\ell+1}$}. %
\end{equation}
%

\begin{cla}\label{Claim:Ell}
There exists a sequence $\{\tbeta_{\ell}\}_{\ell\in\naturals}$ with
$\lim_{\ell\to\infty}\tbeta_{\ell}=0$
such that, for all $\ell'\ge\ell$,
%
\begin{eqnarray}\label{eq:ClaimEll1}
\lim_{n\to\infty}\frac{1}{n}\sum
_{i\in
[n]}\E \bigl\{ \bigl(\bigl(x^{\ell',t}+A^{\sT}z^{\ell',t}
\bigr)_i-\bigl(x^{\ell,t}+A^{\sT
}z^{\ell,t}
\bigr)_i \bigr)^2 \bigr\}&\le&\tbeta_{\ell},
\\
\lim_{n\to\infty}\frac{1}{m}\sum
_{i\in[m]}\E \bigl\{ \bigl(z^{\ell
',t}_i-z^{\ell,t}_i
\bigr)^2 \bigr\}&\le& \tbeta_{\ell}.
\end{eqnarray}
\end{cla}

The proof of this claim is presented below.
It follows from this claim that, by
eventually redefining $n_{\ell'}$ to be larger, we can ensure
\begin{eqnarray*}
\E \bigl\{ \bigl(\bigl(x^{\ell',t}+A^{\sT}z^{\ell',t}
\bigr)_I-\bigl(x^{\ell,t}+A^{\sT
}z^{\ell,t}
\bigr)_I \bigr)^2 \bigr\}&\le&2\tbeta_{\ell},
\\
\E \bigl\{ \bigl(z^{\ell',t}_J-z^{\ell,t}_J
\bigr)^2 \bigr\}&\le& 2\tbeta_{\ell} %
\end{eqnarray*}
for all $n\ge n_{\ell'}$. Here and below, expectation is taken also
with respect to $I$ uniformly random in $[n]$ and $J$ uniformly random
in $[m]$.
By equation (\ref{eq:ConstructionXt}), for all $n\ge n_{\ell}$,
we also have
\begin{eqnarray*}
\E \bigl\{ \bigl(\bigl(x^t+A^{\sT}z^t
\bigr)_I-\bigl(x^{\ell,t}+A^{\sT}z^{\ell,t}
\bigr)_I \bigr)^2 \bigr\}&\le&2\tbeta_\ell,
\\
\E \bigl\{ \bigl(z^t_J-z^{\ell,t}_J
\bigr)^2 \bigr\}&\le& 2\tbeta_\ell. %
\end{eqnarray*}
Applying Lemma~\ref{lemma:Continuity}, we can then construct
$\{\tx^t,\tilde{z}^t\}_{t\ge0}$ as in the statement at point 4, such that
\begin{eqnarray*}
\E \bigl\{ \bigl(\bigl(\tx^t+\tA^{\sT}
\tilde{z}^t\bigr)_I-\bigl(x^{\ell,t}+A^{\sT
}z^{\ell,t}
\bigr)_I \bigr)^2 \bigr\}&\le&K \bigl(
\nu^2+n^{-1/2}\bigr),
\\
\E \bigl\{ \bigl(\tilde{z}^t_J-z^{\ell,t}_J
\bigr)^2 \bigr\}&\le& K \bigl(\nu^2+n^{-1/2}
\bigr), %
\end{eqnarray*}
where $K$ depends on $\ell$ but not on $\nu$ or $n$.
The proof is finished by using triangular inequality and selecting
$\ell=\ell(\nu,t)$ diverging slowly enough as $\nu\to0$.
\end{pf*}

We now prove Claim~\ref{Claim:Ell}.

\begin{pf*}{Proof of Claim~\ref{Claim:Ell}}
To be definite we will focus on equation (\ref{eq:ClaimEll1}).

Fix $\ell$, $\ell'\in\naturals$ (not necessarily distinct).
By an immediate generalization of Corollary~\ref{coro:SE}, we have, in
probability,
%
\begin{equation}\label{eq:Qdef}
\hspace*{30pt}\lim_{n\to\infty}\frac{1}{n}\sum
_{i\in[n]}\E\bigl\{ \bigl(x^{\ell,t}+A^{\sT}z^{\ell,t}-x_0
\bigr)_i \bigl(x^{\ell',t}+A^{\sT}z^{\ell',t}-x_0
\bigr)_i\bigr\}= Q^t_{\ell,\ell'}.\vadjust{\goodbreak}
\end{equation}
Further, the quantities $Q^t_{\ell,\ell'}$ satisfy the state evolution
recursion
%
\begin{equation}
Q^{t+1}_{\ell,\ell'} = \frac{1}{\delta}\E \bigl\{
\bigl[\eta^\ell_t(X+Z_{t,\ell})-X \bigr] \bigl[\eta
^{\ell'}_t(X+Z_{t,\ell'})-X \bigr] \bigr\},
\end{equation}
with initial condition $Q^0_{\ell,\ell'} = (1/\delta)\E\{X^2\}$.
Here expectation is taken with respect to $X\sim p_X$ and the
independent centered Gaussian vector
$(Z_{t,\ell},Z_{t,\ell'})$ with covariance given by $\E\{Z_{\ell,t}^2\} =
Q^t_{\ell,\ell}$, $\E\{Z_{\ell',t}^2\} = Q^t_{\ell',\ell'}$, $\E
\{
Z_{\ell,t}Z_{\ell',t}\} =
Q^t_{\ell,\ell'}$. In order to prove the claim, it is therefore
sufficient to show that
%
\begin{equation}\label{eq:ProofClaimEll}
\lim_{\ell\to\infty}\sup_{\ell': \ell'\ge\ell} \bigl
\llvert Q_{\ell,\ell'}^t-\sigma_t^2\bigr
\rrvert = 0 ,%
\end{equation}
since this implies $\lim_{\ell\to\infty}\sup_{\ell': \ell'\ge
\ell}[Q_{\ell,\ell}^{t}-2Q^{t}_{\ell,\ell'}+Q^{t}_{\ell',\ell'}] =
0$, which in turn implies the claim via equation (\ref{eq:Qdef}).

Finally, recall that $\eta_t^{\ell}$ was constructed using Theorem~\ref{thm:WeightedApproximation} [cf. equation (\ref{eq:EtaApprox})]
in such
a way that, for all $x\in\reals$,
%
\begin{equation}\label{eq:EtaApproxBis}
\bigl\llvert \eta(x;\alpha\sigma_t)-
\eta^{\ell}_t(x)\bigr\rrvert \le\xi_{\ell} \exp \biggl
\{\frac{x^2}{16\max(\sigma_t^2,s^2)} \biggr\},%
\end{equation}
with $\xi_{\ell}\to0$ as $\ell\to\infty$. The desired estimate
(\ref
{eq:ProofClaimEll}) then follows by
recalling that $\sigma^2_{t+1} =
(1/\delta)\E \{[\eta(X+\sigma_tZ)-X]^2 \}$, and using
equation (\ref{eq:EtaApproxBis}) inductively to show that
$ |Q_{\ell,\ell'}^t-\sigma_t^2 | \le K(t) \xi_{\ell}$.
\end{pf*}

We finally sketch the proof of Proposition~\ref{171717171717}.

\begin{pf*}{Proof of Proposition~\ref{171717171717}}
The sequence $\{x^t,z^t\}_{t\ge0}$ is constructed as in the previous
statement. The proof hence follow by using Corollary~\ref{coro:SE},
and taking $\xi$ small enough in equation (\ref{eq:EtaApprox}), since
we can
ensure that $\llvert \tR_{t,s}-R_{t,s}\rrvert \le\beta'$ for any $\beta'>0$ and
any $t,s\le t_{\max}$ (as shown above for the case $t=s$).
\end{pf*}\vspace*{-9pt}

\section{Proof of Lemma \texorpdfstring{\protect\ref{5551555515555}}{5}}\label{app:ProofSubgradient}

Throughout the proof we denote by $C_1$, $C_2$, $C_3$ etc., positive
constants that depend uniquely on $c_1,\dots, c_3$.

Consider the $\ell_1$ minimization problem
\begin{eqnarray*}
&& \mbox{minimize}\qquad \llVert x\rrVert _1 ,
\\
&& \qquad\mbox{subject to}\qquad y = Ax_0 ,
\end{eqnarray*}
and denote by $\hx$ any minimizer. Further, let $v$ be a subgradient
as in the statement, and define, for some $c\in(0,1)$,
%
\begin{equation}
S(c) \equiv \bigl\{ i\in[n] \dvtx  \llvert v_i\rrvert \ge1-c
\bigr\}. %
\end{equation}
Also, let $\oS(c) = [n]\setminus S(c)$ be the complement of this set.
Notice that, by definition of subgradient, we have
$v_i = \sign(x_{0,i})$ for all $i\in S$ and $\llvert v_{0,i}\rrvert \le1$ for all
in $\oS\equiv[n]\setminus S$. This implies that $S\subseteq S(c)$.

We have
%
\begin{eqnarray}\label{eq:LemmaCSFirst}
\llVert \hx\rrVert _1 &=& \llVert x_0\rrVert
_1+ \bigl\langle v,(\hx-x_0) \bigr\rangle +
R_1 + R_2 ,
\\
R_1 & =&\llVert \hx_{S(c)}\rrVert _1-\llVert
x_{0,S(c)}\rrVert _{1} - \bigl\langle v_{S(c)}, (
\hx-x_0)_{S(c)} \bigr\rangle,
\\
R_2 & =& \llVert \hx_{\oS(c)}\rrVert _1-\llVert
x_{0,\oS(c)}\rrVert _{1} - \bigl\langle v_{\oS(c)}, (
\hx-x_0)_{\oS(c)} \bigr\rangle. %
\end{eqnarray}
Since $\oS(c)\subseteq\oS$, we have $x_{0,\oS(c)}=0$ and hence
%
\begin{eqnarray}
R_2&=& \llVert \hx_{\oS(c)}\rrVert _1-
\langle v_{\oS(c)},\hx_{\oS
(c)} \rangle = \sum
_{i\in\oS(c)} \bigl(\llvert \hx_i\rrvert
-v_i \hx_i \bigr)\nonumber\\[-8pt]\\[-8pt]
 &\ge&\sum_{i\in\oS(c)}
\bigl(\llvert \hx_i\rrvert -(1-c)\llvert \hx_i\rrvert
\bigr) = c\llVert \hx_{\oS(c)}\rrVert _1 .\nonumber %
\end{eqnarray}
On the other hand, $v_{S(c)}$ is in the subgradient of
$\llVert x_{S(c)}\rrVert _1$ at $x_{S(c)}=x_{0,S(c)}$. Hence $R_1\ge0$. It follows
that equation (\ref{eq:LemmaCSFirst}) implies $\llVert \hx\rrVert _1\ge\llVert x_0\rrVert _1+
\langle v,(\hx-x_0) \rangle +c\llVert \hx_{\oS(c)}\rrVert _1$.
Since $\hx$ is a minimizer, we
thus get
%
\begin{equation}\label{eq:BoundOutOfSupport}
\qquad \llVert \hx_{\oS(c)}\rrVert _1\le -
\frac{1}{c} \bigl\langle v,(\hx-x_0) \bigr\rangle =-
\frac
{1}{c} \bigl\langle w,(\hx-x_0) \bigr\rangle \le
\frac{\eps}{c} \sqrt{n} \llVert \hx-x_0\rrVert _2
,%
\end{equation}
where in the last step we used Cauchy--Schwarz together with
assumption \ref{H:Grad}. Hereafter we let
$r\equiv\hx-x_0$.

Let $\oS(c) = \bigcup_{\ell=1}^KS_\ell$ be a partition such that
$nc/2\le
\llvert S_{\ell}\rrvert \le nc$, and that $\llvert r_{i}\rrvert \le\llvert r_{j}\rrvert $ for each $i\in
S_{\ell}$, $j\in S_{\ell-1}$. If $\llvert \oS(c)\rrvert <nc/2$, such a partition
does not
exist, but the argument follows by an
obvious modification of the one below. Further define $\oS_+ =
\bigcup_{\ell=2}^KS_{\ell}\subseteq\oS(c)$
and $S_+=[n]\setminus\oS_+$. We have
%
\begin{eqnarray}
\llVert r_{\oS_+}\rrVert _2^2 &=& \sum
_{\ell= 2}^K \llVert r_{S_\ell}\rrVert
_2^2\le \sum_{\ell= 2}^K
\llvert S_{\ell}\rrvert \biggl(\frac{\llVert r_{S_{\ell-1}}\rrVert _1}{\llvert S_{\ell-1}\rrvert } \biggr)^2
\nonumber\\[-8pt]\\[-8pt]
&\le& \frac{4}{nc}\sum_{\ell= 1}^{K-1}
\llVert r_{S_{\ell}}\rrVert _1^2\le
\frac{4}{nc}\llVert r_{\oS(c)}\rrVert _1^2 .\nonumber
\end{eqnarray}
Fix $c= c_1$.
Since $\oS(c)\subseteq\oS$, we have $r_{\oS(c)} = \hx_{\oS(c)}$, and
using equation (\ref{eq:BoundOutOfSupport}), we conclude that there exists
$C_1 \le4/c_1^3$ such that
%
\begin{equation}\label{eq:LemmaCS2}
\llVert r_{\oS_+}\rrVert _2^2\le
C_1 \eps^2\llVert r\rrVert _2^2
. %
\end{equation}
On the other hand, by definition $Ar = 0$, and hence $A_{S_+}r_{S_+} +
A_{\oS_+}r_{\oS_+}=0$.
Since $\oS(c)\subseteq\oS$, we have
$S\subseteq S(c) \subseteq S_+$. Further $S_+\setminus S(c) = S_1$,
whence $\llvert S_+\setminus S(c)\rrvert  \le n c = nc_1$.
By assumption \ref{H:Key}, we have $\sigma_{\min}(A_{S_+})\ge c_2$
and therefore,
\[
\llVert r_{S_+}\rrVert _2\le\frac{1}{c_2}
\llVert A_{S_+}r_{S_+}\rrVert _2=
\frac{1}{c_2} \llVert A_{\oS_+}r_{\oS_+}\rrVert
_2\le \frac{c_3}{c_2}\llVert r_{\oS_+}\rrVert
_2 . %
\]
Combining this with equation (\ref{eq:LemmaCS2}), we deduce that
$\llVert r\rrVert _2\le C_2\eps\llVert r\rrVert _2$ for some $C_2 = C_2(c_1,c_2,c_3)$,
which in turns implies $r=0$ provided that $C_2\eps<1$. The
claim hence follows for $\eps_0 = 1/[2C_2(c_1,c_2,c_3)]$.

%
%
\section{Asymptotic analysis of state evolution: Proof~of~Lemma
\texorpdfstring{\protect\ref{6661666616666166616661}}{6}}\label{app:AsympSE}

Before proceeding, we introduce the following piece of notation
(following \cite{BayatiMontanariLASSO}).
Fix a probability distribution $p_X$ on $\reals$, with $p_X(\{0\}
)=1-\eps$, and $\delta>0$.
For $\theta,\sigma^2>0$, we
define
%
\begin{equation}\label{eq:SEMapL1}
\F\bigl(\sigma^2,\theta\bigr)\equiv\frac{1}{\delta} \E \bigl\{
\bigl[\eta(X+\sigma Z;\theta)-X \bigr]^2 \bigr\},
\end{equation}
where expectation is taken with respect to the independent random
variables $X\sim p_X$ and $Z\sim\normal(0,1)$.
When necessary, we will indicate the dependency on $p_X$ by $\F(\sigma
^2,\theta;p_X)$.
With this notation the state evolution recursion reads $\sigma^2_{t+1}
= \F(\sigma^2_t,\alpha\sigma_t)$. The following properties of the
function $\F$ were proved in \cite{DMM09}; see also
\cite{BayatiMontanariLASSO}, Appendix A, for a more explicit
treatment.

\begin{lemma}[(\cite{DMM09})]\label{lemma:FirstDerivativeF}
For any $\alpha>0$, the mapping $\sigma^2\mapsto
\F(\sigma^2,\alpha\sigma)$ is monotone increasing and concave with
$\F(0,0) = 0$ and
%
\begin{equation}
\frac{\de}{\de(\sigma^2)} \F\bigl(\sigma ^2,\alpha \sigma\bigr)
\bigg\rrvert _{\sigma=0} = \frac{1}{\delta} \bigl\{\eps\bigl(1+
\alpha^2\bigr) + 2(1-\eps) \E \bigl[(Z-\alpha)_+^2 \bigr]
\bigr\}.
\end{equation}
\end{lemma}

It is also convenient to define
%
\begin{eqnarray}\label{eq:GExpression}
G_{\eps}(\alpha) & \equiv&\eps\bigl(1+\alpha^2\bigr) + 2(1-
\eps) \E \bigl\{(Z-\alpha)_+^2 \bigr\}\nonumber\\[-8pt]\\[-8pt]
& =& \eps\bigl(1+\alpha^2\bigr) + 2(1-\eps) \bigl[\bigl(1+
\alpha^2\bigr)\Phi(-\alpha)-\alpha\phi(\alpha) \bigr] .
\nonumber
\end{eqnarray}
The first two derivatives of $\alpha\mapsto G_{\eps}(\alpha)$ will
be used in the proof
%
\begin{eqnarray}\label{eq:FirstG}
G'_{\eps} (\alpha) &= & 2\alpha\eps+4(1-\eps) \bigl[-\phi(
\alpha)+\alpha\Phi(-\alpha ) \bigr] ,
\\
\label{eq:SecondG}
G''_{\eps} (\alpha) &= & 2\eps+4(1-\eps)
\Phi(-\alpha).
\end{eqnarray}
In particular, we have the following.

\begin{lemma}\label{lemma:G}
For any $\eps\in(0,1)$, $\alpha\mapsto G_{\eps}(\alpha)$ is
strictly convex
in $\alpha\in\reals_+$, with a unique minimum on $\alpha_*(\eps
)\in
(0,\infty)$. Further $G_{\eps}(0) = 1$ and\break $\lim_{\alpha\to\infty
}G_{\eps}(\alpha)=\infty$.
Finally, the minimum value satisfies
%
\begin{equation}\label{eq:MinValueG}
G_{\eps}(\alpha_*) = \eps+ 2(1-\eps)\Phi(-\alpha_*) =
\tfrac{1}{2} G''_{\eps}(\alpha_*)
\in(0,1).%
\end{equation}
\end{lemma}

\begin{pf}
By inspection of equation (\ref{eq:SecondG}), $G''_{\eps}(\alpha)>0$ for
all $\alpha>0$, hence $G_{\eps}(\alpha)$ is strictly convex.
Further, from equation (\ref{eq:FirstG}), we have $G'_{\eps}(0) =
-4(1-\eps)\phi(0)<0$
and $G'_{\eps}(\alpha) = 2\alpha\eps+ O_\alpha(1)>0$ as
$\alpha\to\infty$. Hence $\alpha\mapsto G_{\eps}(\alpha)$ has a unique
minimum $\alpha_*(\eps)\in(0,\infty)$.

Finally, equation (\ref{eq:MinValueG}) follows immediately by using the
condition\break $G_{\eps}'(\alpha_*)=0$ in expression (\ref{eq:GExpression}).
\end{pf}

In our proof it is more convenient to use the coordinates
$(\delta,\eps)$ instead of $(\rho,\delta)$. In terms of the latter,
the phase boundary (\ref{eq:PhaseBoundaryApp1}), (\ref
{eq:PhaseBoundaryApp2}) reads
%
\begin{eqnarray}\label{eq:DeltaStar}
&&\delta_*(\eps) = \frac{2\phi(\alpha_*(\eps))}{\alpha_*(\eps)+2[\phi(\alpha
_*(\eps))
-\alpha_*(\eps)\Phi(-\alpha_*(\eps))]},
\\
\label{eq:AlphaStar}
&&\alpha_*(\eps) \quad \mbox{solves}\quad \alpha\eps+2(1-\eps) \bigl[\alpha\Phi(-\alpha)-
\phi(\alpha) \bigr] = 0 .%
\end{eqnarray}
Notice that the use of the symbol $\alpha_*(\eps)$ in the last
equations is not an abuse of notation. Indeed comparing
equation (\ref{eq:AlphaStar}) with (\ref{eq:FirstG}) we conclude that
$\alpha_*(\eps)$ is indeed the unique solution of $G'_{\eps}(\alpha)
= 0$.
Further, comparing equation (\ref{eq:DeltaStar}) with
equation (\ref{eq:GExpression}) we obtain the following.

\begin{lemma}\label{lemma:BoundaryExpression}
Let $(\delta,\rho_*(\delta))$ be the phase boundary defined by
equations (\ref{eq:PhaseBoundaryApp1}), (\ref{eq:PhaseBoundaryApp2}). Then,
for $\rho,\delta\in[0,1]$, $\rho>\rho_*(\delta)$ if and only if,
for $\eps\in(0,1)$, $\delta\in(\eps,1)$
%
\begin{equation}
\delta<\delta_*(\eps) \equiv\min_{\alpha>0}
G_{\eps}(\alpha). %
\end{equation}
Vice versa $\rho<\rho_*(\delta)$ if and only if
$\delta>\delta_*(\eps)$.
\end{lemma}

\subsection{Proof of Lemma \texorpdfstring{\protect\ref{6661666616666166616661}(a)}{6(a)}:
\texorpdfstring{$\rho<\rho_*(\delta)$}{$rho<rho_*(delta)$}}
\mbox{}

\begin{pf*}{Proof of Lemma~\ref{6661666616666166616661}(\normalfont{a1})}
We set $\alpha= \alpha_*(\eps) \equiv\arg\min_{\alpha\ge0}
G_{\eps}(\alpha)$. Hence we have, by Lemma~\ref{lemma:FirstDerivativeF}, and Lemma~\ref{lemma:BoundaryExpression},
%
\begin{equation}
\frac{\de}{\de(\sigma^2)} \F\bigl(\sigma ^2,\alpha _*
\sigma\bigr)\bigg\rrvert _{\sigma^2=0} = \frac{1}{\delta} \min
_{\alpha>0}G_{\eps}(\alpha) = \frac
{\delta_*(\eps
)}{\delta}.
\end{equation}
In particular, by Lemma~\ref{lemma:BoundaryExpression}, for
$\rho<\rho_*(\delta)$, we have
$\frac{\de}{\de(\sigma^2)} \F(\sigma^2,\alpha
_*\sigma
)\equiv
\omega_*(\eps,\delta)\in(0,1)$.
Since, by Lemma~\ref{lemma:FirstDerivativeF}, $\sigma^2\mapsto
\F(\sigma^2,\alpha_*\sigma)$ is concave, it follows that
$\sigma^2_t = B \omega_*^t[1+o_t(1)]$.

Let $S \equiv\{ \alpha\in\reals_+ \dvtx G_{\eps}(\alpha)/\delta<1\}$.
Since $\alpha\mapsto G_{\eps}(\alpha)$ is strictly convex by Lemma~\ref{lemma:G}, with $G_{\eps}(0),G_{\eps}(\infty)>\delta$, we have
$S = (\alpha_1,\alpha_2)$ with $0<\alpha_1<\alpha_*<\alpha
_2<\infty$.
Let $\omega(\alpha) = G_{\eps}(\alpha)/\delta$.
Fixing $\alpha\in(\alpha_1,\alpha_2)$, by concavity of $\sigma
^2\mapsto
\F(\sigma^2,\alpha\sigma)$, we have $\sigma_t^2=B
\omega(\alpha)^t[1+o_t(1)]$. Finally, by continuity of $\alpha
\mapsto
G_{\eps}(\alpha)$, we have $\{\omega(\alpha)\dvtx  \alpha\in(\alpha
_1,\alpha
_2)\} = [\omega_*,1)$
and hence any rate $\omega\in[\omega_*,1)$ can be realized.

Finally by Lemma~\ref{lemma:G}
$G_\eps(\alpha_*)\equiv\eps+2(1-\eps)\Phi(-\alpha_*)<\delta$. Since
$\alpha\mapsto\eps+2(1-\eps)\Phi(-\alpha)$ is decreasing in
$\alpha$,
the last claim follows.
\end{pf*}

In the proof of part (a2) we will make use of the following
analytical result.

\begin{lemma}\label{lemma:LimitFalphEps}
For $\eps\in(0,1)$, $\alpha\ge\alpha_*(\eps)$, consider the function
$\cF_{\alpha,\eps}\dvtx [0,1]\to\reals$ defined by
%
\begin{equation}\label{eq:fQ}
\hspace*{35pt}\cF_{\alpha,\eps}(Q) \equiv\frac{1}{G_{\eps}(\alpha)} \E \bigl\{\bigl[
\eta(X_{\infty}+Z_1;\alpha)-X_{\infty}\bigr] \bigl[
\eta(X_{\infty
}+Z_2;\alpha)-X_{\infty}\bigr] \bigr\},
\end{equation}
where expectation is taken with respect to $X_{\infty}$,
$\prob\{X_{\infty} = 0\} = 1-\eps$, $\prob\{X_{\infty} \in
\{+\infty,-\infty\}\} = \eps$,
and the independent Gaussian vector $(Z_1,Z_2)$ with mean zero and
covariance $\E\{Z_1^2\}= \E\{Z_2^2\} = 1$, $\E\{Z_1Z_2\}=Q$.
(The mapping $x\mapsto[\eta(x+a;b)-x]$ is here extended to $x=
+\infty, -\infty$ by continuity for any $a,b$ bounded.)

Then $\cF_{\alpha,\eps}$ is increasing and convex on $[0,1]$ with
$\cF_{\alpha,\eps}(1) = 1$ and $\cF'_{\alpha,\eps}(1)<1$. In particular,
$\cF_{\alpha,\eps}(Q)>Q$ for all $\in[0,1)$
\end{lemma}

\begin{pf}
It is convenient to change variables and let $Q=e^{-s}$. If we let
$\{U_s\}_{s\in\reals}$ denote the standard Ornstein--Uhlenbeck process,
$\de U_s = -U_s\,\de s+ \sqrt{2} \,\de B_s$ with $\{B_{s}\}_{s\in\reals}$
the standard Brownian motion. Then $\cF_{\alpha,\eps}(Q) =\break
\hcF_{\alpha,\eps}(-\log(Q))$,
with
%
\begin{equation}
\hspace*{40pt}\hcF_{\alpha,\eps}(s) \equiv\frac{1}{G_{\eps}(\alpha)} \E \bigl\{\bigl[
\eta(X_{\infty}+U_0;\alpha)-X_{\infty}\bigr] \bigl[
\eta(X_{\infty
}+U_s;\alpha)-X_{\infty}\bigr] \bigr\}.
\end{equation}
A simple calculation yields
%
\begin{equation}
\frac{\de}{\de s} \hcF_{\alpha,\eps}(s) = -\frac{1}{G_{\eps}(\alpha)} \E
\bigl\{\eta'(X_{\infty}+U_0;\alpha)
\eta'(X_{\infty}+U_s;\alpha ) \bigr\}
e^{-s}, %
\end{equation}
where $\eta'( \cdot;\alpha)$ denotes the derivative of $\eta$ with
respect to its first argument. By the spectral decomposition of the
Ornstein--Uhlenbeck process, we have, for any function $\psi\in
L_2(\reals)$
%
\begin{equation}
\E \bigl\{\psi(U_0)\psi(U_s) \bigr\} = \sum
_{k=1}^{\infty} e^{-\lambda_k
s}
c_k(\psi)^2 , %
\end{equation}
for some nonnegative $\{\lambda_k\}_{k\ge1}$. In particular $e^s
\frac
{\de}{\de s}
\hcF_{\alpha,\eps}(s)$ is strictly negative and increasing in $s$.
We therefore obtain
%
\begin{equation}
\frac{\de}{\de Q} \cF_{\alpha,\eps}(Q) = \frac{1}{G_{\eps}(\alpha)} \E
\bigl\{\eta'(X_{\infty}+Z_1;\alpha)
\eta'(X_{\infty}+Z_2;\alpha ) \bigr\}
\end{equation}
which is strictly positive and increasing in $Q$. Hence $Q\mapsto
\cF_{\alpha,\eps}(Q)$ is increasing and strictly convex. Finally,
since $\eta'(y;\alpha) = \ind(\llvert y\rrvert \ge\alpha)$, we have
%
\begin{eqnarray}
\frac{\de}{\de Q} \cF_{\alpha,\eps}(Q) \bigg\rrvert
_{Q=1}&=&\frac{1}{G_{\eps}(\alpha)} \prob\bigl\{ \llvert X_{\infty} +Z
\rrvert >\alpha\bigr\} \nonumber\\[-8pt]\\[-8pt]
&=& \frac{1}{G_{\eps}(\alpha)} \bigl\{\eps +2(1-\eps)\Phi(-\alpha)
\bigr\} =\frac{G''_{\eps}(\alpha)}{2G_{\eps
}(\alpha)}.\nonumber %
\end{eqnarray}
Since by Lemma~\ref{lemma:G} $\alpha\mapsto G_{\eps}(\alpha)$ is
strictly increasing over $(\alpha_*(\eps),\infty)$ and by
equation (\ref{eq:SecondG}) $\alpha\mapsto G''_{\eps}(\alpha)$ is strictly
decreasing over $\reals_+$, we have
%
\begin{equation}
 \frac{\de}{\de Q} \cF_{\alpha,\eps}(Q) \bigg\rrvert
_{Q=1}< \frac{G''_{\eps}(\alpha_*(\eps))}{2G_{\eps}(\alpha_*(\eps))}=1 , %
\end{equation}
where the last equality follows again by Lemma~\ref{lemma:G}. This
completes the proof.
\end{pf}

We are now in position to prove part (a2) of Lemma~\ref{6661666616666166616661}.

\begin{pf*}{Proof of Lemma~\ref{6661666616666166616661}(\normalfont{a2})}
Throughout the proof we fix $\alpha\in(\alpha_*(\eps,\delta
),\allowbreak\alpha
_2(\eps,\delta))$.
Let the sequence $\{\sigma_t^2\}_{t\ge0}$ be given as per the state
evolution equation (\ref{eq:ell1SE}). Define $Q_t \equiv
R_{t,t-1}/(\sigma_t\sigma_{t-1})$. By Proposition~\ref{171717171717}, $Q_t$ is the covariance of two Gaussian
random variables of variance $1$. Hence $\llvert Q_t\rrvert \le1$. Using
equation (\ref{eq:TwoTimesRecursion}) we further have
%
\begin{eqnarray}\label{eq:FtRecursion}
Q_{t+1} &=& \cF_{t}(Q_t),\hspace*{-40pt}
\\
\label{eq:FtDefinition}
\cF_t(Q)& =& \frac{\sigma_{t-1}}{\delta\sigma_{t+1}} \E \biggl\{ \biggl[\eta \biggl(
\frac{X}{\sigma_t}+Z_1;\alpha \biggr)-\frac
{X}{\sigma
_t} \biggr]
\biggl[\eta \biggl(\frac{X}{\sigma_{t-1}}+Z_2;\alpha \biggr)-
\frac
{X}{\sigma
_{t-1}} \biggr] \biggr\},\hspace*{-40pt}%
\end{eqnarray}
where expectation is taken with respect to $X\sim p_X$ and the
independent Gaussian random vector $(Z_1,Z_2)$ with zero mean and
covariance $\E\{Z_1^2\}=1$, $\E\{Z_2^2\}=1$, $\E\{Z_1Z_2\}=Q_t$.
By induction it is easy to check that $Q_t\ge0$ for all $t$.

For $\alpha\in(\alpha_1,\alpha_2)$, by part (a1) we have $\sigma
_t\to
0$. Hence $X/\sigma_t$ converges in distribution (over the completed
real line) to a random variable $X_{\infty}\sim
(1-\eps)\delta_0+\eps_+\delta_{+\infty}+\eps_-\delta_{-\infty}$ where
$\eps_+ \equiv\prob\{X>0\}$, $\eps_-\equiv\prob\{X<0\}$, $\eps=
\eps_++\eps_-$. Hence the expectation in equation (\ref{eq:FtDefinition})
converges pointwise to
%
\begin{equation}
\E \bigl\{ \bigl[\eta (X_{\infty}+Z_1;\alpha
)-X_{\infty
} \bigr] \bigl[\eta (X_{\infty}+Z_2;\alpha
)-X_{\infty} \bigr] \bigr\}.
\end{equation}
(Notice that this expectation depends on the distribution of
$X_{\infty}$ only through $\eps$, because of the symmetry properties
of the function $\eta$.)

Further, by the proof of part (a1), as $t\to\infty$
we have $\sigma^2_t\to0$ and
%
\begin{equation}
\hspace*{20pt}\sigma^2_{t+1} = \frac{\de}{\de(\sigma^2)} \F
\bigl(\sigma^2,\alpha_*\sigma\bigr)\bigg\rrvert _{\sigma=
0}
\sigma^2_t+o\bigl(\sigma_t^2\bigr)
= \frac{1}{\delta} G_{\eps}(\alpha_*) \sigma_t^2
+o\bigl(\sigma_t^2\bigr). %
\end{equation}
Hence
%
\begin{equation}
\lim_{t\to\infty} \frac{\sigma_{t-1}}{\sigma_{t+1}} =
\frac{\delta}{G_{\eps}(\alpha)}. %
\end{equation}
Comparing equations (\ref{eq:fQ}) and (\ref{eq:FtDefinition}) we conclude
that, for any $Q\in[0,1]$,
%
\begin{equation}
\lim_{t\to\infty} \cF_t(Q) =
\cF_{\alpha,\eps}(Q). %
\end{equation}
Further the convergence is uniform, since the functions $\cF_t$ are
uniformly Lipschitz; see proof of Lemma~\ref{lemma:LimitFalphEps} above.

Consider now the sequence $\{Q_t\}_{t\ge0}$, and let $Q_* =
\lim\inf_{t\to\infty}Q_t$. Since $Q_t\in[0,1]$ for all
$t$, we have $Q_*\in[0,1]$ as well. We claim that in fact $Q_*=1$ and therefore
$\lim_{t\to\infty} Q_t = 1$, which implies the thesis.

In order to prove the claim, let $\{Q_{t(k)}\}_{k\in\naturals}$ be a
subsequence that converges to $Q_*$. Then
%
\begin{eqnarray}
Q_* &=& \lim_{k\to\infty} \cF_{t(k)-1}(Q_{t(k)-1}) =
\lim_{k\to
\infty} \cF_{\alpha,\eps}(Q_{t(k)-1})\nonumber\\[-8pt]\\[-8pt]
&\ge&
\cF_{\alpha,\eps}\Bigl(\lim\inf_{k\to\infty}Q_{t(k)-1}\Bigr)
\ge \cF_{\alpha,\eps}(Q_*),\nonumber
\end{eqnarray}
where, in the last step, we used the fact that
$\cF_{\alpha,\eps}( \cdot)$ is monotone increasing. Since
$\cF_{\alpha,\eps}(q)>q$ for all $q\in[0,1)$ by Lemma~\ref{lemma:LimitFalphEps}, we conclude that $Q_*=1$.
\end{pf*}

Before proving (a3) of Lemma~\ref{6661666616666166616661}, we establish one
more technical result.

\begin{lemma}
Let $p_X$ be a probability measure on the real line such that
$p_X(\{0\}) = 1-\eps$ and $\E_{p_X}\{X^2\}<\infty$.
Assume $p_X$ to be such that $\max(p_X((0,a)),\allowbreak p_X((-a,0)))\le
Ba^{b}$ for some $B,b>0$. 
Then, letting $X_{\infty}\sim
(1-\eps)\delta_0+\eps_+\delta_{+\infty}+\eps_-\delta_{-\infty}$
[with the notation introduced above, namely, $\eps_+=p_X(0,+\infty)$
and $\eps_-=p_X(-\infty,0)$],
%
\begin{eqnarray}\label{eq:BoundFQ}
&&\biggl|\E \biggl\{ \biggl[\eta \biggl(\frac{X}{\sigma_t}+Z_1;
\alpha \biggr)-\frac
{X}{\sigma_t} \biggr]  \biggl[\eta \biggl(
\frac{X}{\sigma_{t-1}}+Z_2;\alpha \biggr)-\frac
{X}{\sigma
_{t-1}} \biggr]
\biggr\}\nonumber\\
&&\hspace*{10pt}{}- \E \bigl\{ \bigl[\eta (X_{\infty}+Z_1;\alpha
)-X_{\infty
} \bigr] \bigl[\eta (X_{\infty}+Z_2;\alpha
)-X_{\infty} \bigr] \bigr\} \biggr|\\
&&\qquad\le B' \bigl(
\sigma_t^b+\sigma_{t-1}^b\bigr),
\nonumber
\end{eqnarray}
for an eventually different constant $B'$.
Here expectation is taken with respect to $X\sim p_X$ and the
independent Gaussian random vector $(Z_1,Z_2)$ with zero mean and
covariance $\E\{Z_1^2\}=1$, $\E\{Z_2^2\}=1$, $\E\{Z_1Z_2\}=Q_t$
and
%
\begin{equation}\label{eq:FExpansion}
\F\bigl(\sigma^2,\theta\bigr) =
\frac{\de\F}{\de(\sigma^2)}\bigl(\sigma^2;\alpha\sigma \bigr)\bigg\rrvert
_{\sigma
=0}\sigma^2 +O\bigl(\sigma^{2+b}\bigr).
\end{equation}
\end{lemma}

\begin{pf}
By triangular inequality, the left-hand side of equation (\ref{eq:BoundFQ})
can be upper bounded as $D_1+D_2$ whereby
\begin{eqnarray*}
D_1 &\equiv&\E \biggl\{ \biggl[\eta \biggl(
\frac{X}{\sigma_t}+Z_1;\alpha \biggr)-\frac{X}{\sigma_t}-\eta
(X_{\infty}+Z_1;\alpha )+X_{\infty} \biggr]\\
&&\hspace*{90pt}{}\times \biggl[
\eta \biggl(\frac{X}{\sigma_{t-1}}+Z_2;\alpha \biggr)-\frac
{X}{\sigma
_{t-1}}
\biggr] \biggr\},
\\
D_2 &\equiv&\E \biggl\{ \bigl[\eta (X_{\infty}+Z_1;
\alpha )-X_{\infty
} \bigr] \\
&&\hspace*{13pt}\times{}\biggl[\eta \biggl(\frac{X}{\sigma_{t-1}}+Z_2;
\alpha \biggr)-\frac
{X}{\sigma
_{t-1}}-\eta (X_{\infty}+Z_2;\alpha
)+X_{\infty} \biggr] \biggr\}. %
\end{eqnarray*}
Here $X$ and $X_{\infty}$ are coupled in such a way that $X=0$ if and
only if $X_{\infty} =0$ and the two variables have the same sign in
the other case.
We focus on bounding $D_1$ since $D_2$ can be treated along the same
lines. Letting $R(x;\theta) = \eta(x;\theta)-x$, we have
\begin{eqnarray*}
D_1 &=& \E \biggl\{ \biggl[R \biggl(
\frac{X}{\sigma_t}+Z_1;\alpha \biggr)-R (X_{\infty}+Z_1;
\alpha ) \biggr] \biggl[R \biggl(\frac{X}{\sigma_{t-1}}+Z_2;\alpha
\biggr)+Z_2 \biggr] \biggr\} \\
&=&D_{1,a}+D_{1,b},
\\
D_{1,a}& =& \E \biggl\{ \biggl[R \biggl(\frac{X}{\sigma_t}+Z_1;
\alpha \biggr)-R (X_{\infty}+Z_1;\alpha ) \biggr] R \biggl(
\frac{X}{\sigma_{t-1}}+Z_2;\alpha \biggr) \biggr\},
\\
D_{1,b}& =& Q_t \E \biggl\{ \biggl[R' \biggl(
\frac{X}{\sigma_t}+Z_1;\alpha \biggr)-R'
(X_{\infty}+Z_1;\alpha ) \biggr] \biggr\}, %
\end{eqnarray*}
where in the last line we used Stein's lemma to integrate over
$Z_2$, and $R'$ denotes derivative with respect to the first argument.
Once again the two terms are treated along the same lines, and
we will only consider $D_{1,a}$. We have
%
\begin{eqnarray}\label{eq:X+X-}
\llvert D_{1,a}\rrvert & \le&\alpha\E \biggl\{ \biggl\llvert R \biggl(
\frac{X}{\sigma_t}+Z_1;\alpha \biggr)-R (X_{\infty
}+Z_1;
\alpha )\biggr\rrvert \biggr\}
\nonumber
\\
& \le&\alpha\eps_+ \E \biggl\{ \biggl\llvert R \biggl(\frac{X_+}{\sigma_t}+Z_1;
\alpha \biggr)-R(+\infty;\alpha )\biggr\rrvert \biggr\} \\
&&{}+ \alpha\eps_- \E \biggl\{
\biggl\llvert R \biggl(\frac{X_-}{\sigma_t}+Z_1;\alpha \biggr)-R(-
\infty;\alpha )\biggr\rrvert \biggr\},\nonumber%
\end{eqnarray}
where $X_+$ (resp., $X_-$) is distributed as $X$
conditioned on $X>0$ (resp., $X<0$).
The function $x\mapsto R(x;\alpha)-R(\infty;\alpha)$ is monotone
decreasing, equal to $2\alpha$ for $x\le-\alpha$ and to $0$ for
$x\ge
\alpha$. Hence $\tR(x) \equiv\E_{Z_1} \{\llvert R(x+Z_1;\alpha)-R(+\infty;\alpha)\rrvert \}$ is
monotone decreasing, takes values in $(0,2\alpha)$ and upper bounded by
$Ce^{-x^2/4}$ for $x\ge0$. Denoting by $F_+$ the distribution of
$X_+$, we have
\begin{eqnarray*}
&&\E \biggl\{ \biggl\llvert R \biggl(\frac{X_+}{\sigma_t}+Z_1;
\alpha \biggr)-R(+\infty;\alpha )\biggr\rrvert \biggr\} \\
&&\qquad= \E\tR(X_+/
\sigma_t) = \int_{0}^{\infty} \bigl
\llvert \tR'(x)\bigr\rrvert F(x\sigma_t)\, \de x\le
B'\sigma_t^b. %
\end{eqnarray*}
The other term in equation (\ref{eq:X+X-}) is bounded by the same argument.
This concludes the proof of equation (\ref{eq:BoundFQ}).

The proof of equation (\ref{eq:FExpansion}) follows from
equation (\ref{eq:BoundFQ}) if we notice that
\begin{eqnarray*}
\F\bigl(\sigma^2,\alpha\sigma\bigr) &=&
\frac{\sigma^2}{\delta} \E \biggl\{ \biggl[\eta \biggl(\frac{X}{\sigma}+ Z;\alpha
\biggr)-X \biggr]^2 \biggr\},
\\
\frac{\de\F}{\de(\sigma^2)}\bigl(\sigma^2;\alpha\sigma\bigr)\bigg\rrvert
_{\sigma=0} & =& \E \bigl\{ \bigl[\eta (X_{\infty}+Z;\alpha
)-X_{\infty} \bigr]^2 \bigr\}. %
\end{eqnarray*}
\upqed\end{pf}

The last lemma has a useful consequence that we will exploit in the
ensuing proof of Lemma~\ref{6661666616666166616661}(a3).

\begin{corollary}\label{coro:cF}
Let $\cF_{\alpha,\eps}(Q)$ be defined as per equation (\ref{eq:fQ}) and
$\cF_t(Q)$ defined as per equation (\ref{eq:FtDefinition}) with $p_X$,
$\alpha$, $\eps$
satisfying the conditions of Lemma~\ref{6661666616666166616661}\emph{(a3)}. Then
there exists a constants $B,B',b>0$ depending on $p_X$ such that
\[
\sup_{Q\in[0,1]}\bigl\llvert \cF_t(Q)-
\cF_{\alpha,\eps}(Q)\bigr\rrvert \le B \sigma _t^b \le
B'\omega^{bt/2}. %
\]
\end{corollary}

\begin{pf}
The second inequality follows from the first one using Lemma~\ref
{6661666616666166616661}(a1).
Using equation (\ref{eq:FExpansion}), we have
\[
\frac{\sigma_{t-1}^2}{\sigma_{t+1}^2} = \frac{\sigma_t^2}{\F(\sigma_t^2;\alpha\sigma_t)} \cdot\frac{\sigma_{t-1}^2}{\F(\sigma_{t-1}^2;\alpha\sigma_{t-1})} =
\frac{\delta^2}{G_{\eps}(\alpha)^2} \bigl\{1+O(\sigma_t^b,\sigma
_{t-1}^b \bigr\}.
\]
The proof of the corollary is obtained by noting that $\sigma_t=
\Theta(\sigma_{t-1})$ and applying equation (\ref{eq:BoundFQ}) to the
expectation in
equation (\ref{eq:FtDefinition}).
\end{pf}

\begin{pf*}{Proof of Lemma~\ref{6661666616666166616661}(\normalfont{a3})}
Define, as in the proof of part (a2), $Q_t\equiv
R_{t,t-1}/(\sigma_t\sigma_{t-1})$, and recall that
\[
Q_{t+1} = \cF_t(Q_t).
\]
By Corollary~\ref{coro:cF}, and Lemma~\ref{lemma:LimitFalphEps}, it
follows that $Q_t\ge1-A \oom^{2t}$ for some constants $A>0$, $\oom
\in
(0,1)$. Indeed,
\[
Q_{t+1} \ge\cF_{\alpha,\eps}(Q_t)-B'
\omega^{bt/2} \ge1-B' \omega ^{bt/2}-
\cF_{\alpha,\eps}'(1) (1-Q_t), %
\]
and the claim follows by noting that $\cF_{\alpha,\eps}'(1) \in(0,1)$
by Lemma~\ref{lemma:LimitFalphEps}.

Next, consider a sequence of centered Gaussian random variables
$(Z_t)_{t\ge0}$ with covariance $\E\{Z_tZ_s\} = R_{t,s}$. By
triangular inequality, we have, for any $t<s$,
%
\begin{eqnarray}\label{eq:QtUniform}
\qquad\biggl(2-2\frac{R_{t,s}}{\sigma_t\sigma_s} \biggr)^{1/2} &=& \E \biggl\{
\biggl(\frac{Z_t}{\sigma_t}-\frac{Z_s}{\sigma_s} \biggr)^2 \biggr
\}^{1/2} \le\sum_{k=t+1}^s \E
\biggl\{ \biggl(\frac{Z_k}{\sigma_k}-\frac{Z_{k-1}}{\sigma
_{k-1}} \biggr)^2
\biggr\}^{1/2} \nonumber\\[-8pt]\\[-8pt]
&=& \sum_{k=t+1}^s(2-2Q_{k})^{1/2}
\le A'\oom^{t}.\nonumber%
\end{eqnarray}
Next consider the quantity in equation (\ref{eq:PointA3}). We have
%
\begin{eqnarray}\label{eq:PointA3Proof}
&&\sup_{t,s\ge t_0}\prob \bigl\{\llvert
X+Z_s\rrvert \ge c \sigma_s; \llvert X+Z_t
\rrvert <c \sigma_t \bigr\}
\nonumber
\\
&&\qquad\le\sup_{t\ge t_0}\prob \bigl\{ \llvert X+Z_t\rrvert
<c \sigma_t; X\neq0 \bigr\}\nonumber\\[-8pt]\\[-8pt]
&&\quad\qquad{}+ \sup_{t,s\ge t_0} \prob
\bigl\{\llvert Z_s/\sigma_s\rrvert \ge c; \llvert
Z_t/\sigma_t\rrvert <c; X =0 \bigr\}\nonumber
\\
&&\qquad= \sup_{t\ge t_0}\prob \bigl\{ \llvert X/\sigma_t+
\tZ_t\rrvert <c; X\neq0 \bigr\}+ \sup_{t,s\ge t_0} \prob
\bigl\{\llvert \tZ_s\rrvert \ge c; \llvert \tZ_t\rrvert
<c \bigr\},
\nonumber%
\end{eqnarray}
where $(\tZ_s,\tZ_t)$ are Gaussian with
$\E\{\tZ_t^2\}=\E\{\tZ_s^2\}=1$, and
$\E\{\tZ_s\tZ_t\}=R_{t,s}/\allowbreak(\sigma_t\sigma_s)$. The first term in
equation (\ref{eq:PointA3Proof}) vanishes as $t_0\to\infty$ since
$\sigma_t\to0$ as $t\to\infty$, and the second vanishes by
equation (\ref{eq:QtUniform}).
\end{pf*}

\subsection{Proof of Lemma \texorpdfstring{\protect\ref{6661666616666166616661}(b)}{6(b)}:
\texorpdfstring{$\rho>\rho_*(\delta)$}{$rho>rho_*(delta)$}}

\mbox{}

\begin{pf*}{Proof of Lemma~\ref{6661666616666166616661}(\normalfont{b1}), (\normalfont{b2})}
First notice that, with the definitions given in the previous section
\begin{eqnarray*}
\lim_{\sigma^2\to\infty} \frac{\de}{\de(\sigma^2)} \F\bigl(
\sigma^2,\alpha _*\sigma\bigr) &=& \frac{2}{\delta} \E \bigl\{(Z-
\alpha)_+^2 \bigr\}
\\
&=& \frac{2}{\delta} \bigl\{\bigl(1+\alpha^2\bigr)\Phi(-\alpha)-
\alpha\phi (\alpha ) \bigr\}. %
\end{eqnarray*}
Notice that the right-hand side is equal to $2/\delta$ for $\alpha= 0$,
monotonically decreasing in $\alpha$ and vanishing as
$\alpha\to\infty$. Hence there exists $\alpha_{\min}(\eps,\delta)$
such that the right-hand side is smaller than $1$ if and only if
$\alpha>\alpha_{\min}(\eps,\delta)$. Further, $\sigma^2\mapsto\F
(\sigma
^2,\alpha\sigma)$ is concave
with $\F(0,0) = 0$ and first derivative larger than $1$ at $\sigma
^2=0$; cf. Lemma~\ref{lemma:FirstDerivativeF}. It follows that for
$\alpha>\alpha_{\min}(\eps,\delta)$ there exists a unique
$\sigma_*(\delta,p_X)$ such that $\F(\sigma^2,\alpha\sigma
)>\sigma^2$
for all $\sigma\in(0,\sigma_*)$ and
$\F(\sigma^2,\alpha\sigma)<\sigma^2$ for $\sigma\in(\sigma
_*,\infty)$.
It follows
that $\sigma^2_t\to\sigma_*$ for any $\sigma_0^2\neq0$.
This proves the first part of claim (b1).

Letting $\sigma_*^2=\sigma_*^2(\alpha)$,
it is easy to check that $\alpha\mapsto\sigma^2_*(\alpha)$ is
continuous for $\alpha\in(\alpha_{\min},\infty)$ with
$\lim_{\alpha\to\alpha_{\min}}\sigma_*^2(\alpha) =+\infty$\vspace*{2pt} (the
limit being taken from the left), and
$\lim_{\alpha\to\infty}\sigma_*^2(\alpha) =+\E\{X^2\}/\delta>0$.
As a consequence,
%
\begin{eqnarray}
\lim_{\alpha\to\alpha_{\min}} \prob\bigl\{\llvert X+\sigma_* Z
\rrvert \ge \alpha\sigma_*\bigr\} &=& 2\Phi(-\alpha_{\min}),
\\
\lim_{\alpha\to\infty} \prob\bigl\{\llvert X+\sigma_* Z\rrvert \ge \alpha
\sigma_*\bigr\} &=& 0 . %
\end{eqnarray}
Notice that by the definition of $\alpha_{\min}$ given above, we have
\[
2 \Phi(-\alpha_{\min}) -2\alpha_{\min} \bigl\{
\phi(\alpha_{\min})-\alpha_{\min}\Phi(-\alpha_{\min
})
\bigr\} =\delta. %
\]
Since $\phi(z)>z\Phi(-z)$ for $z>0$, it follows that $\lim_{\alpha
\to
\alpha_{\min}} \prob\{\llvert X+\sigma_* Z\rrvert \ge
\alpha\sigma_*\} >\delta$. We define
%
\begin{equation}\label{eq:Alpha0Def}
\alpha_{0}(\delta,p_X) \equiv \sup \bigl
\{\alpha>\alpha_{\min}(\eps,\delta) \dvtx  \prob\bigl\{\llvert X+\sigma _* Z
\rrvert \ge \alpha\sigma_*\bigr\}\ge\delta \bigr\}.%
\end{equation}
By the above $\alpha_{0}\in(\alpha_{\min},\infty)$. Further, by
continuity, for $\alpha= \alpha_{0}$, $\prob\{\llvert X+\sigma_* Z\rrvert \ge
\alpha\sigma_*\}= \delta$. We thus proved claim (b2).

In order to prove the second statement in (b1), we proceed
analogously to part (a2), and define $Q_t\equiv
R_t/(\sigma_t\sigma_{t-1})$. This sequence satisfies the recursion
(\ref{eq:FtRecursion}) with $\cF_t$ defined as per
equation (\ref{eq:FtDefinition}).
As $t\to\infty$ we have $\sigma_t\to\sigma_*$ and hence $\cF_t$
converges uniformly to a limit that we denote by an abuse of notation
$\cF_{\alpha,\delta,p_X}$, where
%
\begin{equation}
\hspace*{38pt}\cF_{\alpha,\delta,p_X} (Q) \equiv\frac{1}{\delta} \E \biggl\{ \biggl[\eta \biggl(
\frac{X}{\sigma_*}+Z_1;\alpha \biggr)-\frac{X}{\sigma_*} \biggr]
\biggl[\eta \biggl(\frac{X}{\sigma_{*}}+Z_2;\alpha \biggr)-
\frac
{X}{\sigma
_{*}} \biggr] \biggr\}.
\end{equation}
Proceeding as in the proof of Lemma~\ref{lemma:LimitFalphEps}, we
conclude that $Q\mapsto\cF_{\alpha,\delta,p_X}(Q)$ is increasing and
convex on $[0,1]$. Further [for $Z\sim\normal(0,1)$]
%
\begin{equation}
\hspace*{38pt}\cF_{\alpha,\delta,p_X}(1)= \frac{1}{\delta} \E \biggl\{ \biggl[\eta
\biggl(\frac{X}{\sigma_*}+Z_1;\alpha \biggr)-\frac
{X}{\sigma
_*}
\biggr]^2 \biggr\} = \frac{1}{\sigma_*^2} \F\bigl(\sigma_*^2,
\alpha\sigma_*\bigr) = 1 . %
\end{equation}
Finally, for $\alpha\ge\alpha_{0}(\delta,p_X)$,
%
\begin{equation}
\frac{\de}{\de Q} \cF_{\alpha,\delta,p_X}(Q)\bigg\rrvert
_{Q=1}= \frac{1}{\delta}\prob \biggl\{\biggl\llvert
\frac{X}{\sigma_*}+Z_1 \biggr\rrvert >\alpha \biggr\} \le 1 ,
\end{equation}
and therefore $ \cF_{\alpha,\delta,p_X}(Q)> Q$ for all $Q\in[0,1)$.
Hence, proceeding again as in the proof of part (a2) we conclude
that $\lim_{t\to\infty}Q_t= 1$ and therefore
$\lim_{t\to\infty}R_{t,t-1} = \sigma_*^2$ as claimed.
\end{pf*}

\begin{pf*}{Proof of Lemma~\ref{6661666616666166616661}(\normalfont{b3})}
Throughout this proof we fix $p_X = (1-\eps)\delta_0+\eps\gamma$,
$\delta\in(\eps,\delta_*(\eps))$.
By part (b1), we have $\lim_{t\to\infty}\E\{\llvert \eta(X+\sigma_t
X;\alpha\sigma_t)\rrvert \} = \E\{\llvert \eta(X+\sigma_*Z;\alpha\sigma_*)\rrvert \}$.
It is therefore sufficient to prove that $\E\{\llvert \eta(X+\sigma
_*Z;\alpha
\sigma_*)\rrvert \}<\E\{\llvert X\rrvert \}$.

Consider the function
$\cE\dvtx (\sigma^2,\theta)\mapsto\cE(\sigma^2,\theta)$ defined on
$\reals_+\times\reals_+$ by
%
\begin{equation}\label{eq:EnergyDef}
\cE\bigl(\sigma^2,\theta\bigr)\equiv-
\frac{1}{2}(1-\delta)\frac{\sigma
^2}{\theta} +\E\min_{s\in\reals}
\biggl\{\frac{1}{2\theta}(s-X-\sigma Z)^2 + \llvert s\rrvert \biggr
\}, %
\end{equation}
where expectation is taken with respect to $X\sim p_X$ and $Z\sim
\normal(0,1)$. Notice that the minimum over $s\in\reals$ is uniquely
achieved
at $s=\eta(X+\sigma Z;\theta)$.
It is not hard to compute the partial derivatives
%
\begin{eqnarray}\label{eq:DeDtheta}
\hspace*{20pt}\frac{\partial\cE}{\partial\theta}\bigl(\sigma^2,\theta\bigr) &= & {-}
\frac{\delta}{2\theta^2} \biggl\{ \biggl(1-\frac{2}{\delta}\prob\bigl\{\llvert X+
\sigma Z\rrvert \ge \theta\bigr\} \biggr)\sigma^2+\F\bigl(
\sigma^2,\theta\bigr) \biggr\},
\\
\label{eq:DeDsigma}
\hspace*{20pt}\frac{\partial\cE}{\partial\sigma^2}\bigl(\sigma^2,\theta\bigr) &= &
\frac{\delta}{2\theta} \biggl\{1-\frac{1}{\delta}\prob\bigl\{\llvert X+\sigma Z
\rrvert \ge \theta\bigr\} \biggr\}, %
\end{eqnarray}
where $\F(\sigma^2,\theta)$ is defined as per equation (\ref{eq:SEMapL1}).
Using these expressions in equation (\ref{eq:EnergyDef}) we conclude that
%
\begin{equation}
\hspace*{30pt}\frac{\partial\cE}{\partial\theta}\bigl(\sigma^2,\theta\bigr) =
\frac
{\partial
\cE}{\partial\sigma^2}\bigl(\sigma^2,\theta\bigr) =0 \quad\Rightarrow\quad\cE\bigl(
\sigma^2,\theta\bigr) = \E\bigl\{\bigl\llvert \eta(X+\sigma Z;\theta)
\bigr\rrvert \bigr\}. %
\end{equation}
In particular, one can check from equations (\ref{eq:DeDtheta}),
(\ref{eq:DeDsigma}) that a stationary
point\footnote{Indeed this is the unique saddle point of the function
$(\theta^{-1},\sigma^2)\mapsto\cE(\theta,\sigma^2)$
as it can be proved by the general minimax theorem.} is given by
setting $\sigma=
\sigma_*(\delta,p_X)$ and $\theta= \theta_*(\delta,p_X) \equiv
\alpha_0(\delta,p_X)\sigma_*(\delta,p_X)$.

Define $E(\sigma^2) = \cE(\sigma^2,\alpha_0(\delta,p_X)\sigma)$. Using
again equations (\ref{eq:DeDtheta}), (\ref{eq:DeDsigma}) we get
%
\begin{equation}
\frac{\de E}{\de\sigma^2}\bigl(\sigma^2\bigr) =
\frac{\delta}{4\alpha\sigma^3} \bigl\{\sigma^2-\F\bigl(\sigma^2,
\alpha _0\sigma \bigr) \bigr\}. %
\end{equation}
In particular, as a consequence of Lemma~\ref{lemma:FirstDerivativeF},
and of the analysis at point (b1), we have $\frac{\de
E}{\de\sigma^2} <0$ for $\sigma^2\in(0,\sigma_*^2)$
(\ref{eq:DeDtheta}). Therefore, setting
$\alpha=\alpha_0(\delta,p_X)$, we have
\begin{eqnarray*}
&&\E\bigl\{\bigl\llvert \eta(X+\sigma_*Z;\alpha\sigma_*)\bigr\rrvert
\bigr\} \\
&&\qquad= E\bigl(\sigma_*^2\bigr) <\lim_{\sigma\to0}E
\bigl(\sigma^2\bigr)
\\
&&\qquad= -\lim_{\sigma\to0} \frac{1}{2\alpha}\sigma(1-\delta)+ \lim
_{\sigma\to0} \frac{\sigma}{2\alpha}\E \biggl\{ \biggl[\eta \biggl(
\frac
{X}{\sigma}+ Z;\alpha \biggr)-\frac{X}{\sigma}- Z \biggr]^2
\biggr\} \\
&&\quad\qquad{}+ \lim_{\sigma
\to0} \E\bigl\{\bigl\llvert \eta (X+\sigma Z;
\alpha\sigma)\bigr\rrvert \bigr\}
\\
&&\qquad = \lim_{\sigma\to0} \frac{\sigma}{2\alpha} \alpha^2 +\E
\bigl\{ \llvert X\rrvert \bigr\} = \E\bigl\{\llvert X\rrvert \bigr\}.
\end{eqnarray*}
This completes the proof.
\end{pf*}

%
%
\section{Reference results}\label{app:calculus}
The following calculus fact is used in the main text.

\begin{lemma}\label{lem:x_s_ineq}
For all $s,x>0$ we have $x^s\leq ( \frac{s}{e} )^se^x$.
\end{lemma}

\begin{pf}
Since $f(x)=\ln(x)$ for $x>0$ is concave, when $x\ge s$, then
%
\begin{equation}
\frac{\ln(x)-\ln(s)}{x-s}\le f'(s)=\frac{1}{s}.
\end{equation}
This is equivalent to $(x/s)^s\le e^{x-s}$ which proves the result. The
case of $x<s$ is proved similarly.
\end{pf}

We also use an estimate on the minimum singular value of perturbed
rectangular matrices, which was proved in \cite{Cucker}, Theorem~1.1.

\begin{theorem}
For $M,N\in\naturals$, $N\le(1-a)M$, let $B\in\reals^{M\times N}$,
$\llVert B\rrVert _2\le1/a$ be any deterministic
matrix and $G\in\reals^{M\times N}$ be a matrix with i.i.d. entries
$G_{ij}\sim\normal(0,1/M)$. Then there exist constants $a_1$, $a_2$
depending only on $a$ and bounded for $a>0$ such that,
for all $z<a_2$,
%
\begin{equation}
\prob \bigl\{\sigma_N (A+\nu G )\le\nu z \bigr\}\le
(a_1 z)^{M-N+1}.
\end{equation}
\end{theorem}
\end{appendix}

\section*{Acknowledgments}
Andrea Montanari is grateful to Amir Dembo, David Donoho and Van Vu for
stimulating conversations.



\printaddresses


\begin{thebibliography}{31}
\bibitem{Adamczak}
\begin{barticle}[mr]
\bauthor{\bsnm{Adamczak},~\bfnm{R.}\binits{R.}},
\bauthor{\bsnm{Litvak},~\bfnm{A.~E.}\binits{A.~E.}},
\bauthor{\bsnm{Pajor},~\bfnm{A.}\binits{A.}} \AND
\bauthor{\bsnm{Tomczak-Jaegermann},~\bfnm{N.}\binits{N.}}
(\byear{2011}).
\btitle{Restricted isometry property of matrices with independent columns and neighborly polytopes by random sampling}.
\bjournal{Constr. Approx.}
\bvolume{34}
\bpages{61--88}.
\bid{doi={10.1007/s00365-010-9117-4}, issn={0176-4276}, mr={2796091}}
\end{barticle}
\bptok{imsref}%
\endbibitem\vspace*{-9pt}

\bibitem{Affentranger}
\begin{barticle}[mr]
\bauthor{\bsnm{Affentranger},~\bfnm{Fernando}\binits{F.}} \AND
\bauthor{\bsnm{Schneider},~\bfnm{Rolf}\binits{R.}}
(\byear{1992}).
\btitle{Random projections of regular simplices}.
\bjournal{Discrete Comput. Geom.}
\bvolume{7}
\bpages{219--226}.
\bid{doi={10.1007/BF02187839}, issn={0179-5376}, mr={1149653}}
\end{barticle}
\bptok{imsref}%
\endbibitem

\bibitem{Guionnet}
\begin{bbook}[mr]
\bauthor{\bsnm{Anderson},~\bfnm{Greg~W.}\binits{G.~W.}},
\bauthor{\bsnm{Guionnet},~\bfnm{Alice}\binits{A.}} \AND
\bauthor{\bsnm{Zeitouni},~\bfnm{Ofer}\binits{O.}}
(\byear{2010}).
\btitle{An Introduction to Random Matrices}.
\bseries{Cambridge Studies in Advanced Mathematics}
\bvolume{118}.
\bpublisher{Cambridge Univ. Press},
\blocation{Cambridge}.
\bid{mr={2760897}}
\bptnote{check year}%
\end{bbook}
\bptok{imsref}%
\endbibitem

\bibitem{BaiSilverstein}
\begin{bbook}[mr]
\bauthor{\bsnm{Bai},~\bfnm{Zhidong}\binits{Z.}} \AND
\bauthor{\bsnm{Silverstein},~\bfnm{Jack~W.}\binits{J.~W.}}
(\byear{2009}).
\btitle{Spectral Analysis of Large Dimensional Random Matrices},
\bedition{2nd} ed.
\bpublisher{Springer},
\blocation{New York}.
\bid{doi={10.1007/978-1-4419-0661-8}, mr={2567175}}
\bptnote{check year}%
\end{bbook}
\bptok{imsref}%
\endbibitem

\bibitem{BaiSilversteinPaper}
\begin{barticle}[mr]
\bauthor{\bsnm{Bai},~\bfnm{Z.~D.}\binits{Z.~D.}} \AND
\bauthor{\bsnm{Silverstein},~\bfnm{Jack~W.}\binits{J.~W.}}
(\byear{1998}).
\btitle{No eigenvalues outside the support of the limiting spectral distribution of large-dimensional sample covariance matrices}.
\bjournal{Ann. Probab.}
\bvolume{26}
\bpages{316--345}.
\bid{doi={10.1214/aop/1022855421}, issn={0091-1798}, mr={1617051}}
\end{barticle}
\bptok{imsref}%
\endbibitem

\bibitem{BayatiMontanariLASSO}
\begin{barticle}[mr]
\bauthor{\bsnm{Bayati},~\bfnm{Mohsen}\binits{M.}} \AND
\bauthor{\bsnm{Montanari},~\bfnm{Andrea}\binits{A.}}
(\byear{2012}).
\btitle{The {LASSO} risk for {G}aussian matrices}.
\bjournal{IEEE Trans. Inform. Theory}
\bvolume{58}
\bpages{1997--2017}.
\bid{doi={10.1109/TIT.2011.2174612}, issn={0018-9448}, mr={2951312}}
\end{barticle}
\bptok{imsref}%
\endbibitem

\bibitem{bolthausen2012iterative}
\begin{barticle}[mr]
\bauthor{\bsnm{Bolthausen},~\bfnm{Erwin}\binits{E.}}
(\byear{2014}).
\btitle{An iterative construction of solutions of the {TAP} equations for the {S}herrington--{K}irkpatrick model}.
\bjournal{Comm. Math. Phys.}
\bvolume{325}
\bpages{333--366}.
\bid{doi={10.1007/s00220-013-1862-3}, issn={0010-3616}, mr={3147441}}
\bptnote{check year}%
\end{barticle}
\bptok{imsref}%
\endbibitem

\bibitem{Cucker}
\begin{barticle}[mr]
\bauthor{\bsnm{B{\"u}rgisser},~\bfnm{Peter}\binits{P.}} \AND
\bauthor{\bsnm{Cucker},~\bfnm{Felipe}\binits{F.}}
(\byear{2010}).
\btitle{Smoothed analysis of {M}oore--{P}enrose inversion}.
\bjournal{SIAM J. Matrix Anal. Appl.}
\bvolume{31}
\bpages{2769--2783}.
\bid{doi={10.1137/100782954}, issn={0895-4798}, mr={2740632}}
\end{barticle}
\bptok{imsref}%
\endbibitem

\bibitem{DonohoTannerUniversality}
\begin{barticle}[mr]
\bauthor{\bsnm{Donoho},~\bfnm{David}\binits{D.}} \AND
\bauthor{\bsnm{Tanner},~\bfnm{Jared}\binits{J.}}
(\byear{2009}).
\btitle{Observed universality of phase transitions in high-dimensional geometry, with implications for modern data analysis and signal processing}.
\bjournal{Philos. Trans. R. Soc. Lond. Ser. A Math. Phys. Eng. Sci.}
\bvolume{367}
\bpages{4273--4293}.
\bid{doi={10.1098/rsta.2009.0152}, issn={1364-503X}, mr={2546388}}
\bptnote{check year}%
\end{barticle}
\bptok{imsref}%
\endbibitem

\bibitem{Donoho2005a}
\begin{bmisc}[auto:STB|2014/02/12|14:17:21]
\bauthor{\bsnm{Donoho},~\bfnm{D.~L.}\binits{D.~L.}}
\bhowpublished{({2005}).
{Neighborly polytopes and sparse solution of underdetermined linear equations}.
{Technical report, Statistics Dept.},
Stanford Univ., Stanford, CA.}
\end{bmisc}
\bptok{imsref}%
\endbibitem

\bibitem{Donoho2005b}
\begin{barticle}[mr]
\bauthor{\bsnm{Donoho},~\bfnm{David~L.}\binits{D.~L.}}
(\byear{2006}).
\btitle{High-dimensional centrally symmetric polytopes with neighborliness proportional to dimension}.
\bjournal{Discrete Comput. Geom.}
\bvolume{35}
\bpages{617--652}.
\bid{doi={10.1007/s00454-005-1220-0}, issn={0179-5376}, mr={2225676}}
\bptnote{check year}%
\end{barticle}
\bptok{imsref}%
\endbibitem

\bibitem{donoho2013information}
\begin{barticle}[mr]
\bauthor{\bsnm{Donoho},~\bfnm{David~L.}\binits{D.~L.}},
\bauthor{\bsnm{Javanmard},~\bfnm{Adel}\binits{A.}} \AND
\bauthor{\bsnm{Montanari},~\bfnm{Andrea}\binits{A.}}
(\byear{2013}).
\btitle{Information-theoretically optimal compressed sensing via spatial coupling and approximate message passing}.
\bjournal{IEEE Trans. Inform. Theory}
\bvolume{59}
\bpages{7434--7464}.
\bid{doi={10.1109/TIT.2013.2274513}, issn={0018-9448}, mr={3124654}}
\end{barticle}
\bptok{imsref}%
\endbibitem

\bibitem{donoho2013accurate}
\begin{barticle}[mr]
\bauthor{\bsnm{Donoho},~\bfnm{David~L.}\binits{D.~L.}},
\bauthor{\bsnm{Johnstone},~\bfnm{Iain}\binits{I.}} \AND
\bauthor{\bsnm{Montanari},~\bfnm{Andrea}\binits{A.}}
(\byear{2013}).
\btitle{Accurate prediction of phase transitions in compressed sensing via a connection to minimax denoising}.
\bjournal{IEEE Trans. Inform. Theory}
\bvolume{59}
\bpages{3396--3433}.
\bid{doi={10.1109/TIT.2013.2239356}, issn={0018-9448}, mr={3061255}}
\end{barticle}
\bptok{imsref}%
\endbibitem

\bibitem{DMM09}
\begin{barticle}[auto:STB|2014/02/12|14:17:21]
\bauthor{\bsnm{Donoho},~\bfnm{D.~L.}\binits{D.~L.}},
\bauthor{\bsnm{Maleki},~\bfnm{A.}\binits{A.}} \AND
\bauthor{\bsnm{Montanari},~\bfnm{A.}\binits{A.}}
(\byear{2009}).
\btitle{Message passing algorithms for compressed sensing}.
\bjournal{Proc. Natl. Acad. Sci. USA}
\bvolume{106}
\bpages{18914--18919}.
\end{barticle}
\bptok{imsref}%
\endbibitem

\bibitem{NSPT}
\begin{barticle}[mr]
\bauthor{\bsnm{Donoho},~\bfnm{David~L.}\binits{D.~L.}},
\bauthor{\bsnm{Maleki},~\bfnm{Arian}\binits{A.}} \AND
\bauthor{\bsnm{Montanari},~\bfnm{Andrea}\binits{A.}}
(\byear{2011}).
\btitle{The noise-sensitivity phase transition in compressed sensing}.
\bjournal{IEEE Trans. Inform. Theory}
\bvolume{57}
\bpages{6920--6941}.
\bid{doi={10.1109/TIT.2011.2165823}, issn={0018-9448}, mr={2882271}}
\end{barticle}
\bptok{imsref}%
\endbibitem

\bibitem{DoTa05b}
\begin{barticle}[mr]
\bauthor{\bsnm{Donoho},~\bfnm{David~L.}\binits{D.~L.}} \AND
\bauthor{\bsnm{Tanner},~\bfnm{Jared}\binits{J.}}
(\byear{2005}).
\btitle{Neighborliness of randomly projected simplices in high dimensions}.
\bjournal{Proc. Natl. Acad. Sci. USA}
\bvolume{102}
\bpages{9452--9457 (electronic)}.
\bid{doi={10.1073/pnas.0502258102}, issn={1091-6490}, mr={2168716}}
\end{barticle}
\bptok{imsref}%
\endbibitem

\bibitem{DoTa05a}
\begin{barticle}[mr]
\bauthor{\bsnm{Donoho},~\bfnm{David~L.}\binits{D.~L.}} \AND
\bauthor{\bsnm{Tanner},~\bfnm{Jared}\binits{J.}}
(\byear{2005}).
\btitle{Sparse nonnegative solution of underdetermined linear equations by linear programming}.
\bjournal{Proc. Natl. Acad. Sci. USA}
\bvolume{102}
\bpages{9446--9451 (electronic)}.
\bid{doi={10.1073/pnas.0502269102}, issn={1091-6490}, mr={2168715}}
\end{barticle}
\bptok{imsref}%
\endbibitem

\bibitem{DoTa08}
\begin{barticle}[mr]
\bauthor{\bsnm{Donoho},~\bfnm{David~L.}\binits{D.~L.}} \AND
\bauthor{\bsnm{Tanner},~\bfnm{Jared}\binits{J.}}
(\byear{2009}).
\btitle{Counting faces of randomly projected polytopes when the projection radically lowers dimension}.
\bjournal{J. Amer. Math. Soc.}
\bvolume{22}
\bpages{1--53}.
\bid{doi={10.1090/S0894-0347-08-00600-0}, issn={0894-0347}, mr={2449053}}
\end{barticle}
\bptok{imsref}%
\endbibitem

\bibitem{javanmard2013state}
\begin{barticle}[auto:STB|2014/02/12|14:17:21]
\bauthor{\bsnm{Javanmard},~\bfnm{Adel}\binits{A.}} \AND
\bauthor{\bsnm{Montanari},~\bfnm{Andrea}\binits{A.}}
(\byear{2013}).
\btitle{State evolution for general approximate message passing algorithms, with applications to spatial coupling}.
\bjournal{Information and Inference}
\bvolume{2}
\bpages{115--144}.
\end{barticle}
\bptok{imsref}%
\endbibitem

\bibitem{KabashimaTanaka}
\begin{barticle}[auto:STB|2014/02/12|14:17:21]
\bauthor{\bsnm{Kabashima},~\bfnm{Y.}\binits{Y.}},
\bauthor{\bsnm{Wadayama},~\bfnm{T.}\binits{T.}} \AND
\bauthor{\bsnm{Tanaka},~\bfnm{T.}\binits{T.}}
(\byear{2009}).
\btitle{A typical reconstruction limit for compressed sensing based on lp-norm minimization}.
\bjournal{J. Stat. Mech.}
\bnote{L09003}.
\end{barticle}
\bptok{imsref}%
\endbibitem

\bibitem{krzakala2012statistical}
\begin{barticle}[auto:STB|2014/02/12|14:17:21]
\bauthor{\bsnm{Krzakala},~\bfnm{Florent}\binits{F.}},
\bauthor{\bsnm{M{\'e}zard},~\bfnm{Marc}\binits{M.}},
\bauthor{\bsnm{Sausset},~\bfnm{Fran{\c{c}}ois}\binits{F.}},
\bauthor{\bsnm{Sun},~\bfnm{Y.F.}\binits{Y. F.}} \AND
\bauthor{\bsnm{Zdeborov\'a},~\bfnm{Lenka}\binits{L.}}
(\byear{2012}).
\btitle{Statistical-physics-based reconstruction in compressed sensing}.
\bjournal{Physical Review X}
\bvolume{2}
\bnote{021005}.
\end{barticle}
\bptok{imsref}%
\endbibitem

\bibitem{Lubinsky}
\begin{barticle}[mr]
\bauthor{\bsnm{Lubinsky},~\bfnm{D.~S.}\binits{D.~S.}}
(\byear{2007}).
\btitle{A survey of weighted polynomial approximation with exponential weights}.
\bjournal{Surv. Approx. Theory}
\bvolume{3}
\bpages{1--105}.
\bid{issn={1555-578X}, mr={2276420}}
\bptnote{check year}%
\end{barticle}
\bptok{imsref}%
\endbibitem

\bibitem{MalekiComplex}
\begin{bmisc}[auto:STB|2014/02/12|14:17:21]
\bauthor{\bsnm{Maleki},~\bfnm{A.}\binits{A.}},
\bauthor{\bsnm{Anitori},~\bfnm{L.}\binits{L.}},
\bauthor{\bsnm{Yang},~\bfnm{A.}\binits{A.}} \AND
\bauthor{\bsnm{Baraniuk},~\bfnm{R.}\binits{R.}}
\bhowpublished{(2011). {Asymptotic analysis of complex LASSO via complex approximate message passing (CAMP)}.
Available at \arxivurl{arXiv:1108.0477}}.
\end{bmisc}
\bptok{imsref}%
\endbibitem

\bibitem{mezard1987spin}
\begin{bbook}[mr]
\bauthor{\bsnm{M{\'e}zard},~\bfnm{Marc}\binits{M.}},
\bauthor{\bsnm{Parisi},~\bfnm{Giorgio}\binits{G.}} \AND
\bauthor{\bsnm{Virasoro},~\bfnm{Miguel~Angel}\binits{M.~A.}}
(\byear{1987}).
\btitle{Spin Glass Theory and Beyond}.
\bseries{World Scientific Lecture Notes in Physics}
\bvolume{9}.
\bpublisher{World Scientific},
\blocation{Teaneck, NJ}.
\bid{mr={1026102}}
\end{bbook}
\bptok{imsref}%
\endbibitem

\bibitem{RanganGAMP}
\begin{bincollection}[auto:STB|2014/02/12|14:17:21]
\bauthor{\bsnm{Rangan},~\bfnm{S.}\binits{S.}}
(\byear{2011}).
\btitle{Generalized approximate message passing for estimation with random linear mixing}.
 In \bbooktitle{IEEE Intl. Symp. on Inform. Theory (St. Perersbourg), August
 2011}. \bpublisher{IEEE},
\blocation{Piscataway, NJ}.
\end{bincollection}
\bptok{imsref}%
\endbibitem

\bibitem{RanganFletcherGoyal}
\begin{bincollection}[auto:STB|2014/02/12|14:17:21]
\bauthor{\bsnm{Rangan},~\bfnm{S.}\binits{S.}},
\bauthor{\bsnm{Fletcher},~\bfnm{A.~K.}\binits{A.~K.}} \AND
\bauthor{\bsnm{Goyal},~\bfnm{V.~K.}\binits{V.~K.}}
(\byear{2009}).
\btitle{Asymptotic analysis of map estimation via the replica method and applications to compressed sensing}.
In \bbooktitle{Neural Information Processing Systems (NIPS)}.
\blocation{Vancouver}.
\end{bincollection}
\bptok{imsref}%
\endbibitem

\bibitem{SchniterTurbo}
\begin{bincollection}[auto:STB|2014/02/12|14:17:21]
\bauthor{\bsnm{Schniter},~\bfnm{P.}\binits{P.}}
(\byear{2010}).
\btitle{Turbo reconstruction of structured sparse signals}.
In \bbooktitle{Proceedings of the Conference on Information Sciences and Systems}.
\blocation{Princeton, NJ}.
\end{bincollection}
\bptok{imsref}%
\endbibitem

\bibitem{schniter2012compressive}
\begin{bincollection}[auto:STB|2014/02/12|14:17:21]
\bauthor{\bsnm{Schniter},~\bfnm{Philip}\binits{P.}} \AND
\bauthor{\bsnm{Rangan},~\bfnm{Sundeep}\binits{S.}}
(\byear{2012}).
\btitle{Compressive phase retrieval via generalized approximate message passing}.
In \bbooktitle{Communication, Control, and Computing (Allerton), 2012 50th Annual Allerton Conference on}
\bpages{815--822}.
\bpublisher{IEEE},
\blocation{Piscataway, NJ}.
\end{bincollection}
\bptok{imsref}%
\endbibitem

\bibitem{TaoVuReview}
\begin{bmisc}[mr]
\bauthor{\bsnm{Tao},~\bfnm{Terence}\binits{T.}} \AND
\bauthor{\bsnm{Vu},~\bfnm{Van}\binits{V.}}
\bhowpublished{(2012). Random matrices: The Universality phenomenon for
  Wigner ensembles. Available at \arxivurl{arXiv:1202.0068}.}
\end{bmisc}
\bptok{imsref}%
\endbibitem

\bibitem{thouless1977solution}
\begin{barticle}[auto:STB|2014/02/12|14:17:21]
\bauthor{\bsnm{Thouless},~\bfnm{D.~J.}\binits{D.~J.}},
\bauthor{\bsnm{Anderson},~\bfnm{P.~W.}\binits{P.~W.}} \AND
\bauthor{\bsnm{Palmer},~\bfnm{R.~G.}\binits{R.~G.}}
(\byear{1977}).
\btitle{Solution of ``solvable model of a spin glass.''}
\bjournal{Philosophical Magazine}
\bvolume{35}
\bpages{593--601}.
\end{barticle}
\bptok{imsref}%
\endbibitem

\bibitem{VershikSporyshev}
\begin{barticle}[mr]
\bauthor{\bsnm{Vershik},~\bfnm{A.~M.}\binits{A.~M.}} \AND
\bauthor{\bsnm{Sporyshev},~\bfnm{P.~V.}\binits{P.~V.}}
(\byear{1992}).
\btitle{Asymptotic behavior of the number of faces of random polyhedra and the neighborliness problem}.
\bjournal{Selecta Math. Soviet.}
\bvolume{11}
\bpages{181--201}.
\bnote{Selected translations}.
\bid{issn={0272-9903}, mr={1166627}}
\end{barticle}
\bptok{imsref}%
\endbibitem

\end{thebibliography}
\end{document}